\documentclass[12pt]{article}
\usepackage[T1]{fontenc}  
\usepackage{enumerate} 
\usepackage{agzt_macros}
 
\usepackage{a4wide}

\newcommand{\bb}[1]{{\color{blue} #1}}
 
\DeclareMathOperator{\dplus}{\circ \kern -0.33em\circ}

\DeclareRobustCommand{\mk}{\genfrac{(}{)}{0pt}{}}

\newcommand{\BARI}{\operatorname{BARI}}
\newcommand{\ARI}{\operatorname{ARI}}
\newcommand{\ekma}{\operatorname{EKMA}}
\newcommand{\biekma}{\operatorname{BEKMA}}
\newcommand{\carma}{\operatorname{CARMA}}
\newcommand{\bicarma}{\operatorname{BCARMA}}

\usepackage[OT2,T1]{fontenc}

\newcommand{\ganit}{\mathit{ganit}}
\newcommand{\ganitpoc}{\mathit{ganit_{poc}}}
\newcommand{\Adari}{\operatorname{Ad}_{ari}}
\newcommand{\adari}{\operatorname{ad}_{ari}}

\newcommand{\arit}{\mathit{arit}}
\newcommand{\axit}{\mathit{axit}}
\newcommand{\ari}{\mathit{ari}}
\newcommand{\uri}{\mathit{uri}}

\newcommand{\pil}{\mathit{pil}}
\newcommand{\lopil}{\mathit{lopil}}

\newcommand{\lu}{\mathit{lu}}

\newcommand{\swap}{\mathit{swap}}
\newcommand{\push}{\mathit{push}}
\newcommand{\anti}{\mathit{anti}}
\renewcommand{\neg}{\mathit{neg}}

\renewcommand{\Pi}{\mathit{Pi}}

\newcommand{\Ad}{\mathit{Ad}}
\newcommand{\asna}{\mathit{asna}}
\newcommand{\ganitpic}{\mathit{ganit}_{pic}} 
\newcommand{\preari}{\mathit{preari}}
\newcommand{\urit}{\mathit{urit}}

\title{Moulds, Bimoulds and some Lie algebras}
\author{Annika Burmester, Ulf K\"uhn, Leila Schneps }
\date{\today}

\begin{document}
\maketitle
\begin{abstract}
The Lie algebra of multiple zeta values is realized as the space $\ARI^{\operatorname{pol}}_{\underline{al}\ast \underline{il}}$ of alternal moulds whose swap is alternil up to a constant mould, equipped with the $ari$ bracket. In this paper, we study the larger space $\BARI_{\underline{il},swap}^{\operatorname{pol}}$ of alternil, swap-invariant bimoulds, which is conjecturally the Lie algebra for multiple $q$-zeta values. We propose an explicit formula for a Lie bracket $uri$ on this space. Moreover, we prove that the Lie algebra $\ARI^{\operatorname{pol}}_{\underline{al}\ast \underline{il}}$ embeds into $\BARI_{\underline{il},swap}^{\operatorname{pol}}$, with the $ari$ bracket translating directly into the $uri$ bracket. Finally, we examine the associated-depth graded of $\BARI_{\underline{il},swap}^{\operatorname{pol}}$, extending the known depth-graded setup for $\ARI_{\underline{al}\ast \underline{il}}^{\operatorname{pol}}$.
\end{abstract}

\tableofcontents
\section{Introduction}
 
In recent years, the theory of moulds as introduced by Ecalle has become increasingly important and has thus lost its surrounding mystique.
Although Ecalle wrote a large part of his work for bimoulds, most of his applications relate to moulds.
Our present work is an attempt to apply the theory of bimoulds to multiple $q$-zeta values and thus bring the full power of Ecalle's ideas to bear.

Of central importance in the theory of multiple zeta values is the Lie algebra $\big( \ARI^{\operatorname{pol}}_{\underline{al}*\underline{il}}, ari\big)$. Its elements are 
alternal polynomial moulds whose swapees are alternil up to a constant mould and in depth 1 only even polynomials are allowed. Furthermore, $ari$ denotes the Lie bracket.  More precisely, 
denote by $\Z$ the algebra of multiple zeta values and consider the generating series of shuffle regularized multiple zeta values
\begin{align*}
\mathfrak{z}(u_1,\ldots,u_d)&=\sum_{k_1,\ldots,k_d\geq1} \zeta_{\shuffle}(k_1,\ldots,k_d)\ (u_1+\cdots+u_d)^{k_1-1}(u_2+\cdots+u_d)^{k_2-1}\cdots u_d^{k_d-1}\\
&\in \Z[[u_1,\ldots,u_d]].
\end{align*}
For simplicity of the presentation in this introduction, we assume that $\Z$ satisfies the extended double shuffle relations and no more relations, cf \cite[Conjecture 1]{IKZ}. Then, one shows that the family of these generating series
considered modulo products and $\zeta(2)$ 
decomposes into products given by a representative  $\bar \zeta$ for a class of $\Z$ modulo products and $\zeta(2)$ and sequences of polynomials $F_{\bar \zeta} \in \ARI_{\underline{al}\ast \underline{il}}^{\operatorname{pol}}$. i.e. we have
\begin{align}\label{eq:fct_eqn_to_alil}
 (\mathfrak{z}(u_1,\ldots,u_d))_{d\geq0}  \equiv \sum_{\bar \zeta}  \bar\zeta  F_{\bar \zeta}.   
\end{align}

Our previous assumption\footnote{In more technical terms this assumption is nothing other then that the algebra of formal multiple zeta values equals $\Z$.}   then implies an algebra isomorphism, which is called Ecalle's theorem in \cite{Racinet_thesis}, 
\begin{align}\label{eq:ecalle_iso}
\Z\simeq \Q[\zeta(2)]\otimes \mathcal{U}(\ARI^{\operatorname{pol}}_{\underline{al}*\underline{il}})^\vee.
\end{align}
Another big conjecture claims that 
this Lie algebra is a free Lie algebra generated by a family of elements $\widehat\xi_{2n+1}$, $n\geq1$, i.e.,
\begin{align}\label{eq:free-generation_conj}
\big( \ARI^{\operatorname{pol}}_{\underline{al}*\underline{il}}, ari\big) \overset{?}{=} 
\Lie\big(  \{\widehat\xi_{2n+1}\}_{n\ge 1}\,; ari \big). 
\end{align}

Here   we use the notation that 
$\Lie( S  ; [\,,\,] )$  denotes the Lie subalgebra of $(S', [\,,\,])$ generated with respect to the Lie bracket $[\,,\,]$
by the elements of  the subset $S\subset S'$. In the cases we consider the ambient set $S'$ will be clear and therefore suppressed.

In this article we focus on the space 
$\BARI^{\operatorname{pol}}_{\underline{il},swap}$ of alternil, swap-invariant, polynomial bimoulds with even depth $1$ component, since the elements of this space occur in the context of multiple $q$-zeta values. 
We also study its depth graded version $\BARI^{\operatorname{pol}}_{\underline{al},swap}$. 
There are various ways to describe the algebra $\Z_q$ of multiple $q$-zeta values, see for example \cite{BaKue_dimconj}. Here, we focus on the combinatorial bi-multiple Eisenstein series $G\binom{k_1,\ldots,k_d}{m_1,\ldots,m_d},\ k_1,\ldots,k_d\geq1, m_1,\ldots,m_d\geq0$, which are introduced in \cite{BaBu_CMES} and which span the space $\Z_q$. We consider the generating series
\begin{align*}
\mathfrak{g}\binom{u_1,\ldots,u_d}{v_1,\ldots,v_d}&=\sum_{\substack{k_1,\ldots,k_d\geq1 \\ m_1,\ldots,m_d\geq0}} G\binom{k_1,\ldots,k_d}{m_1,\ldots,m_d} \frac{u_1^{k_1-1}}{(k_1-1)!}v_1^{m_1}\cdots \frac{u_d^{k_d-1}}{(k_d-1)!}v_d^{m_d} \\
&\in \Z_q[[u_1,v_1,\ldots,u_d,v_d]].
\end{align*}
The combinatorial bi-multiple Eisenstein series satisfy the bi-stuffle product formula and swap invariance. For simplicity of the presentation we assume that there are no further relations. 
Then  the family of these generating series
considered modulo products and $G\binom{2}{0}, G\binom{4}{0}, G\binom{6}{0}$ 
decomposes into products given by a representative  $\bar \zeta_q$ for a class of $\Z_q$ modulo products and $G\binom{2}{0}, G\binom{4}{0}, G\binom{6}{0}$ and sequences of polynomials $F_{\bar \zeta_q} \in
\BARI_{\underline{il},swap}^{\operatorname{pol}}$, i.e.
\[  
\left(\mathfrak{g}\binom{u_1,\ldots,u_d}{v_1,\ldots,v_d}\right)_{d\geq0}
\equiv \sum_{\bar \zeta_q} \bar \zeta_q \,F_{\bar \zeta_q}
\]

Therefore, we want to study the space $\BARI_{\underline{il},swap}^{\operatorname{pol}}$, where the components are  polynomials with rational coefficients. In Definition \ref{def:uri}, we give the definition of the uri Lie bracket 
on $\BARI^{\operatorname{rat}}_{il}$ and explain why we expect that $uri$ restricts to a Lie bracket on various subspaces and in particular to $\BARI_{\underline{il},swap}^{\operatorname{pol}}$. Similar to the Ecalle isomorphism \eqref{eq:ecalle_iso} for multiple zeta values, we expect for multiple $q$-zeta values an algebra isomorphism
\begin{align}\label{eq:ecalle_Zq_conj}
 \Z_q \simeq \widetilde{M}_\Q(\operatorname{Sl}_2(\mathbb{Z})) \otimes \mathcal{U}(\BARI_{\underline{il},swap}^{\operatorname{pol}})^\vee,   
\end{align}
where we used that the algebra $\Q\big[G\binom{2}{0},G\binom{4}{0},G\binom{6}{0}\big]$ is equal to the algebra $\widetilde{M}_\Q(\operatorname{Sl}_2(\mathbb{Z}))$ of quasi-modular forms with rational coefficients. The isomorphism \eqref{eq:ecalle_Zq_conj} will be proven in \cite{BB26} assuming the previously mentioned simplification for $\Z_q$.

One of our main results is to identify the Lie algebra $\big( \ARI^{\operatorname{pol}}_{\underline{al}*\underline{il}}, ari \big)$  with a Lie algebra  inside
$\big( \BARI^{\operatorname{rat}}_{il}, uri \big)$. In addition, that image is   contained in the space of swap invariant, and polynomial bimoulds.

\begin{Theorem}  \label{thm:uri_alil}
Define the map
\begin{align*}  
\iota:\ARI^{\operatorname{pol}}_{\underline{al}*\underline{il}}&\rightarrow \BARI^{\operatorname{pol}}_{\underline{il},swap}\cr
A&\mapsto A+swap(A)+C_A,  
\end{align*} 
where $C_A$ denotes the constant mould such that $swap(A)+C_A$ is alternil. For $A,B\in \ARI_{\underline{al}*\underline{il}}^{\operatorname{pol}}$, we have
\[
\iota\big(ari(A,B)\big)=uri\big(\iota(A),\iota(B)\big).
\]
In particular, the tuple
$\Big( \iota\big(\ARI_{\underline{al}*\underline{il}}^{\operatorname{pol}}\big), uri \Big)$ is a
  Lie algebra.    
\end{Theorem}
In  $\ARI_{\underline{al}*\underline{il}}^{\operatorname{pol}}$  alternal 
moulds are multiplied by the ari bracket, and it is laboriously demonstrated that in this multiplication, the swap  of the product remains  alternil. 
However, our construction of $\iota\big(\ARI_{\underline{al}*\underline{il}}^{\operatorname{pol}}\big)$ equipped with the uri bracket provides an explicit description for a Lie bracket on the alternil parts.

We may summarize our previous discussion in the following diagram
\[
\xymatrix{ 
\Z / (\zeta(2)) \ar@{<-}[rr]_{q \to 1}\ar@{-->}[d] && \Z_q / \widetilde{M}_\Q(\operatorname{Sl}_2(\mathbb{Z})) \ar@{-->}[d]\\
\big( \ARI^{\operatorname{pol}}_{\underline{al}*\underline{il}},\, ari \big)  \ar@{^{(}->}_\iota[rr] &&     
\big( \BARI^{\operatorname{pol}}_{\underline{il},swap} ,\, uri \big),
} 
\]
where the Lie algebra morphism $\iota$ in the lower row should  correspond to $q\to 1$.


It follows from well-established results in mould theory that the $uri$ bracket preserves alternility. Our work is based on the following Conjecture \ref{conj:intro_ilswap}.

\begin{Conjecture} \label{conj:intro_ilswap}
   $\big( \BARI_{\underline{il},swap}^{\operatorname{pol}},uri\big)$ is a Lie subalgebra 
of $\big( \BARI^{\operatorname{rat}}_{il},uri\big)$. 
\end{Conjecture}

We present evidence both proven and experimental towards this conjecture and related conjectures. In particular, in this article we show  
\begin{Theorem} \label{thm:intro_ilpol}
$\big(\BARI_{il}^{\operatorname{pol}},uri\big)$ is a Lie subalgebra of $\big( \BARI^{\operatorname{rat}}_{il},uri\big)$.
\end{Theorem}

From Proposition \ref{prop:swap_upto_3} together with Theorem \ref{thm:intro_ilpol}
then follows that Conjecture \ref{conj:intro_ilswap} holds up to depth $3$.
In Appendix \ref{app:uri_swap_alternative} we give  further evidence for the compatibility of the $uri$ bracket and the $swap$ invariance in terms of a reformulation  in terms of mould theory.

\begin{Conjecture} \label{conj:intro_mystic_derivation}
 The map
$\widehat\delta : \BARI^{\operatorname{pol}} \to \BARI^{\operatorname{pol}} $ given by Definition \ref{def:mystic_der} is a derivation on the Lie algebra $\big(\BARI_{\underline{il},swap}^{pol},uri\big)$ of weight 2. Its associated depth-graded  map equals the ari derivation from Definition \ref{def:ari_derivation}.
\end{Conjecture}

A refinement of Conjecture \ref{conj:intro_ilswap}, which by means of \eqref{eq:ecalle_Zq_conj}   perfectly complements the dimension conjectures for multiple q-zeta values \cite{BaKue_dimconj}, is as follows.

\begin{Conjecture} \label{conj:intro_ilswap_refined}
\begin{enumerate}
\item\label{conj:intro_ilswap_3} Define the elements $\widehat\xi_{2n+1,0} = \iota( \widehat\xi_{2n+1} )$, $n \ge 0$. Then, we have in analogy to \eqref{eq:free-generation_conj}
\[
\big(\BARI_{\underline{il},swap}^{\operatorname{pol}} , uri \big)  =  \Lie\big(   \{\widehat\delta^m(\widehat\xi_{2n+1,0})\}_{m,n\ge 0}; uri \big).
\]
\item For the dimensions of the homogeneous subspaces of the universal enveloping algebra of $(\BARI_{\underline{il},swap}^{\operatorname{pol}},uri)$, we have
\begin{align*}
\sum_{k\geq0} \dim \mathcal{U}(\BARI_{\underline{il},swap}^{\operatorname{pol}})_{k} \,x^k =\frac{1}{1-\mathsf{D}(x)\mathsf{O}_1(x)+\mathsf{D}(x)\sum_{k\geq4}  \dim(M_k\oplus S_k)x^k},
\end{align*}
where $\mathsf{D}(x)=\frac{1}{1-x^2}$, $\mathsf{O}_1(x)=\frac{x}{1-x^2}$, and $M_k$, $S_k$ denote the space of modular forms and cusp forms for $\operatorname{Sl}_2(\mathbb{Z})$.
\end{enumerate}
\end{Conjecture}

The reformulation of the above conjecture into the non-commutative setup (see \cite{Bu_thesis},  \cite{BuKu_postlie}) allows testing for weights up to $14$ with Pari/GP. 

For the associated depth-graded Lie algebra, which is a Lie subalgebra of 
$\big( \BARI_{\underline{al},swap}^{\operatorname{pol}}, ari\big)$,  we have the following.

A distinguished Lie subalgebra is given by the Lie algebra of the ekma bimoulds.
These arise from Ecalle's ekma moulds $\xi_{2n+1}$ after applying the canonical derivation $\delta$ for the ari Lie bracket, cf. Definition \ref{def:ari_derivation}. Precisely, let $\xi_{1,0}=(1,0,\ldots)$ and $\xi_{2n+1,0}=\iota(\xi_{2n+1})=(u_1^{2n}+v_1^{2n},0,\ldots)$, $n\geq1$, and set
\[
\biekma = \Lie \big( \{\delta^m(\xi_{2n+1,0})\}_{m,n \ge 0}; ari \big).
\]
Those ekmas $\xi_{a,b}$ satisfy quadratic relations whose number equals the square of the dimension of spaces modular forms. Actually a conceptual way to describe these relations is via even bi-period polynomials, which correspond to pairs of even period polynomials $f,g\in W_k^+$ and pairs of odd period polynomials $f,g\in W_k^-$. We elaborate this point of view, which was shown to one of us by Don Zagier in 2017, in Appendix \ref{app:bi_period}.
To our knowledge this is the first occurrence of a generating series, which involves the square $(\dim M_k)^2$. 

\begin{Theorem} \label{thm_intro:bekma}
Define the coefficients $e_{k,d}$ by
\begin{align*}
\sum_{k,d\geq0} e_{k,d} x^k y^d  
&=\frac{1}{ 1- \mathsf{D}(x)\mathsf{O}_1(x)y    + \mathsf{D}(x) \big( \sum_{k \ge 4} (\dim M_{k})^2  \, x^k  \big) y^2 -  \mathsf{D}(x) x \mathsf{S}(x) y^3}.
\end{align*}
Then, we have
\begin{align*}
\dim \mathcal{U}(\biekma)_{k,d} &= e_{k,d}   \quad \mbox{ for }  d=1,2  \mbox{ and all } k\ge 0\, .   
\end{align*}
Here, we used the notation from Conjecture \ref{conj:intro_ilswap_refined} and   
$\mathsf{S}(x) = \sum_{k\ge 12} \dim S_k \, x^k$.
\end{Theorem}

We expect no further relations or intersection of relations in $\biekma$ in depth $\geq 3$, which can be reformulated as follows. 

\begin{Conjecture}\label{conj:q_ekma_dim_intro}
The numbers $e_{k,d}$ in Theorem \ref{thm_intro:bekma}  satisfy for all $k,d\geq0$
\[
\dim \mathcal{U}(\biekma)_{k,d}=e_{k,d}.
\]
\end{Conjecture}

We also checked this conjecture for $d$ up to $6$ and $k$ in dependence of $d$ within the range of computability.

The dimension conjecture \ref{conj:intro_ilswap_refined}.2 implies that not all quadratic relations satisfied by the ekmas $\xi_{a,b}$ lift  to relation for their alternil, swap-invariant extension $\widehat\xi_{a,b}=\widehat\delta^b \,(\widehat\xi_{a-b,0})$. Using this observation we construct alternal, swap-invariant bimoulds $\chi_{f,g}$ indexed by pairs of even period polynomials $f,g\in W_k^+$ or odd period polynomials $f,g\in W_k^-$.  This is closely related to  Ecalle's construction of his carma  moulds.


\begin{Conjecture}\label{conj:bari_ilsw} We have
\[
\gr_D \big( \BARI_{\underline{il},swap}^{\operatorname{pol}} , uri \big)
\cong \Lie\big(\{\delta^m(\xi_{2n+1,0})\}_{m,n \ge0} \cup \{\delta^m(\chi_{f,g})\}_{f,g \in W^\pm,m\geq0}; ari \big).
\]
The generating series of the dimensions of the homogeneous subspaces of the universal enveloping algebra of $\gr_D \big( \BARI_{\underline{il},swap}^{\operatorname{pol}} , uri \big)$ is given by
\begin{align*}
\sum_{k,d\geq0} \dim \mathcal{U}&\Big(\gr_D \big( \BARI_{il,swap}^{\operatorname{pol}} , uri \big)\Big)_{k,d} \,x^k y^d \\
&\hspace{2cm}=\frac{1}{1-b_1(x)y + b_2(x) y^2  - b_3(x) y^3 -  b_4(x)y^4 +b_5(x) y^5}.
\end{align*}
Here, we used the notation $\mathsf{D}(x)=\frac{1}{1-x^2}$, $\mathsf{O}_1(x)=\frac{x}{1-x^2}$ and $\mathsf{S}(x) = \sum_{k\ge 12} \dim S_k \, x^k$ as before and set
\begin{align*}
b_1(x)&= \mathsf{D}(x)\, \mathsf{O}_1(x)=\frac{1}{1-x^2}\frac{x}{1-x^2},    \\
b_2(x)&= \mathsf{D}(x)   \sum_{k \ge 4} ( \dim  M_{k}  )^2  \, x^k  =\frac{1}{1-x^2}  \left( \frac{  1+ x^{12}  }{1-x^{12}} 
\frac{  1   }{(1-x^4)(1-x^6)} -1 \right), \\
b_3(x)&=\mathsf{D}(x)\, x\, \mathsf{S}(x)=\frac{x}{1-x^2}\frac{x^{12}}{(1-x^4)(1-x^6)} ,\\
b_4(x) &=  \mathsf{D}(x) \, \sum_{k \ge 12} (\dim  S_{k} )^2  \, x^k = \frac{1}{1-x^2} \frac{  1+ x^{12}  }{1-x^{12}} 
\frac{  x^{12}   }{(1-x^4)(1-x^6)},\\
b_5(x) &=  \mathsf{D}(x) \,x\, \mathsf{S}(x)=\frac{x}{1-x^2}\frac{x^{12}}{(1-x^4)(1-x^6)}. 
\end{align*}
\end{Conjecture}

As a summary, we expect the following picture of the previously introduced (conjectural) Lie algebras
\[
\xymatrix{ 
\text{weight-graded} & \big( \ARI^{\operatorname{pol}}_{\underline{al}*\underline{il}},\, ari \big)  \ar@{^{(}->}_{\iota\quad}[r] \ar@{-->}[d] &  
\big( \BARI^{\operatorname{pol}}_{\underline{il},swap} ,\, uri \big)\ar@{-->}[d] \ar@{-->}[dr]\\
\shortstack{weight-graded and \\ depth-graded} & \big( \ARI^{\operatorname{pol}}_{\underline{al},\underline{al}} ,\, ari \big)    \ar@{^{(}->}_{\gr_D\iota\qquad}[r] &  \big( \gr_D \BARI^{\operatorname{pol}}_{\underline{il},swap},\, ari \big) \ar@{^{(}->}[r]  & \big( \BARI^{\operatorname{pol}}_{\underline{al},swap},\, ari \big).
} 
\]
One expects that the downward arrow on the left stands for a surjective map. In contrast, 
explicit calculations imply that the rightmost inclusion on the bottom line is strict.

 Finally we point to the fact, that in \cite{BuKu_postlie} we study in particular the
 uri Lie bracket in terms of post-Lie structures on free Lie algebras. The appendix of loc cit. contains a dictionary  between these points of view. We plan to explore this further in \cite{BuKue26}.

Many of the conjectures presented here in this paper have been developed in the last decade and several versions of them had been presented at various places before. 
One purpose of this paper is to publish these conjectures  and report on the progress towards establishing them.
Besides new progress we focus here  on unpublished results from the thesis of the first author \cite{Bu_thesis}, as well as results that can also be found in the master thesis of N. Confurius \cite{Con_thesis}.
While finishing this article we were informed that H. Bachmann and H. Kawamura have proved the complete Conjecture \ref{conj:intro_ilswap}, including Conjecture \ref{conj:preuri_with_swap}, independently \cite{BaKa}. 
Moreover, H. Bachmann \cite{Ba_sl2} showed that the derivation $\widehat\delta$ from Definition \ref{def:mystic_der} can be written as
\[
\widehat\delta  = 
ganit_{pic}\circ \delta\circ ganit_{poc} + 
uri(-, \delta_{BvI}(\mathbf{1}_{mu}))
 \]
and proved Conjecture \ref{conj:intro_mystic_derivation}. He studied $\widehat\delta$ as part of an $\mathfrak{sl}_2$-action on $\BARI_{\underline{il},swap}^{\operatorname{pol}}$.

\textbf{Acknowledgments.} We thank Henrik Bachmann, Niclas Confurius, Dominique Manchon, and Don Zagier for helpful discussions on the topics presented in this paper. The first author was supported by the Klaus Tschira Boost Fund.

\section{Moulds and Bimoulds}

Let $R$ be any commutative $\Q$-algebra. A sequence
\begin{align*}
M=(M_d)_{d\geq0}=\left(M_0,M_1\binom{u_1}{v_1},M_2\binom{u_1,u_2}{v_1,v_2},\ldots\right) \in \prod_{d\geq0} R[u_1,v_1,\ldots,u_d,v_d]
\end{align*}
is called a polynomial bimould. 
We will often omit the subscript $d$ if the depth is clear from the number of variables, i.e., we write $M\binom{u_1,\ldots u_d}{v_1,\ldots, v_d} $ instead of $M_d\binom{u_1,\ldots u_d}{v_1,\ldots, v_d}$. In addition we set $M(\emptyset)= M_0$.  We often abbreviate
\begin{align*}
w=(w_1,\ldots,w_d)=\binom{u_1,\ldots,u_d}{v_1,\ldots v_d}.
\end{align*}

\begin{Definition}
We denote the space of all polynomial bimoulds $M=(M_d)_{d\geq0}$ with $M_0=0$ and only finitely nonzero components $M_d$ by $\BARI^{\operatorname{pol}}$. For a bimould $M \in \BARI^{\operatorname{pol}}$ we omit the vanishing first entry $M_0$ and just write 
$M=(M_d)_{d\geq1}$. 
\end{Definition}

At some points it will be necessary to consider sequences
\begin{align*}
M=(M_d)_{d\geq0}=\left(M_0,M_1\binom{u_1}{v_1},M_2\binom{u_1,u_2}{v_1,v_2},\ldots\right) \in \prod_{d\geq0} R(u_1,v_1,\ldots,u_d,v_d).
\end{align*}
The set of those bimoulds we denote by $\BARI^{\operatorname{rat}}$.   Most results hold for both $\BARI^{\operatorname{pol}}$ and $\BARI^{\operatorname{rat}}$; in this case, we simply use the unspecified notation $\BARI$. Statements in which we use the notations $\BARI^ {\operatorname{pol}}$ or $\BARI^{\operatorname{rat}}$ either do not apply to the other space, or it is not clear whether they also apply to the other space.

\begin{Definition} A polynomial mould $M$ is a polynomial bimould that is only a function of the variables $u_i$, i.e. a polynomial mould is a sequence
\begin{align*}
M=(M_d)_{d\geq0}=(M_0,M_1(u_1),M_2(u_1,u_2),\ldots)\in \prod_{d\geq0} R[u_1,\ldots,u_d].
\end{align*}
We denote the set of all polynomial moulds $M=(M_d)_{d\geq0}$ with $M_0=0$ and only finitely many nonzero components by $\ARI^{\operatorname{pol}}$.

Similarly, a polynomial $v$-mould $M$ is a polynomial bimould that is only a function of the variables $v_i$. We denote the set of all polynomial $v$-moulds $M=(M_d)_{d\geq0}$ with $M_0=0$ and only finitely many nonzero components by $\overline{\ARI}^{\operatorname{pol}}$. As for bimoulds, sometimes it is necessary to consider moulds in $\prod_{d\geq0} R(u_1,\ldots,u_d)$ resp. $v$-moulds in $\prod_{d\geq0} R(v_1,\ldots,v_d)$, we denote those spaces by $\ARI^{\operatorname{rat}}$ resp. $\overline{\ARI}^{\operatorname{rat}}$. As before, we simply write $\ARI$ (and similarly $\overline{\ARI}$) when a statement applies to both $\ARI^{\operatorname{pol}}$ and $\ARI^{\operatorname{rat}}$.
\end{Definition}

\begin{Example} 
 \label{ex:ekma} There are the ekma moulds $\xi_{2n+1}\in\ARI^{\operatorname{pol}}$, $n\geq1$, given by
\begin{align*}
\xi_{2n+1}=(u_1^{2n},0,0,\ldots).
\end{align*}
\end{Example}

\begin{Definition} \label{def:weight_depth} A bimould $M\in \BARI$ is homogeneous of \emph{depth} $d$ if $M_d$ is the only non-trivial component of $M$. Furthermore, $M\in\operatorname{BARI}$ is homogeneous of \emph{weight} $k$ if each component $M_d$ is homogeneous of degree $k-d$. We denote the homogeneous space of bimoulds of weight $k$ and depth $d$ by
\[
\BARI_{k,d}.
\]
\end{Definition}
We will use the same notation for subspaces of $\BARI$.

The notion of depth defines a decreasing filtration on the space $\BARI$,
\[
\operatorname{Fil}_d(\BARI) = \{
A = (A_i)_{i \ge 0} \,|\, A_i= 0 \mbox{ for all } i < d\}.
\]
The space $\BARI^{\operatorname{pol}}$ is equal to
its associated depth-graded via the
map
\begin{align}\label{def:gr_D_map}
\gr_D:  \BARI^{\operatorname{pol}}  \to \BARI^{\operatorname{pol}}
\end{align}
given by sending a bimould 
$A= ( 0,\ldots,0, A_d, A_{d+1}, \ldots)\in \operatorname{Fil}_d(\BARI^{\operatorname{pol}})$ (where $A_d\neq 0$) to $(0,\ldots,0,A_d,0,\ldots)$. Observe that $\gr_D$ is not linear, for example we have for $A_2\neq B_2$
\begin{align*}
&\gr_D((A_1,A_2,0,\ldots)-(A_1,B_2,\ldots))=\gr_D((0,A_2-B_2,0,\ldots))=A_2-B_2\neq 0, \\
&\gr_D((A_1,A_2,\ldots))-\gr_D((A_1,B_2,\ldots))=A_1-A_1=0.
\end{align*}
Let $(\BARI^{\operatorname{pol}},[-,-])$ be a filtered Lie algebra, i.e.
\[
\big[\operatorname{Fil}_d(\BARI^{\operatorname{pol}}) , \operatorname{Fil}_{d'}(\BARI^{\operatorname{pol}})\big] 
\subseteq \operatorname{Fil}_{d+d'}(\BARI^{\operatorname{pol}}).
\]
Then its associated depth-graded Lie algebra $\gr_D (\BARI^{\operatorname{pol}},[-,-])$ is the space $\BARI^{\operatorname{pol}}$ equipped with the induced depth-graded Lie bracket
\begin{align} \label{eq:depth-graded_bracket}
\gr_D[A,B]=[A,B] \mod \operatorname{Fil}_{d+d'+1}(\BARI^{\operatorname{pol}}),
\end{align}
were $A\in \operatorname{Fil}_d(\BARI^{\operatorname{pol}}),\ B\in\operatorname{Fil}_{d'}(\BARI^{\operatorname{pol}})$.

\subsection{Basic symmetries and operations}

\begin{Definition} For a bimould $M\in\BARI$, define
\[
swap(M)\binom{ 
u_1,u_2,\ldots,u_d}{v_1,v_2,\ldots,v_d}
= M \binom{v_d,v_{d-1}-v_d,\ldots,v_1-v_2}{u_1+\cdots+u_d,u_1+\cdots+u_{d-1},\ldots,u_1}. 
\]
A bimould $M$ is called $swap$ invariant if $swap(M)=M$.
\end{Definition}      

\begin{Definition} \label{def:push} For a bimould $M\in\BARI$, set
\[
push(M)\binom{ 
u_1,u_2,\ldots,u_d}{v_1,v_2,\ldots,v_d}
= M \binom{-u_1-\cdots -u_d,u_1,\ldots,u_{d-1}}{-v_d,v_1-v_d,\ldots,v_{d-1}-v_d}. 
\]
Similarly, $M$ is called $push$ invariant if $push(M)=M$.
\end{Definition}     

\begin{Definition} Let $\Q\langle Y\rangle$ be the free noncommutative algebra spanned by the alphabet $Y=\{y_{k,m}\mid k\geq1,m\geq0\}$. 
We endow $\Q\langle Y\rangle$ either with the shuffle product $\shuffle$ from \eqref{eq:def_shuffle} or with the stuffle product $*$ from \eqref{eq:def_stuffle}.
 
\begin{enumerate}
    \item A bimould $A\in\BARI$ is \emph{alternal}, if there is a coefficient map $\varphi_\shuffle:\Q\langle Y\rangle \to R$ such that 
    \[
    \varphi_\shuffle(w_1\shuffle w_2)=0 \quad \text{ for all } w_1,w_2\in Y^*\backslash\{1\}
    \]
    and
    \[
    A\binom{u_1,\ldots,u_d}{v_1,\ldots,v_d}=\sum_{\substack{k_1,\ldots,k_d\geq1 \\ m_1,\ldots,m_d\geq0}} \varphi_\shuffle(y_{k_1,m_1}\cdots y_{k_d,m_d}) \frac{u_1^{m_1}}{m_1!}v_1^{k_1-1}\cdots \frac{u_d^{m_d}}{m_d!} v_d^{k_d-1}.
    \]
 We denote the space of alternal bimoulds by $\BARI_{al}$. By restricting to the variables $u_i$, we obtain the definition of alternal moulds and we denote the space accordingly by $\ARI_{al}$.
    \item A bimould $A\in\BARI$ is \emph{alternil}, if there is a coefficient map $\varphi_\ast:\Q\langle Y\rangle \to R$ such that 
    \[
    \varphi_\ast(w_1\ast w_2)=0 \quad \text{ for all } w_1,w_2\in Y^*\backslash\{1\}
    \]
    and
    \[
    A\binom{u_1,\ldots,u_d}{v_1,\ldots,v_d}=\sum_{\substack{k_1,\ldots,k_d\geq1 \\ m_1,\ldots,m_d\geq0}} \varphi_\ast(y_{k_1,m_1}\cdots y_{k_d,m_d}) \frac{u_1^{m_1}}{m_1!}v_1^{k_1-1}\cdots \frac{u_d^{m_d}}{m_d!} v_d^{k_d-1}.
    \]
  The space of alternil bimoulds is denoted by $\BARI_{il}$. By restricting to the variables $v_i$, we get the definition of alternil $v$-moulds and we denote the space accordingly by $\overline{\ARI}_{il}$.
\end{enumerate}
\end{Definition}
For more details and a comparison to alternative descriptions  for these two properties we refer to Appendix \ref{app:bimoulds_quasish}.

\begin{Remark} \label{rem:il_to_al}
The alternility condition in depth $d$ only involves terms of depth $\leq d$.
Therefore the map given by \eqref{def:gr_D_map}  induces a natural map  
\begin{align*}
\gr_D: \BARI_{il}\to \BARI_{al}.
\end{align*}
\end{Remark}

In the later sections, we will consider the following specific spaces which all have two symmetries. Some of these spaces were already studied in \cite{ecalle} and \cite{SchnepsARI}.

\begin{Definition} \label{def:spaces}
We set
\begin{align*}
\ARI_{\underline{al}\ast\underline{il}} &=\left\{ A\in \ARI_{al} \ \middle|\ \begin{matrix} swap(A)+C_A \text{ alternil for some constant mould } C_A,\\
\ A_1 \text{ non-constant and even}\end{matrix}\right\}, \\
\ARI_{\underline{al}/\underline{al}}&=\{ A\in \ARI_{al} \mid swap(A) \text{ alternal},\ A_1 \text{ non-constant and even}\}, \\
\ARI_{\underline{al}\ast\underline{al}}&=\left\{ A\in \ARI_{al}\ \middle|\ \begin{matrix} swap(A)+C_A \text{ alternal for some constant mould } C_A,\\ A_1 \text{ non-constant and even} \end{matrix}\right\},
\end{align*}
and in the same way we define the spaces $\overline{\ARI}_{\underline{al}/\underline{al}}$ and $\overline{\ARI}_{\underline{al}\ast\underline{al}}$ of $v$-moulds. Furthermore, let
\begin{align*}
\BARI_{\underline{al},swap}&=\{A\in \BARI\mid A \text{ alternal, swap invariant},\ A_1 \text{ even}\}, \\
\BARI_{\underline{il},swap}&=\{A\in \BARI\mid A \text{ alternil, swap invariant},\ A_1 \text{ even}\}.
\end{align*}
\end{Definition}

\subsection{The ari bracket}

The space $\BARI$ is an algebra with the componentwise addition and the multiplication given by
\begin{align*}
mu(A,B)(w_1,\ldots,w_d)=\sum_{j=0}^d A(w_1,\ldots,w_j)B(w_{j+1},\ldots,w_d).
\end{align*}
In particular, $\BARI$ is a Lie algebra with the commutator bracket
\begin{align} \label{eq:def_lu}
lu(A,B)=mu(A,B)- mu(B,A).
\end{align}
It is easy to see that the algebra $(\BARI^{\operatorname{pol}},mu)$ is generated by the depth $1$ bimoulds
\begin{align} \label{eq:P_km}
P_{k,m}=(u^kv^m,0,\ldots),\qquad k, m\geq0.
\end{align}
Hence, it suffices to define algebra morphisms on $(\BARI^{\operatorname{pol}},mu)$ on the generators $P_{k,m}$, or more generally, on depth $1$ bimoulds and of course the same applies to derivations.

Decompose $ w=\binom{u_1,\ldots,u_d}{v_1,\ldots,v_d}$
into $w=abc$ with 
\[  
a= \binom{u_1,\ldots,u_r}{v_1,\ldots,v_r}, \quad
b=\binom{u_{r+1},\ldots,u_{r+s}}{v_{r+1},\ldots,v_{r+s}}, \quad
c=\binom{u_{r+s+1},\ldots,u_d}{v_{r+s+1},\ldots,v_d}.
\] 
Their flexions are defined by
$ \lceil c =c $  and 
$ a\rceil  =a $ if $b=\emptyset$, 
$ b\rfloor  =b $ if $c=\emptyset$, 
$ \lfloor b =b $ if $a=\emptyset$ 
and else by
\begin{align*}
&b\rfloor = \binom{u_{r+1},\ldots,u_{r+s}}{v_{r+1}-v_{r+s+1},\ldots,v_{r+s}-v_{r+s+1}}, \quad
\lceil c=\binom{u_{r+1}+\cdots+u_{r+s+1},u_{r+s+2},\ldots,u_d}{v_{r+s+1},v_{r+s+2},\ldots,v_d}, \\
&a\rceil=\binom{u_1,\ldots,u_{r-1},u_r+u_{r+1}+\cdots+u_{r+s}}{v_1,\ldots,v_{r-1},v_r}, \quad
\lfloor b=\binom{u_{r+1},\ldots,u_{r+s}}{v_{r+1}-v_r,\ldots,v_{r+s}-v_r}.
\end{align*}
\begin{Definition} \label{def:arit} For a bimould $B\in \BARI$ the operator $arit(B)$ is given by
\begin{align*}
arit(B)(A)(w)=\sum_{\substack{w=abc \\ b,c\neq \emptyset}} A(a\lceil c) B(b\rfloor)-\sum_{\substack{w=abc \\ a,b\neq \emptyset}} A(a\rceil c) B(\lfloor b).
\end{align*}
\end{Definition}
 
\begin{Remark} \label{rem:arit_depth1} In fact, $arit(B)$ is a Lie derivation with respect to  $lu$ and its value on a generator $P_{i,j}$ of $(\BARI^{\operatorname{pol}},mu)$ is given by

  \begin{align*}
&arit(B)(P_{i,j})\binom{u_1,\ldots,u_n}{v_1,\ldots,v_n}\\
&=(u_1+\cdots+u_n)^i\Bigg[v_n^j B\binom{u_1,\ldots,u_{n-1}}{v_1-v_n,\ldots,v_{n-1}-v_n}-v_1^j B\binom{u_2,\ldots,u_n}{v_2-v_1,\ldots,v_n-v_1}\Bigg]\\
\end{align*}
\end{Remark}

\begin{Definition} 
The ari bracket of two bimoulds $A,B\in\BARI$ is defined as
\[
ari(A,B)=arit(B)(A)-arit(A)(B)+lu(A,B).
\]
\end{Definition}      
Explicitly, we have 
\begin{align} \label{eq:ari_expl}
ari(A,B)  (w) &= \underset{\substack{ w=abc \\ b \neq \emptyset}}{\sum} \Big(A( a \lceil c) B( b \rfloor) - 
B( a \lceil c) A( b \rfloor)\Big)
\\
&\hspace{5cm}-\underset{\substack{ w=abc \\ a,b \neq \emptyset}}{\sum} \Big(A( a \rceil c) B(  \lfloor b) - 
B( a \rceil c) A(  \lfloor b)\Big). \nonumber
\end{align}

\begin{Definition} \label{def:preari}
For two bimoulds $A,B\in\BARI$, denote
\begin{align*}
preari(A,B)=arit(B)(A)+mu(A,B).
\end{align*}
\end{Definition}
With this operator, we can rewrite the ari bracket as 
\[ari(A,B)=preari(A,B)-preari(B,A).\]

\begin{Example}
Let $A,B\in \BARI$. Then  
$ari(A,B)$ is in depth $2$ given by
\begin{align*}
ari(A,B)\binom{u_1,u_2}{v_1,v_2}&= 
lu(A,B)\binom{u_1+u_2,u_1}{v_2,v_1-v_2} +
lu(A,B)\binom{u_2,u_1+u_2}{v_2-v_1,v_1} \\
&\hspace{0.4cm} +lu(A,B)\binom{u_1,u_2}{v_1,v_2}
\end{align*}
\end{Example}

\begin{Definition} \label{def:ad_ari}
For $A\in \BARI$, define the operator $\operatorname{ad}_{ari}$ on $\BARI$ by
\begin{align*}
\operatorname{ad}_{ari}(A)(B)=ari(A,B).
\end{align*}
\end{Definition} 

The pair $(\BARI,ari)$ is a Lie algebra. By restricting to the variables $u_i$, one obtains the corresponding operations on $\ARI$. In particular, the pair $(\operatornamewithlimits{ARI},ari)$ is also a Lie algebra.

\begin{Proposition} \label{prop:al_ari}
The pairs $(\BARI_{al},ari)$ and $(\ARI_{al},ari)$ are Lie algebras.
\end{Proposition}

\begin{proof}
See \cite[Appendix A]{SalernoSchneps}.
\end{proof}

\begin{Remark} \label{rem:ari_postLie}
We also have the finer statement that $(\BARI_{al}^{\operatorname{pol}},lu,-arit)$ is a post-Lie algebra. A proof in terms of non-commutative polynomials is given in \cite{BuKu_postlie} and the translation into bimoulds in \cite{Bu_depthgraded}. For every post-Lie algebra $(\mathfrak{g},[-,-],\triangleright)$, another Lie bracket on $\mathfrak{g}$ is given by
\[
\{f,g\}= f \triangleright g - g \triangleright f +[f,g].
\]
The new Lie bracket for $(\BARI_{al}^{\operatorname{pol}},lu,-arit)$ is exactly the ari bracket, thus we obtain Proposition \ref{prop:al_ari} also with this approach.
\end{Remark}

\begin{Proposition} \label{prop:alal_push}
A mould $A$ which is either an element of $\ARI_{\underline{al}/\underline{al}}$, 
$\ARI_{\underline{al}\ast\underline{al}}$ or $ \overline{\ARI}_{\underline{al}\ast\underline{al}}$ is $push$ invariant.
\end{Proposition}
\begin{proof}
  See \cite[Lemma 2.5.5]{SchnepsARI}.   
\end{proof} 

\begin{Proposition} \label{prop:pushinv_swap} For push invariant moulds $A,B\in \ARI$, we have
\[
swap\big(ari(A,B)\big)=ari\big(swap(A),swap(B)\big).
\]
\end{Proposition}
\begin{proof}
See \cite[Lemma 2.4.1]{SchnepsARI}
\end{proof}

\begin{Lemma}\label{lem:al_swap_parity}
For a mould $A\in \ARI_{\underline{al}/\underline{al}}$, we have for all $d\geq1$
\[
A(u_1,\ldots,u_d)=A(-u_1,\ldots,-u_d).
\]
\end{Lemma}

\begin{proof}
See \cite[Lemma 2.5.5]{SchnepsARI}.   
\end{proof} 

\begin{Theorem}
The pair $(\ARI_{\underline{al}/\underline{al}},ari)$ is a Lie algebra.
\end{Theorem} 

\begin{proof}
See \cite[Theorem 2.5.6]{SchnepsARI}. 
\end{proof}

\begin{Lemma} \label{lem:orth_rel} Let $A\in\ARI$ and denote $\xi_1=(1,0,0,\ldots) \in \BARI$. Then, we have
\[
ari(A,\xi_{1})=0.
\]
\end{Lemma}
\begin{proof} As $\xi_1$ is homogeneous of depth $1$, we get with the explicit formula for the ari bracket from \eqref{eq:ari_expl} 
\begin{align*}
ari(A,\xi_{1})(u_1,\ldots,u_d)&= \sum_{i=1}^d A(u_1,\ldots,u_{i-1},u_i+u_{i+1},u_{i+2},\ldots,u_d)\xi_1(u_i) \\
&- \xi_1(u_1) A(u_2,\ldots,u_d)- \xi_1(u_1+\cdots+u_d) A(u_1,\ldots,u_{d-1})
\\
&-\sum_{i=2}^d A(u_1,\ldots,u_{i-2},u_{i-1}+u_i,u_{i+1},\ldots,u_d) \xi_1( u_i)\\ 
&+\xi_1(u_1+\cdots+u_d) A(u_2,\ldots,u_d). 
\end{align*}
As we have $\xi_1(u_1)=1$, all terms cancel and we get $ari(A,\xi_1)=0$.
\end{proof}

\begin{Lemma}\label{lem:ari_vs_push}
Let $A \in \operatorname{ARI}$ and $M \in \overline{ARI}$ be push invariant. Then 
\[ari(A,M) = 0.\]  
\end{Lemma} 
\begin{proof}
For $A\in\ARI$ and $M\in \overline{ARI}$, we can remove the lower flexions from $A$ and the upper flexions from $M$ in the formula \eqref{eq:ari_expl} for the ari bracket and obtain
\begin{align}\label{eq:leila1}
ari(A,M)(w)&=\sum_{\substack{w=abc \\ b\ne \emptyset}} 
\Big(A(a\lceil c)M(b\rfloor)-M(ac)A(b)\Big)\\
&\hspace{3cm}
-\sum_{\substack{w=abc \\ a,b\ne \emptyset}}
\Big(A(a\rceil c)M(\lfloor b)-M(ac)A(b)\Big). \nonumber 
\end{align}
Consider the words $\omega$ (including flexions) that can appear in $A$ in the above
sum.  There are several possibilities, which we divide up into cases 
according to whether the word $\omega$ contains flexions or not, and
whether $w_1$ appears as the first letter of $\omega$, or at the beginning of
$\omega$ but within a flexion, or not at the beginning of $\omega$ at all.

{\it $\omega$ contains no flexions and 
$w_1$ appears as the first letter of $\omega$.}  Then $\omega=(w_1,\ldots,w_i)$
for some $i\in \{1,\ldots,r-1\}$.
The word $\omega=(w_1,\ldots,w_i)$ occurs in $A$ only in the first 
sum in \eqref{eq:leila1}, either as $A(a)$ where $c=\emptyset$, with coefficient 
$M(b\rfloor)=M(b)=M(w_{i+1},\ldots,w_r)$, or as $A(b)$ where 
$a=\emptyset$, with coefficient $-M(c)=-M(w_{i+1},\ldots,w_r)$. Thus
for each $i$, these two terms sum to zero.

{\it $\omega$ contains no flexions and $w_1$ does not appear in $\omega$.}
Let $\omega=(w_i,\ldots,w_j)$ with $1<i<j$.  This type of term appears in the 
first sum of \eqref{eq:leila1} as $A(b)$ with coefficient $-M(ac)$, and also in the second
term as $A(b)$, with coefficient $M(ac)$.  Thus again these two terms sum
to zero.

{\it $\omega$ contains a flexion.} In this case the letter $w_1$
necessarily appears in $A$, either inside or not inside the flexion.
These terms occur in \eqref{eq:leila1} as $A(a\lceil c)$ with $c\ne \emptyset$ and
coefficient $M(b\rfloor)$, and as $A(a\rceil c)$ in the second sum
with coefficient $M(\lfloor b)$. 
More precisely, each decomposition 
\[abc=(w_1,\ldots,w_{i-1})(w_i,w_{i+1},\ldots,w_{j-1})(w_j,\ldots,w_r)\]
in the first sum, for $1<i<j$, gives the term 
\begin{align} \label{eq:decomp_A(a lceil c)_M(rfloor b)} 
&A(a\lceil c)M(b \rfloor)=A(u_1,\ldots,u_{i-1},u_i+
\cdots+u_j,u_{j+1},\ldots,u_r)\\
&\hspace{7cm} M(v_i-v_j,v_{i+1}-v_j,\ldots,v_{j-1}-v_j), \nonumber
\end{align}
and pairs up with the decomposition
\[abc=(w_1,\ldots,w_i)(w_{i+1},\ldots,w_j)(w_{j+1},\ldots,w_r)\] in the
second sum which gives the term 
\begin{align} \label{eq:decomp_A(a rceil c)_M(lfloor b)} 
&-A(a\rceil c)M(\lfloor b)=-A(u_1,\ldots,u_i+
\cdots u_j,u_{j+1},\ldots,u_r)\\
&\hspace{7cm}M(v_{i+1}-v_i,v_{i+2}-v_i,\ldots,v_j-v_i). \nonumber
\end{align}
Indeed we claim that the terms \eqref{eq:decomp_A(a lceil c)_M(rfloor b)} and \eqref{eq:decomp_A(a rceil c)_M(lfloor b)} sum to zero, i.e.~that
\begin{align}\label{eq:leila2}
M(v_i-v_j,v_{i+1}-v_j\ldots,v_{j-1}-v_j)=M(v_{i+1}-v_i,v_{i+2}-v_i,\ldots,v_j-v_i).
\end{align} 
By definition of the push operator given in Definition \ref{def:push} we have
\[
push(M)(v_{i+1}-v_i,v_{i+2}-v_i,\ldots,v_j-v_i)=M(v_i-v_j,v_{i+1}-v_j,\ldots,v_{j-1}-v_j)
\]
Since $M$ is push invariant by assumption, we get the equality in \eqref{eq:leila2}, and
all the paired terms cancel out.
\end{proof}

Central for us is the following famous theorem of Ecalle.
\begin{Theorem}\label{thm:alil_lie}
The pair $(\ARI_{\underline{al}\ast\underline{il}},ari)$ is a Lie algebra.
\end{Theorem}

\begin{proof}
See \cite[Theorem 4.6.1]{SchnepsARI}. 
\end{proof}

\begin{Definition}\label{def:ari_derivation}
Let $\delta:\BARI\to \BARI$ be the map given by
\begin{align*}
\delta(A)\binom{u_1,\ldots,u_d}{v_1,\ldots,v_d}=(u_1v_1+\cdots+u_dv_d)A\binom{u_1,\ldots,u_d}{v_1,\ldots,v_d}.
\end{align*}
\end{Definition}

Based on the next theorem, we refer to $\delta$ as the ari derivation.

\begin{Theorem} 
The map $\delta$ is a derivation on the Lie algebra 
$\big( \BARI, ari \big)$, i.e.,
\[ 
\delta(ari(A,B))=ari(\delta(A),B)+ari(A,\delta(B)).
\]
Furthermore, the ari derivation $\delta$ preserves the space $\BARI_{\underline{al},swap}$.
\end{Theorem}
\begin{proof} The first claim is the consequence of a straightforward calculation. 
The term $u_1v_1+\cdots+u_dv_d$ is invariant under every reordering of the variables $\binom{u_1}{v_1},\ldots,\binom{u_d}{v_d}$. Hence the derivation $\delta$ preserves alternal bimoulds. Moreover, one computes
\begin{align*}
&v_d(u_1+\cdots u_d)+(v_{d-1}-v_d)(u_1+\cdots+u_{d-1})+\cdots+(v_1-v_2)u_1\\
&= v_du_d+v_{d-1}u_{d-1}+\cdots+v_1u_1.
\end{align*}
So, $u_1v_1+\cdots +u_dv_d$ is also invariant under the variable substitution from the operator $swap$, and hence $\delta$ also preserves the space of swap invariant bimoulds.
\end{proof}

\subsection{The gari product}

Let $\operatorname{GBARI}$ be the set of all bimoulds $M=(M_d)_{d\geq0}$ with $M_0=1$. If we have $M_d\in R[u_1,v_1,\ldots,u_d,v_d]$ for all $d\geq1$, we write $M\in \operatorname{GBARI}^{\operatorname{pol}}$ and if $M_d\in R(u_1,v_1,\ldots,u_d,v_d)$ for all $d\geq1$ we write $M\in \operatorname{GBARI}^{\operatorname{rat}}$. As before, $\operatorname{GARI}$ is the set of moulds $M=(M_d)_{d\geq0}$ with $M_0=1$ and $\overline{GARI}$ the set of $v$-moulds with the same condition. 

\begin{Example}\label{ex:pic_poc} 
The bimoulds $pic,\ poc\in \overline{GBARI}^{\operatorname{rat}}$ are defined by 
\begin{align*}
pic\binom{u_1,\ldots,u_r}{v_1,\ldots,v_r}=\frac{1}{v_1\cdots v_r},\qquad poc\binom{u_1,\ldots,u_r}{v_1,\ldots,v_r}=-\frac{1}{v_1(v_1-v_2)\cdots (v_{r-1}-v_r)}.
\end{align*} 
\end{Example}

The pair $(\operatorname{GBARI},mu)$ forms a group. The neutral element $\mathbf{1}_{mu}$ equals $(1,0 , \ldots) \in  \operatorname{GBARI}$.

\begin{Definition}
For a bimould $B\in\operatorname{GBARI}$, define the automorphism $garit_B$ on $\operatorname{GBARI}$ by
\begin{align*}
garit_B(A)&=\sum_{\substack{w=a_1b_1c_1\cdots a_sb_sc_s \\ b_i\neq \emptyset, a_ic_{i+1}\neq\emptyset}} A(\lceil b_1\rceil\cdots \lceil b_s\rceil)B(a_1\rfloor)\cdots B(a_s\rfloor)\\
&\hspace{6cm}\cdot invmu(B)(\lfloor c_1)\cdots invmu(B)(\lfloor c_s).
\end{align*}
Here, $invmu(B)$ denotes the inverse of $B$ with respect to the product $mu$.
The gari product of two bimoulds $A,B\in\operatorname{GBARI}$ is defined as
\begin{align*}
gari(A,B)=mu\big(garit_B(A),B\big).
\end{align*}
\end{Definition}

The pair $(\operatorname{GBARI},gari)$ is a group, which corresponds to the Lie algebra $(\operatorname{BARI},ari)$ via an exponential map.

\begin{Definition}
For a bimould $A\in\BARI$, define
\begin{align*}
\exp_{ari}(A)=\sum_{k\geq0} \frac{1}{k!} \underbrace{preari(A,preari(A,\ldots, preari(A,A)\cdots))}_{k \text{ times}},
\end{align*}
where $preari$ was given in Definition \ref{def:preari}.
\end{Definition}

The map 
\[
\exp_{ari}:\BARI\to\operatorname{GBARI}
\]
is a bijection of sets, and we denote the inverse map by
\begin{align*}
\log_{ari}:\operatorname{GBARI}\to \BARI.
\end{align*}
Note that for $\BARI^{\operatorname{pol}}$, the map $\exp_{ari}$ is only a bijection if we pass to the completion, which consists of bimoulds whose components are power series.

\begin{Definition}
For $A\in\operatorname{GBARI}$, define the operator $\operatorname{Ad}_{ari}$ on $\operatorname{BARI}$ by
\begin{align*}
\operatorname{Ad}_{ari}(A)(B)=gari\big(preari(A,B),invgari(A)\big),
\end{align*}
where $invgari(A)$ denotes the inverse of $A$ with respect to the $gari$ product.
\end{Definition}

\begin{Lemma} \label{lem:Ad_ari}
The group $\operatorname{GBARI}$ acts on the Lie algebra $\operatorname{BARI}$ via the adjoint operator $\operatorname{Ad}_{ari}$, i.e., we have for $G\in\operatorname{GBARI}$ and $A,B\in \BARI$ that
\[
\operatorname{Ad}_{ari}(G)\big(ari(A,B)\big)=ari\big(\operatorname{Ad}_{ari}(G)(A),\operatorname{Ad}_{ari}(G)(B)\big).
\]
\end{Lemma}
\begin{proof}
See \cite[p. 32]{SchnepsARI}
\end{proof}

As the Lie algebra $(\BARI,ari)$ corresponds to the group $(\operatorname{GBARI},gari)$ under the exponential $\exp_{ari}$, we have
\begin{align} \label{eq:Ad_ad_ari_exp}
\operatorname{Ad}_{ari}(\exp_{ari}(A))=\exp(\operatorname{ad}_{ari}(A))
\end{align}

or, equivalently
\begin{align} \label{eq:Ad_ad_ari_log}
\log(\operatorname{Ad}_{ari}(A))=\operatorname{ad}_{ari}(\log_{ari}(A)).
\end{align}
Here, $\exp$ and $\log$ denote the usual exponential and logarithm map on operators and $\operatorname{ad}_{ari}$ was given in Definition \ref{def:ad_ari}.

There are two moulds, which will be important in the following.

\begin{Definition}
\begin{enumerate}
\item Define the $v$-mould $lopil\in \overline{\ARI}^{\operatorname{rat}}$ by
\[ lopil(v_1,\ldots,v_d)=c_d\frac{v_1+\cdots+v_d}{v_1(v_1-v_2)\cdots(v_{d-1}-v_d)v_d},
\]
where the coefficients $c_d$ are obtained from
\begin{align} \label{eq:def_coeffs_lopil}
\Big(\exp\Big(\sum_{d\geq1} c_d x^{d+1} \frac{d}{dx}\Big)\Big)x=1-e^{-x}.
\end{align}
Then, the mould $pil \in \overline{\operatorname{GARI}}^{\operatorname{rat}}$ is given by
\begin{align} \label{eq:pil_lopil}
pil=\exp_{ari}(lopil).
\end{align}
\item Define the mould $pal \in \operatorname{GARI}^{\operatorname{rat}}$ by 
\[
pal=swap(pil).
\]
\end{enumerate}

\begin{Proposition} \label{prop:Ad_ari_alal_alil} The operator $\operatorname{Ad}_{ari}(pal)$ restricts to an isomorphism of Lie algebras
\[
\operatorname{Ad}_{ari}(pal):\ARI_{\underline{al}\ast\underline{al}}^{\operatorname{rat}}\to \ARI_{\underline{al}\ast\underline{il}}^{\operatorname{rat}}.
\]
\end{Proposition}

\begin{proof}
See \cite[Theorem 4.6.1]{SchnepsARI}
\end{proof}
\end{Definition}

\begin{Definition} \label{def:ganit} For a bimould $B\in \operatorname{GBARI}$, define the operator  $ganit_B$ on bimoulds by
\begin{align*}
ganit_B(A)(w)=\sum_{\substack{w=b_1c_1\cdots b_sc_s \\ \text{only } c_s=\emptyset \text{ is allowed}}} A(b_1\rceil \cdots b_s\rceil) B(\lfloor c_1) \cdots B(\lfloor c_s).
\end{align*}
\end{Definition}

\begin{Lemma}\label{lem:ganit_mu_auto}
For $B\in \operatorname{GBARI}$, the operator $ganit_B$ is an automorphism for the $mu$ multiplication on bimoulds.
\end{Lemma}
\begin{proof}
See \cite[Proposition 3.5]{Komiyama}.
\end{proof}

\begin{Example} For the bimoulds $pic,\ poc$ defined 
in Example \ref{ex:pic_poc}  
we have 
\begin{align*}
ganit_{pic}(A)\binom{u_1,u_2}{v_1,v_2}&=A\binom{u_1,u_2}{v_1,v_2}+\frac{1}{v_2-v_1}A\binom{u_1+u_2}{v_1}, \\
ganit_{poc}(A)\binom{u_1,u_2}{v_1,v_2}&=A\binom{u_1,u_2}{v_1,v_2}-\frac{1}{v_2-v_1}A\binom{u_1+u_2}{v_1}.
\end{align*}
\end{Example}

\begin{Proposition} \label{prop:ganit_pic_poc} The map $ganit_{pic}$ restricts to a vector space isomorphism 
\begin{align*}
ganit_{pic}: \BARI_{al}^{\operatorname{rat}}\to \BARI_{il}^{\operatorname{rat}}.
\end{align*}
The inverse map is $ganit_{poc}$.
\end{Proposition}
 
\begin{proof}
See \cite[Theorem 3.24]{Komiyama}.
\end{proof}

The following result is known as Ecalle’s second fundamental identity.
\begin{Theorem} Let $M\in \operatorname{BARI}$ be push invariant, then
\begin{align}\label{eq:ecalle_fundamental} swap (\operatorname{Ad}_{ari}(pal)(M)) = ganit_{pic} \big(\operatorname{Ad}_{ari}(pil)(swap(M))\big).
\end{align}   
\end{Theorem}
\begin{proof}
See \cite[Theorem 4.5.2]{SchnepsARI}
\end{proof}
Setting $P =\operatorname{Ad}_{ari}(pal)(M)$ in 
\eqref{eq:ecalle_fundamental}, we obtain the 
following version of  \eqref{eq:ecalle_fundamental} which is valid for all bimoulds $P$ that are images of push invariant bimoulds $M$
under $\operatorname{Ad}_{ari}(pal)$:
\begin{align}\label{eq:ecalle_fundamental_applied}
ganit_{poc}(swap(P)) = \operatorname{Ad}_{ari}(pil)\big( swap(\operatorname{Ad}_{ari}(pal)^{-1} (P))\big).   
\end{align}

\subsection{The uri bracket and the mystic derivation}\label{subsec:uri}

\begin{Definition} \label{def:uri} The uri bracket on $\BARI^{\operatorname{rat}}$ is defined by
\[
uri(A,B)=ganit_{pic}\big(ari\big(ganit_{poc}(A),ganit_{poc}(B)\big)\big).
\]
\end{Definition}
An obvious first result is the following.
\begin{Proposition} \label{prop:uri_alternil}
The pair $(\BARI_{il}^{\operatorname{rat}},uri)$ is a Lie algebra.
\end{Proposition}
\begin{proof}
The ari bracket preserves alternal bimoulds by Proposition \ref{prop:al_ari}.
By Proposition \ref{prop:ganit_pic_poc} the map $ganit_{poc}$ sends alternil to alternal bimoulds and $ganit_{pic}$ is the inverse map. Thus the operator $uri$ is well-defined on alternil bimoulds and by construction the axioms for being a Lie bracket hold.
\end{proof}

\begin{Lemma} \label{lem:grD_ilswap_alswap} We have
\[ \gr_D\big(\BARI_{il}^{\operatorname{rat}},uri\big) \subset \big( \BARI_{al}^{\operatorname{rat}},ari\big).
 \]
\end{Lemma}
The notion of the associated depth-graded Lie algebra was introduced in \eqref{eq:depth-graded_bracket}.

\begin{proof}
By Remark \ref{rem:il_to_al}, we have
\[
\gr_D\BARI_{il}^{\operatorname{rat}} \subset \BARI_{al}^{\operatorname{rat}}.
\]
Moreover, the operators $\gr_D ganit_{pic}$ and $\gr_D ganit_{poc}$ are both just the identity maps, hence we deduce from Definition \ref{def:uri} that $\gr_D uri= ari$.
\end{proof}
  
The uri bracket can be also build out of a derivation and the $lu$ bracket.
Recall that the algebra $(\BARI^{\operatorname{pol}},mu)$ is generated by the depth $1$ bimoulds $P_{k.m}$ from \eqref{eq:P_km}. We use
the description of the derivation
$arit$ given in Remark \ref{rem:arit_depth1}
to make the following generalisation.
\begin{Definition} \label{def:urit} Let $B\in\BARI^{\operatorname{pol}}$ and $d,r\geq1$. We define the derivation $urit(B)^{(d,r)}$ with respect to $mu$ on depth $1$ bimoulds $A$ as
\begin{align*}
&urit(B)^{(d,r)}(A)\binom{u_1,\ldots,u_{d+r+1}}{v_1,\ldots,v_{d+r+1}}=
A\binom{u_1+\dots+u_{d+r+1}}{v_1} \\
&\hspace{1.4cm}\cdot \Bigg[\sum_{s\in\{1,\ldots,r,d+r+1\}}\prod_{\substack{k\in\{1,\ldots,r,d+r+1\} \\ k\neq s}}\frac{1}{(v_k-v_s)}B\binom{u_{r+1},\ldots,u_{d+r}}{v_{r+1}-v_s,\ldots,v_{d+r}-v_s} \\	
&\hspace{6cm}-\sum_{s=1}^{r+1} \prod_{\substack{k=1 \\ k\neq s}}^{r+1}\frac{1}{(v_k-v_s)}B\binom{u_{r+2},\ldots,u_{d+r+1}}{v_{r+2}-v_s,\ldots,v_{d+r+1}-v_s}\Bigg].
\end{align*}
Then, we set
\[urit(B)(A)\binom{u_1,\ldots,u_n}{v_1,\ldots,v_n}=arit(B)(A)\binom{u_1,\ldots,u_n}{v_1,\ldots,v_n}+\sum_{\substack{d+r+1=n \\ d,r\geq 1}}urit(B)^{(d,r)}(A)\binom{u_1,\ldots,u_n}{v_1,\ldots,v_n}.\] 
\end{Definition}
\begin{Remark} For two bimoulds $A,B\in\BARI$ with arbitrary depths, we have 
\begin{align*}
&urit(B)(A)\binom{u_1,\ldots,u_n}{v_1,\ldots,v_n} =arit(B)(A)\binom{u_1,\ldots,u_n}{v_1,\ldots,v_n}\\
&+\sum_{\substack{d+r+l=n \\ d,r,l\geq1}}\sum_{i=1}^{l} A\binom{u_1,\ldots,u_{i-1},u_i+\dots+u_{d+r+i},u_{d+r+i+1},\ldots,u_{d+r+l}}{v_1,\ldots,v_{i-1},v_i,v_{d+r+i+1},\ldots,v_{d+r+l}} \\
&\hspace{1,3cm}\cdot \Bigg[\sum_{\substack{0\leq s \leq r-1 \\ \text{or} \\ s=d+r }} \prod_{\substack{0\leq k\leq r-1, \ k\neq s \\ \text{or} \\ k=d+r,\ k\neq s}}\frac{1}{(v_{i+k}-v_{i+s})}B\binom{u_{i+r},\ldots,u_{i+d+r-1}}{v_{i+r}-v_{i+s},\ldots,v_{i+d+r-1}-v_{i+s}}\\	
&\hspace{3,7cm}-\sum_{s=0}^{r} \prod_{\substack{k=0 \\ k\neq s}}^{r}\frac{1}{(v_{i+k}-v_{i+s})}B\binom{u_{i+1+r},\ldots,u_{i+d+r}}{v_{i+1+r}-v_{i+s},\ldots,v_{i+d+r}-v_{i+s}}\Bigg].
\end{align*}
\end{Remark}
\begin{Proposition} \label{prop:urit_polynomial} 
For any $B\in\BARI^{\operatorname{pol}}$ the derivation $urit(B)$ preserves the space $\BARI^{\operatorname{pol}}$.
\end{Proposition}
\begin{proof} By \cite[Exercise 7.4.]{stanley}, for some commutative variables $a_1,\ldots,a_r$ the following holds
\begin{align} \label{eq:stanley}
\sum_{\substack{j_1+\dots+j_r=m+1\\ j_1,\ldots,j_r\geq1}} a_1^{j_1-1}\dots a_r^{j_r-1}=\sum_{s=1}^r \prod_{\substack{k=1 \\ k\neq s}}^r \frac{1}{(a_s-a_k)} a_s^m.
\end{align}
Due to the linearity of $urit$ assume that 
\[B_d\binom{u_1,\ldots,u_d}{v_1,\ldots,v_d}=u_1^{m_1}\cdots u_d^{m_d}v_1^{k_1-1}\dots v_d^{k_d-1}\]
for some $k_1,\ldots,k_d\geq 1,\ m_1,\ldots,m_d\geq 0$. Then, we obtain 
\begin{align*}
&\sum_{s=1}^{r+1} \prod_{\substack{t=1 \\ t\neq s}}^{r+1}\frac{1}{v_t-v_s}B_d\binom{u_{r+2},\ldots,u_{d+r+1}}{v_{r+2}-v_s,\ldots,v_{d+r+1}-v_s} \\
&= u_{r+2}^{m_1}\cdots u_{d+r+1}^{m_d} \sum_{l_1+\cdots+l_d+l=k_1+\cdots+k_d} (-1)^l \binom{k_1-1}{l_1-1}\cdots \binom{k_d-1}{l_d-1} v_{r+2}^{l_1-1}\cdots v_{d+r+1}^{l_d-1} \\
&\hspace{2cm}\cdot\sum_{s=1}^{r+1}\prod_{\substack{t=1 \\ t\neq s}}^{r+1}\frac{1}{v_t-v_s} v_s^l \\
&=u_{r+2}^{m_1}\cdots u_{d+r+1}^{m_d} \sum_{l_1+\cdots+l_d+l=k_1+\cdots+k_d} (-1)^l \binom{k_1-1}{l_1-1}\cdots \binom{k_d-1}{l_d-1}v_{r+2}^{l_1-1}\cdots v_{d+r+1}^{l_d-1} \\
&\hspace{2cm} \cdot\sum_{\substack{j_1+\cdots+j_{r+1}=l+1 \\ j_1,\ldots,j_{r+1}\geq1}} v_1^{j_1-1}\cdots v_{r+1}^{j_{r+1}-1}.
\end{align*}
We showed that the poles in the last sum in the Definition \ref{def:urit} cancel out. A completely analogous calculation shows that the same is true for the first sum. Therefore, $urit(B)(A)$ is a polynomial bimould whenever $A,B$ are polynomial bimoulds.
\end{proof} 
\begin{Remark} Assume that the bimould $B$ is of depth $d$ and write
\[
B\binom{u_1,\ldots,u_d}{v_1,\ldots,v_d}=\sum_{\substack{\mathbf{m}\in\mathbb{Z}_{\geq0}^d \\ \mathbf{k}\in \mathbb{Z}_{\geq1}^d}} b_{\mathbf{m},\mathbf{k}} u^{\mathbf{m}}v^{\mathbf{k}-\mathbf{1}},
\]
where for $\mathbf{m}=(m_1,\ldots,m_d),\ \mathbf{k}=(k_1,\ldots,k_d)$ we abbreviate $u^\mathbf{m}=u_1^{m_1}\cdots u_d^{m_d}$ and $v^{\mathbf{k}-\mathbf{1}}=v_1^{k_1-1}\cdots v_d^{k_d-1}$, and $b_{\mathbf{m},\mathbf{k}}\in\mathbb{Q}$. Then, with the calculations in the proof of Proposition \ref{prop:urit_polynomial} we have
\begin{align*}
&urit(B)(P_{i,j})\binom{u_1,\ldots,u_n}{v_1,\ldots,v_n}\\
&=(u_1+\cdots+u_n)^i\Bigg[v_n^j B\binom{u_1,\ldots,u_{n-1}}{v_1-v_n,\ldots,v_{n-1}-v_n}-v_1^j B\binom{u_2,\ldots,u_n}{v_2-v_1,\ldots,v_n-v_1}\Bigg]\\
&+ (u_1+\cdots+u_n)^iv_1^j \sum_{\substack{\mathbf{m}\in\mathbb{Z}_{\geq0}^d \\ \mathbf{k}\in \mathbb{Z}_{\geq1}^d}} b_{\mathbf{m},\mathbf{k}}\sum_{\mathbf{l}\in\mathbb{Z}_{\geq0}^d} m_{\mathbf{k},\mathbf{l}} \sum_{\substack{d+r+1=n \\ d,r\geq1}} \sum_{\substack{|\mathbf{s}|+s_{r+1}=|\mathbf{k}|-|\mathbf{l}|+1 \\ \mathbf{s}\in \mathbb{Z}_{\geq1}^r, s_{r+1}\geq1}} \\
&\hspace{3cm} \cdot\Big[u[r+1]^\mathbf{m}v[r+1]^{\mathbf{l}-\mathbf{1}} v^{\mathbf{s}-\mathbf{1}}v_n^{s_{r+1}-1}-u[r+2]^\mathbf{m}v[r+2]^{\mathbf{l}-\mathbf{1}}v^{\mathbf{s}-\mathbf{1}}v_{r+1}^{s_{r+1}-1}\Big]
\end{align*}
where for $\mathbf{k}=(k_1,\ldots,k_d),\ \mathbf{l}=(l_1,\ldots,l_d)$ we write $|\mathbf{k}|=k_1+\cdots+k_d$ and
\begin{align*}
m_{\mathbf{k},\mathbf{l}}=(-1)^{|\mathbf{k}|+|\mathbf{l}|}\binom{k_1-1}{l_1-1}\cdots \binom{k_d-1}{l_d-1},
\end{align*}
and $v[r+1]^{\mathbf{l}-\mathbf{1}}$ means a shift of the indices of the $v_i$ by $r$, i.e., 
\begin{align*}
v[r+1]^{\mathbf{l}-\mathbf{1}}=v_{r+1}^{l_1-1}\cdots v_{r+d}^{l_d-1}.
\end{align*}
\end{Remark}

\begin{Theorem} \label{thm:ganit_with_urit}
 Let $A,B \in \BARI^{\operatorname{pol}}$, then we have 
 \begin{align} \label{eq:uri_preuri_conj_new}
 urit(B)(A)=  ganit_{pic} \big( arit(ganit_{poc}(B))(ganit_{poc}(A))\big).    
 \end{align} 
\end{Theorem}

\begin{proof}
See Appendix \ref{subsec:proof_urit}.
\end{proof}
A consequence of this result is Theorem \ref{thm:intro_ilpol} from the introduction.
\begin{Corollary} \label{cor:il_pol_Lie}
The pair $(\BARI_{il}^{\operatorname{pol}},uri)$ is a Lie algebra.
\end{Corollary}

\begin{proof} 
By Proposition \ref{prop:uri_alternil} the $uri$ bracket is a Lie bracket on $\BARI^{\operatorname{rat}}_{il}$. Moreover, we have for $A,B\in\BARI_{il}^{\operatorname{pol}}$
\begin{align*}
&uri(A,B)=ganit_{pic}(ari(ganit_{poc}(A),ganit_{poc}(B))\\
&=ganit_{pic}\big(arit(ganit_{poc}(B))(ganit_{poc}(A))\big)-ganit_{pic}\big(arit(ganit_{poc}(A))(ganit_{poc}(B))\big)\\
&\hspace{0.4cm}+ganit_{pic}\big(lu(ganit_{poc}(A),(ganit_{poc}(B))\big) \\
&=urit(B)(A)-urit(A)(B)+lu(A,B),
\end{align*}
where the last step follows from Theorem \ref{thm:ganit_with_urit} and Lemma \ref{lem:ganit_mu_auto}. As $urit(B)(A)$ and $urit(A)(B)$ both lie in $\BARI^{\operatorname{pol}}$ by Proposition \ref{prop:urit_polynomial} and $lu$ trivially preserves polynomial bimoulds, we deduce that $uri(A,B)\in \BARI_{il}^{\operatorname{pol}}$.
\end{proof}

\begin{Remark}
Similar to Remark \ref{rem:ari_postLie}, it is shown in \cite[Key Lemma 3.18, Theorem 5.52]{Bu_thesis} that $(\BARI^{\operatorname{pol}},lu, -urit )$ is a post-Lie algebra. By Theorem \ref{thm:ganit_with_urit} the post-Lie structure restricts to $\BARI_{il}^{\operatorname{pol}}$. This also implies that $\BARI_{il}^{\operatorname{pol}}$ is a Lie algebra with the bracket
\[
urit(B)(A)-urit(A)(B)+lu(A,B).
\]
\end{Remark} 

\begin{Definition} \label{def:preuri}
For bimoulds $A,B\in\BARI$, define the bimould $preuri(A,B)\in\BARI$ by
\[preuri(A,B)=urit(B)(A)+mu(A,B).\]
\end{Definition}
By Proposition \ref{prop:urit_polynomial}, we get that $preuri$ also preserves the space $\BARI^{\operatorname{pol}}$.

\begin{Conjecture}\label{conj:preuri_with_swap}
Let $A,B \in \BARI_{\underline{il},swap}^{\operatorname{pol}}$, then $preuri(A,B)$ is swap-invariant.
\end{Conjecture}
Actually, we do not expect that we need the alternility of $A$ in this conjecture. An equivalent conjecture is given in Appendix \ref{app:uri_swap_alternative}.

\begin{Proposition}\label{prop:swap_upto_3}
 Conjecture \ref{conj:preuri_with_swap} holds up to depth $3$.   
\end{Proposition}

\begin{proof}
The bimould $preuri(A,B)$ does not have a depth $1$ component and the depth $2$ calculation is straight forward only using the swap invariance of $A$ and $B$ term by term. We note that the swap invariance of the depth $2$ component of $preuri(A,B)$ also follows from Theorem \ref{thm:ari_al_swap}, as the lowest depth component of $A$ and $B$ is alternal and swap invariant and the lowest depth component of $preuri(A,B)$ is just $preari(A,B)$. 

For depth $3$, we use the alternility of $B$ and rewrite $urit(B)(A)$ as
\begin{align*}
&urit(B)(A)\binom{u_1,u_2,u_3}{v_1,v_2,v_3}=arit(B)(A)\binom{u_1,u_2,u_3}{v_1,v_2,v_3}-A\binom{u_1+u_2+u_3}{v_1}\\
&\cdot\Bigg(B\binom{u_2+u_3,-u_3}{v_2-v_1,v_2-v_3}+B\binom{-u_3,u_2+u_3}{v_2-v_3,v_2-v_1}+B\binom{u_2,-u_2-u_3}{v_2-v_3,v_1-v_3}\\
&\hspace{11cm}+B\binom{-u_2-u_3,u_2}{v_1-v_3,v_2-v_3}\Bigg).    
\end{align*}
Then, one checks term by term using the swap invariance of $A$ and $B$ that $preuri(A,B)-swap(preuri(A,B))$ vanishes in depth $3$.
\end{proof}

Combining Corollary \ref{cor:il_pol_Lie} and Conjecture \ref{conj:preuri_with_swap} leads to Conjecture \ref{conj:intro_ilswap} from the introduction.

\begin{Conjecture}\label{conj:uri_il_swap}
The pair $\big(\BARI_{\underline{il},swap}^{\operatorname{pol}},uri\big)$ is a Lie algebra.
\end{Conjecture}
 
By the results above the conjecture holds for depth up to $3$ and in the introduction we already mentioned the numerical evidences for this conjecture obtained by testing the non-commutative version of it with the help of a computer.

In Appendix \ref{app:BvI-stuff} we considered the map $\delta_{BvI}$ on bimoulds. From Example \ref{exa:bvi_one}, we recall
\[
\delta_{BvI}(\mathbf{1}_{mu}) = (  -\frac{u_1 +v_1}{2}, \frac{1}{4}, 0, \ldots ) \in \BARI.
\]
Here $\mathbf{1}_{mu}=(1,0,0,\ldots)$ is the neutral element of $\big(\operatorname{GBARI}, mu\big)$. 

\begin{Definition}\label{def:mystic_der}
 We define the \emph{mystic derivation} to be the map $\widehat\delta: \BARI\to \BARI$ given by
\[
\widehat\delta (A) = \delta_{BvI}(A) - preuri(\delta_{BvI}(\mathbf{1}_{mu}),A)
.
\]   
\end{Definition}

\begin{Remark}
Our definition of $\widehat\delta$ is motivated by  considerations in post-Hopf algebras. More details will be presented in \cite{BuKue26}.
Another description for such a derivation in the language of non-commutative free algebras will be given in \cite{BaBuvI}.  
Finally, we remark that the map $ganit_{pic}\circ  \delta\circ ganit_{poc}$, where $\delta$ is the ari derivation from Definition \ref{def:ari_derivation}, is naturally a derivation for the uri bracket, but  it does not preserve the swap invariance. A correction term is needed, which was found in \cite{Ba_sl2}. 
\end{Remark}

From the construction of the mystic derivation $\widehat\delta$ we immediately get the following property.

\begin{Lemma}  Let $\gr_D$ be the map given in \eqref{def:gr_D_map}, then we have for $A\in\BARI$
\[
\gr_D \widehat\delta(A) =\delta \gr_D(A).
\]
\end{Lemma}

\begin{Conjecture}\label{conj:BILS_derivation} 
The mystic derivation  $\widehat\delta$ is a derivation of   weight 2  on the Lie algebra $\big(\BARI_{\underline{il},swap}^{\operatorname{pol}},uri\big)$.
\end{Conjecture}

The Conjecture \ref{conj:BILS_derivation} actually splits in two claims. The first is that 
\[
\widehat\delta:  \BARI_{\underline{il},swap}^{\operatorname{pol}} \to \BARI_{\underline{il},swap}^{\operatorname{pol}} 
\]
is a well-defined map. The second claim is that
\[
\widehat\delta 
\big( uri(A,B)\big)= uri(  \widehat\delta (A),B) + uri(A,  \widehat\delta (B)).
\]
We tested these claims  within the range of our computing resources. We observe neither $\delta_{BvI}$ preserves $\BARI_{\underline{il},swap}^{\operatorname{pol}}$ nor $preuri$ preserves $\BARI_{\underline{il},swap}^{\operatorname{pol}}$  but their difference does.

\section{Lie subalgebras of \texorpdfstring{$(\ARI_{al},ari)$}{ARIalari}}

In this section, we recall results and conjectures on certain Lie subalgebras of $(\ARI_{al},ari)$, which we will generalize in the following chapters.

\subsection{On the structure of \texorpdfstring{$(\ARI^{\operatorname{pol}}_{\underline{al}\ast\underline{il}},ari)$}{ARIalil}}
 
The structure of the Lie algebra $(\ARI^{\operatorname{pol}}_{\underline{al}\ast\underline{il}},ari)$ is expected to be very simple. Precisely, it is isomorphic 
the double shuffle Lie algebra coming from the theory of multiple zeta values. 
The latter Lie algebra is expected 
to be the free Lie algebra generated by some specific elements, see e.g. \cite{brown_annals}, \cite{ecalle}, \cite{Racinet_thesis}. 

\begin{Conjecture}\label{conj:ALIL_generators}
The Lie algebra $(\ARI^{\operatorname{pol}}_{\underline{al}\ast\underline{il}},ari)$ is freely generated by elements of the form
\[
\widehat{\xi}_{2n+1} = (u_1^{2n},\ldots ), \quad n\geq 1.
\]
With $\mathsf{O}_3(x)=\frac{x^3}{1-x^2}$ we have
\begin{align*}
\sum_{k\geq0} \dim \mathcal{U}(\ARI^{\operatorname{pol}}_{\underline{al}\ast\underline{il}})_{k}\, x^k =\frac{1}{1-\mathsf{O}_3(x)}.
\end{align*}

\end{Conjecture}

The elements $\widehat{\xi}_{2n+1}$ are not unique and an effective construction of those elements is computational difficult \cite{ENR}. A  proposal with canonical choices is presented in \cite{canonicalzeta}. There are also some general approaches by Ecalle \cite{ecalle}, which are not completely understood yet. The number of these generators $\widehat{\xi}_{2n+1}$ is counted by the term $\mathsf{O}_3(x)$.

\subsection{The \texorpdfstring{$\ekma$}{ekma} Lie algebra} \label{subsec:ekma}

To understand the structure of the Lie algebra $(\ARI^{\operatorname{pol}}_{\underline{al}/\underline{al}},ari)$, we first study certain Lie subalgebras.

\begin{Definition}
Let $\ekma = \Lie( \{\xi_{2n+1}\}_{n\ge 1}\,;\,ari)$ be the Lie subalgebra of $(\ARI^{\operatorname{pol}}_{\underline{al}/\underline{al}},ari)$ generated by the ekma moulds $\xi_{2n+1}$ from Example \ref{ex:ekma}.
\end{Definition}

The number of generators of $\ekma$ is given by the generating series $\mathsf{O}_3(x)$.

\begin{Theorem} \label{thm:al_al_cusp}
The space of relations in the Lie algebra  $(\ekma,ari)$ in weight $k$ and depth $2$ is isomorphic to the space of even period polynomials $\widetilde W_k^+ = W_k^+ \slash p_k$ (see Appendix \ref{app:period}). 

The generating series of the dimensions is therefore given by $\mathsf{S}(x)$.
\end{Theorem}

For later purposes, we recall the key steps of the proof of this theorem, see \cite{brown_depth}, \cite{ecalle}.

\begin{proof}
The Lie bracket $ari: \ekma_1 \times \ekma_1 \to \ekma_2$ factors
through the space $\ekma_1 \wedge \ekma_1$, which is spanned by the elements
$ \xi_{r_1}(u_1) \xi_{r_2}(u_2) - \xi_{r_1}(u_2)   \xi_{r_2}(u_1)$, $r_1,r_2\geq3$ odd. Thus, we can identify
$\ekma_1 \wedge \ekma_1$ with the space of homogeneous polynomials $P \in  u_1 u_2 \Q[u_1,u_2]$ of even degree, such that
\begin{align*}
P(u_1,u_2)+P(u_2,u_1) &=0, \\
P(u_1,u_2) =P(\pm u_1,u_2) &=P(u_1,\pm u_2).
\end{align*}
We have a short exact sequence 
\begin{align}\label{ari_l2}
\xymatrix{
0 \ar[r] &\ar[r] \ker (ari) & \ar[r]^{\qquad ari} \ekma_1 \wedge \ekma_1 &\ar[r] \ekma_2 &0 ,
}
\end{align}
where we have by definition of the map $ari$ an identification
\begin{align*}
\ker(ari)=& \{P \in  \ekma_1 \wedge \ekma_1 \mid
P(u_1,u_2)+P(u_1+u_2,u_1)+P(u_2,u_1+ u_2)=0 \}.
\end{align*}
Together with the above description of the space $\ekma_1\wedge\ekma_1$, we conclude that $\ker(ari)$ can be identified with the space of homogeneous polynomials $P \in u_1u_2\Q[u_1,u_2]$ of even degree, such that
\begin{align*}
&P(u_1,u_2)+P(u_2,u_1) =0,\\
&P(u_1,u_2) =P(-u_1,u_2), \\
&P(u_1,u_2)+P(u_1+u_2,u_1)+P(u_2,u_1+ u_2) =0.
\end{align*}
In Appendix \ref{app:period} we identify this space with the even period polynomials coming from cusp forms.   
\end{proof}

It is expected that these cusp form relations between the generators $\xi_{2n+1}$ in the Lie algebra $\ekma$ are all relations and that these relations are free of dependencies. This results into the following dimension conjecture \cite{brown_depth}, \cite{ecalle}.

\begin{Conjecture} We have for the homogeneous subspaces $\mathcal{U}(\ekma)_{k,d}$ of weight $k$ and depth $d$ that
\[ \sum_{k,d\geq0} \dim \mathcal{U}(\ekma)_{k,d} x^ky^d=\frac{1}{1-\mathsf{O}_3(x)y+\mathsf{S}(x)y^2}.\]
\end{Conjecture}

\subsection{The \texorpdfstring{$\carma$}{carma} Lie algebra} \label{subsec:carma}

The Lie algebra $\ekma$ is a strict Lie subalgebra of $(\ARI^{\operatorname{pol}}_{\underline{al}/\underline{al}},ari)$. There are two different constructions to obtain the complementary Lie algebra. 

On the one hand side there is an explicit construction due to Brown \cite{brown_depth}. Starting with a period polynomial $f$ some averaging process yields canonical  elements $e_f \in(\ARI^{\operatorname{pol}}_{\underline{al}/\underline{al}},ari) $  in depth 4.

On the other hand there are three constructions proposed by Ecalle \cite{ecalle}. We restrict ourselves to the carma moulds. 
By Theorem \ref{thm:al_al_cusp}, each even period polynomial $f \in \widetilde{W}^+$ corresponds to a relation of depth $2$
\[
\sum c_{n,m} ari(\xi_{2n+1},\xi_{2m+1}) =0
\]
in the Lie algebra $\ekma$. We denote
\[
\widehat\chi_f = \sum c_{n,m} ari(\widehat\xi_{2n+1},\widehat\xi_{2m+1})  \in \ARI^{\operatorname{pol}}_{\underline{al}*\underline{il}}.
\]
In depth $1$ the generators $\widehat\xi_{2n+1}\in \ARI^{\operatorname{pol}}_{\underline{al}*\underline{il}}$ and $\xi_{2n+1}\in \ekma$ agree, hence the depth $2$ component of $\widehat\chi_f$ is zero. Moreover, the depth $3$ component of $\widehat\chi_f$ must be a polynomial of odd degree and hence by Lemma \ref{lem:al_swap_parity} is must also be zero. Thus, $\widehat\chi_f$ is a mould, which starts at least in depth $4$. We denote by 
$\chi_f  \in \ARI^{\operatorname{pol}}_{\underline{al}/\underline{al}}$ its projection on the depth $4$ component, which is conjecturally nonzero, i.e., conjecturally $\chi_f= \gr_D \widehat \chi_f$. Now the following is believed \cite{brown_depth}, \cite{ecalle}.

\begin{Conjecture}
The Lie algebra $\operatorname{CARMA}=\Lie(\{\chi_f\}_{f\in \widetilde{W}^+}; ari)$ is free and has the generating series
\[
\sum_{k,d\geq0} \dim \mathcal{U}(\operatorname{CARMA})_{k,d} x^k y^d=\frac{1}{1-\mathsf{S}(x) y^4}.
\]
\end{Conjecture}

\subsection{On the structure of \texorpdfstring{$(\ARI^{\operatorname{pol}}_{\underline{al}/\underline{al}},ari)$}{ARIalal}}

It is expected that the ekma moulds from Subsection \ref{subsec:ekma} and the carma moulds from Subsection \ref{subsec:carma} together generate the whole Lie algebra $(\ARI^{\operatorname{pol}}_{\underline{al}/\underline{al}},ari)$. Moreover, no nontrivial relations between ekma and carma moulds are expected. This  yields the following conjecture, which might be seen as a structural refinement of the Broadhurst-Kreimer conjecture \cite{broadhurst-kreimer}, \cite{brown_depth}, \cite{ecalle}. 

\begin{Conjecture}\label{conj:BK_lie} We have
  \[
(\ARI^{\operatorname{pol}}_{\underline{al}/\underline{al}},ari) \cong \Lie(  \{\xi_{2n+1}\}_{n \ge 1} \cup \{ \chi_f \}_{f \in \widetilde{W}^+}\,;\, ari).
  \]  
The generating series of the dimensions of the homogeneous subspaces is given by
\[
\sum_{k,d\geq0} \dim \mathcal{U}(\ARI^{\operatorname{pol}}_{\underline{al}/\underline{al}})_{k,d} x^ky^d=\frac{1}{ 1 - \mathsf{O}_3(x) y +\mathsf{S}(x) y^2 -\mathsf{S}(x) y^4}.
\]
\end{Conjecture}

Conjecturally, $\ARI^{\operatorname{pol}}_{\underline{al}/\underline{al}}$ is isomorphic to the associated depth-graded Lie algebra of $\ARI^{\operatorname{pol}}_{\underline{al}\ast\underline{il}}$.

We now consider the Lie algebra 
\[
(\ARI^{\operatorname{pol}}_{\underline{al}/\underline{al}})^+ = \Lie(  \{\xi_{2n+1}\}_{n \ge 0} \cup \{ \chi_f \}_{f \in \widetilde{W}^+}\,;\, ari).
\]
i.e. the Lie subalgebra  of $\ARI_{al}$ generated by the ekma moulds $\xi_{2n+1}$, $n\geq1$, the carma moulds $\chi_f$, $f\in \widetilde{W}^+$, together with the element $\xi_1=(1,0,0,\ldots)\in \ARI$. 
Then we get a reformulation of Conjecture \ref{conj:BK_lie}, which bears even more similarities to Conjecture \ref{conj:bari_ilsw}.

\begin{Proposition}
If Conjecture \ref{conj:BK_lie} holds, then we have for the homogeneous subspaces $\mathcal{U}((\ARI^{\operatorname{pol}}_{\underline{al}/\underline{al}})^+)_{k,d}$  that
\begin{align*}
\sum_{k,d\geq0} \dim \mathcal{U}( (\ARI^{\operatorname{pol}}_{\underline{al}/\underline{al}})^+)_{k,d} x^ky^d=\frac{1}{1-\mathsf{O}_1(x)y+\mathsf{M}(x) y^2-x\mathsf{S}(x)y^3 - \mathsf{S}(x)y^4  + x\mathsf{S}(x)y^5   },
\end{align*}
where $\mathsf{O}_1(x)=\frac{x}{1-x^2}$ and 
$\mathsf{M}(x) = \sum_{k\ge 4} \dim M_k x^k$.
\end{Proposition}

\begin{proof} From Lemma \ref{lem:orth_rel}, we know that $\xi_1$ is orthogonal to all other elements in $\ARI$ with respect to the ari bracket. Thus
assuming Conjecture \ref{conj:BK_lie}, we get 
\[
\sum_{k,d\geq0} \dim \mathcal{U}( (\ARI^{\operatorname{pol}}_{\underline{al}/\underline{al}})^+)_{k,d} x^ky^d= 
\frac{1}{1-xy} \cdot 
\sum_{k,d\geq0} \dim \mathcal{U}( \ARI^{\operatorname{pol}}_{\underline{al}/\underline{al}})_{k,d} x^ky^d. \hfill\qedhere
\]
\end{proof}
For later purpose we give an interpretation of the individual terms. We call the relations 
\begin{align} \label{eq:ekma_eis_rel}
ari(\xi_1,\xi_{2n+1})=0,\quad n\geq1,
\end{align}
Eisenstein relations, those are counted by the term  
$\mathsf{E}_4(x)=\frac{x^4}{1-x^2}$. Thus taking into account Theorem \ref{thm:al_al_cusp} the ekma moulds satisfy at least
\[ \mathsf{M}(x) = \mathsf{S}(x)+\mathsf{E}_4(x) \]
quadratic relations. 
The Eisenstein relations intersect with the cusp forms relations from Theorem \ref{thm:al_al_cusp} in depth $3$. For example, we have the cusp form relation
\begin{align*}
ari(\xi_3,\xi_9)-3ari(\xi_5,\xi_7)=0,
\end{align*}
and hence
\begin{align*}
&0=ari\big(\xi_1,ari(\xi_3,\xi_9)-3ari(\xi_5,\xi_7)\big)\\
&=-ari\big(\xi_3,ari(\xi_9,\xi_1)\big)-ari\big(\xi_9,ari(\xi_1,\xi_3)\big)+3ari\big(\xi_5,ari(\xi_7,\xi_1)\big)\\
&\hspace{0.4cm}+3ari\big(\xi_7,ari(\xi_1,\xi_5)\big).
\end{align*}
The first row is a Lie product with a cusp form relations and the second row consists of Lie products with Eisenstein relations. These intersections of relations are counted by the term $x\mathsf{S}(x)$.
The carma mould describe the $\mathsf{S}(x)$ generators in depth $4$ and that these are orthogonal to $\xi_1$ is reflected by the 
relations in depth $5$.

\section{Lie subalgebras of \texorpdfstring{$(\BARI_{al,swap},ari)$}{BARIal}}

We study Lie subalgebras of $(\BARI_{al,swap},ari)$, which generalize the Lie algebras studied in Subsection \ref{subsec:ekma} and \ref{subsec:carma}.

\begin{Theorem} \label{thm:ari_al_swap}
The pair $(\BARI_{\underline{al},swap},ari)$ is a bi-graded Lie algebra. 
\end{Theorem}  

The bi-grading is meant with respect to the weight and depth given in Definition \ref{def:weight_depth}.
\begin{proof} See \cite[Section 2.5]{ecalle}, \cite[Theorem 2.5.6]{SchnepsARI}, or \cite{Bu_depthgraded}.
\end{proof} 
 
\begin{Remark}
We have the following table for $\dim \big(\BARI^{\operatorname{pol}}_{\underline{al},swap} \big)_{k,d}$  
	{\small
		\begin{table}[H]  
			\begin{center}
				\begin{tabular}{c|c|c|c|c|c|c|c|c|c|c|c|c|c|c|c|c|c|} 
					$d \backslash k$&1&2&3&4&5&6&7&8&9&10&11&12&13&14&15&16&17\\ \hline
					1&1&0&2&0&3&0&4&0&5&0&6&0&7&0&8&0&9\\ \hline
					2&-&0&0&1&0&2&0& 8&0&14&0&23&0&38&0&58&0\\ \hline
					3&-&-&0&0&1&0&3&0&9&0&27&0&62&0&125&0&238\\ \hline
					4&-&-&-&0&0&1&0&3&0&12&0&37&?&?&?&?&?\\ \hline
					5& -&-&-&-&0&0&1&0&4&0&15&?&?&?&?&?&?\\ \hline
					6& -&-&-&-&-&0&0&1&0&4&0&?&?&?&?&?&? \\ \hline
				\end{tabular}
				\caption{ $  \dim \big(\BARI^{\operatorname{pol}}_{\underline{al},swap} \big)_{k,d}$  \label{tab:dim_bari_al_swap}}
			\end{center}
	\end{table}}
We were not able to deduce any meaningful  pattern out of this table and therefore we did not tried to compute more of these numbers, which would require high performance computers.  The most efficient way to obtain these numbers for $d>3$ is to use that $\BARI^{\operatorname{pol}}_{\underline{al},swap}$ is isomorphic
to  a subspace of the free Lie algebra
  $\Lie(V)$ determined by some involution $\tau$ as  in \cite{Bu_depthgraded}.  
The direct approach by determining those swap invariant multivariate polynomials that satisfy the  alternality equations is limited by the computing resources much earlier.  
\end{Remark}

\begin{Theorem} \label{thm:ALAL_iota} 
 We have a natural embedding of Lie algebras
\[
\iota: \, (\ARI_{\underline{al}/\underline{al}},ari) \to (\BARI_{\underline{al},swap},ari)
\]
given by $\iota(A)(w)=A(u)+swap(A)(v)$.
\end{Theorem}
\begin{proof} A proof of this statement in terms of non-commutative polynomials together with a translation into bimoulds is given in \cite{Bu_depthgraded}.
\end{proof}

\begin{Remark} \label{rem:induced_rel_xi}
We point to the fact that because of Theorem \ref{thm:ALAL_iota} any relation satisfied by elements in  $\ARI_{\underline{al}/\underline{al}}$ 
lifts to a relation in $\BARI_{\underline{al},swap}$ satisfied by the images of these elements.
\end{Remark}

\subsection{The \texorpdfstring{$\biekma$}{bekma} Lie algebra} \label{subsec:bekma}

\begin{Definition}\label{def:q-ekma-algebra}
We define the elements $\xi_{1,0}=\iota(\xi_{2n+1})=(1,0,0,\ldots)$ and
\[
\xi_{2n+1,0}=\iota(\xi_{2n+1})=(u_1^{2n}+v_1^{2n},0,0,\ldots),\quad n\geq1,
\]
and their images under the ari derivation $\delta$ 
\begin{align} \label{def:ekma_bimoulds}
\xi_{2n+1+m,m}&=\delta^m(\xi_{2n+1,0})=\big(  (u_1v_1)^m (u_1^{2n}+v_1^{2n}),0,0,\ldots\big),\quad m,n\geq1,
\end{align} 
and similarly 
$\xi_{m+1,m}=\delta^m(\xi_{1,0})=\big(  (u_1v_1)^m,0,0,\ldots\big)$.

Then, let $\biekma \subset \BARI^{\operatorname{pol}}_{\underline{al},swap}$ be the Lie algebra given by
\[
\biekma = \Lie \big( \{\xi_{2n+1+m,m}\}_{m,n\ge 0}; ari \big).
\]
\end{Definition}

\begin{Lemma} \label{lem:HPS_eq_depth1} The number of generators of $\biekma$ is given by the generating series
\[
b_1(x)= \mathsf{D}(x) \mathsf{O}_1(x)
\]
where $\mathsf{D}(x)=\frac{1}{1-x^2},\ \mathsf{O}_1(x)=\frac{x}{1-x^2}$.
\end{Lemma}

\begin{proof} The generating series $\mathsf{O}_1(x)$ counts the generators $\xi_{2n+1,0}$ and the term $\mathsf{D}(x)$ encodes the derivation $\delta$.
\end{proof}

The elements $\xi_{2n+1}\in \ARI^{\operatorname{pol}}_{\underline{al}/\underline{al}}$ satisfy the Eisenstein relations from \eqref{eq:ekma_eis_rel} and the cusp form relations from Theorem \ref{thm:al_al_cusp}. Therefore, their images
\begin{align*}
\xi_{2n+1,0}=\iota(\xi_{2n+1}),\quad n\geq1,
\end{align*}
satisfy the same relations.

Moreover, if we have a relation in $\biekma$ in weight $n$, then applying the derivation $\delta$ we get a relation in weight $n+2$. For example, for all odd $k \ge 3$ we have in weight $k+1$ the Eisenstein relation 
$ari( \xi_{1,0}, \xi_{k,0})=0$  and therefore we have in weight $k+3$ the relation 
\[
0=ari( \delta  (\xi_{1,0}), \xi_{k,0}) + ari(   \xi_{1,0}, \delta (\xi_{k,0}))=ari(\xi_{2,1},\xi_{k,0})+ari(\xi_{1,0},\xi_{k+1,1}).
\]
We say a relation is primitive if it is non-trivial modulo the image of $\delta$.

In the following theorem we use the notation from Appendix \ref{app:period}.

\begin{Theorem} \label{thm:HPS_eq_depth2}
The vector space  spanned by the primitive relations in weight $k$ and depth $2$ in the Lie algebra $\biekma$ correspond to the vector space of primitive even bi-period polynomials
$\mathcal{P}^{\epsilon}_k$. 
In particular, for the dimensions  
of all relations in depth $2$ we have the generating series
\[ 
b_2(x) = 
\mathsf{D}(x) \, \sum_{k \ge 4} \dim ( M_{k} ( \operatorname{Sl}_2(\mathbb{Z})))^2  \, x^k.
\]  
\end{Theorem}

\begin{proof}
The multiplication $ari: \biekma_1 \times \biekma_1 \to \biekma_2$ factors through
the space $\biekma_1 \wedge \biekma_1$, which is spanned by the elements 
\[
\xi_{r_1,s_1}(w_1)  \xi_{r_2,s_2}(w_2)-\xi_{r_1,s_1}(w_2)   \xi_{r_2,s_2}(w_1)\]
for $r_1,r_2\geq1,\ s_1,s_2\geq0$, $r_1+s_1,r_2+s_2$ odd.
In particular, the space $\biekma_1\wedge \biekma_1$ can be identified with the space of all homogeneous polynomials $P\in \Q[u_1,v_1,u_2,v_2]$ of even degree, such that
\begin{align*}
&P\binom{u_1,u_2}{v_1,v_2}+P\binom{u_2,u_1}{v_2,v_1}=0, \\
&P\binom{u_1,u_2}{v_1,v_2}=P\binom{\pm u_1,u_2}{\pm v_1,v_2}=P\binom{u_1,\pm u_2}{v_1, \pm v_2}, \\
&P\binom{u_1,u_2}{v_1,v_2}=P\binom{v_1,u_2}{u_1,v_2}=P\binom{u_1,v_2}{v_1,u_2}.
\end{align*}
With the notation from Appendix \ref{app:bi_period}, we can rewrite this as
\begin{align*}
\biekma_1\wedge\biekma_1
&=
\left\{\begin{matrix}P\in \Q[u_1,u_2, v_1, v_2] \text{ homogeneous} \\ \text{and of even degree} \end{matrix}\ \middle| \ \begin{matrix} P(X) + P(XS) = 0 \\ 
P( X^t) - P(X) = 0 \\ P ( \epsilon X \epsilon) -P(X) = 0 \end{matrix}\right\},
\end{align*}
where
\[
X= \left(\begin{matrix} u_1 & u_2\\ -v_2 & v_1  \end{matrix} \right),\quad   
S= \left(\begin{matrix} 0& 1  \\ -1& 0 \end{matrix}\right) , \quad  
\epsilon=  \left(\begin{matrix} -1& 0  \\ 0& 1 \end{matrix}\right).
\]
We have a short exact sequence 
\begin{align}\label{eq:ari_l2}
\xymatrix{
0 \ar[r] &\ar[r] \ker (ari) & \ar[r]^{\qquad ari} \biekma_1 \wedge \biekma_1 &\ar[r] \biekma_2 &0 ,
}
\end{align}
where by definition of the bracket $ari$ the space $\ker(ari)$ can be identified with polynomials $P\in \biekma_1 \wedge \biekma_1$ satisfying
\[
P\binom{u_1,u_2}{v_1,v_2} +
P\binom{u_1+u_2,u_1}{v_2,v_1- v_2} +
P\binom{u_2,u_1+u_2}{v_2-v_1,v_1} =0.
\]
As before, we rewrite this as
\begin{align*}
\ker(ari)= \{P \in \biekma_1 \wedge \biekma_1 \mid
P(X)+P(XU)+P(XU^2)=0\},
\end{align*}
where $U=\left(\begin{smallmatrix} 0 & 1 \\ -1 & 1 \end{smallmatrix}\right)$. Together with the description of $\biekma_1\wedge\biekma_1$, we conclude that $\ker(ari)$ can be identified with the homogeneous polynomials $P\in \Q[u_1,u_2, v_1, v_2]$ of even degree, such that
\begin{align*}
&P(X)+P(XS)=P(X)+P(XU)+P(XU^2)=P(X^t)-P(X)=0, \\ 
&P (\epsilon X\epsilon)-P(X)=0.
\end{align*}
In Appendix \ref{app:bi_period} we explain that this space in degree $k-2$ is equal to the space $\mathcal{W}_k^\epsilon$ of even bi-period polynomials. The primitive relations correspond to primitive bi-period polynomials.
\end{proof}

\begin{Example} \label{exa:bekma_relations} 
Recall from  Appendix \ref{app:bi_period} that
a basis for the vector space of primitive even bi-periodpolynomials  $\mathcal{P}^{\epsilon}_k$ is given by the elements $P_{f,g}$, where $f,g$ are either elements of a basis of $W_k^+$  or the same for $W_k^-$.
By Theorem \ref{thm:HPS_eq_depth2} we therefore  have $4$ primitive relations in $\biekma$ in weight $12$. 
The first one is the Eisenstein relation 
\[
ari(\xi_{1,0},\xi_{11,0})=0,
\]
corresponding to the polynomials
\[f=g =  x^{10} - y^{10}.\]

The second one is the cusp form relation induced from $\ekma$, cf. Remark \ref{rem:induced_rel_xi},
\begin{align*}
 0&= ari( \xi_{3,0}, \xi_{9,0}) -3\, ari( \xi_{5,0}, \xi_{7,0}), 
\end{align*}
corresponding to the polynomials
\[
f= (y^2x^8-y^8x^2) - 3(y^4x^6 - y^6x^4),\quad g=x^{10} - y^{10}.
\]
The other two are given by
\begin{align*}
0=&  14 \,ari( \xi_{1, 0}, \xi_{9, 2}) + 27 \,ari( \xi_{1, 0}, \xi_{7, 4})-252 \,ari( \xi_{3, 0},\xi_{ 7, 2}) -405 \,ari( \xi_{3, 0},\xi_{ 5, 4})\\ &-224 \,ari( \xi_{2, 1}, \xi_{8, 1})
-648 \,ari( \xi_{2, 1}, \xi_{6, 3})+630 \,ari( \xi_{5, 0}, \xi_{5, 2})+1008 \,ari( \xi_{4, 1}, \xi_{6, 1})\\
&+2160 \,ari( \xi_{4, 1}, \xi_{4, 3})
+392 \,ari( \xi_{3, 2}, \xi_{7, 0})+2430 \,ari( \xi_{3, 2},\xi_{ 5, 2}),
\end{align*}
corresponding to the polynomials
\[
f=g=(y^2x^8-y^8x^2) -3(y^4x^6 - y^6x^4),
\]
and
\begin{align*}
 0&= 192 \,ari(\xi_{1, 0}, \xi_{10, 1})+  625 \,ari( \xi_{1, 0}, \xi_{8, 3})+ 840 \,ari( \xi_{1, 0}, \xi_{6, 5}) -3600 \,ari( \xi_{3, 0}, \xi_{8, 1}) \\
 &-10500 \,ari( \xi_{3, 0}, \xi_{6, 3}) -1728 \,ari( \xi_{2, 1}, \xi_{9, 0}) -13125 \,ari( \xi_{2, 1}, \xi_{7, 2})-21000 \,ari( \xi_{2, 1}, \xi_{5, 4})\\
 &+ 10080 \,ari( \xi_{5, 0}, \xi_{6, 1})+ 21875 \,ari( \xi_{5, 0}, \xi_{4, 3})+8400 \,ari( \xi_{4, 1}, \xi_{7, 0})+ 52500 \,ari( \xi_{4, 1},\xi_{ 5, 2})\\
 &+ 39375 \,ari( \xi_{3, 2},\xi_{ 6, 1})+ 84000 \,ari( \xi_{3, 2}, \xi_{4, 3}),
\end{align*}
corresponding to the polynomials 
\[f=g=4(yx^9+y^9x) - 25(y^3x^7+y^7x^3) + 42y^5x^5.
\]
\end{Example}

As we have seen in the previous example, the relations in depth $2$ in $\biekma$ contain two special families. Applying the derivation $\delta$ from Definition \ref{def:ari_derivation} to the relations in Lemma \ref{lem:orth_rel}, we obtain 
\[
\delta^m ari(\iota(A),\xi_{1,0})=0 \qquad \text{ for } m\geq0,\ A\in \ARI^{\operatorname{pol}}_{\underline{al}/\underline{al}}.
\]
We call those relations the Eisenstein relations and denote the vector space spanned by them by $\mathfrak{R}_{Eis}$. On the other hand, there are the cusp form relations between the generators $\xi_{2n+1,0}\in \iota\big(\ARI^{\operatorname{pol}}_{\underline{al}/\underline{al}}\big)$, cf. Remark \ref{rem:induced_rel_xi}. By applying also the derivation $\delta$ to those relations, we obtain a second space of relations in depth $2$ which we denote by $\mathfrak{R}_{Cusp}$.

Because of the Jacobi identity the ideals of  the relations 
$\mathfrak{R}_{Cusp}$  and $\mathfrak{R}_{Eis}$ have  a non-trivial intersection in depth $3$. For example, we have
\begin{align*}
0&=ari\big( ari(\xi_{3,0}, \xi_{9,0})    -  3  ari( \xi_{5,0}, \xi_{7,0} ) \big), \xi_{1,0} \big)\\
 &=  - ari\big(   ari( \xi_{9,0} , \xi_{1,0} ) , \xi_{3,0} \big) -ari\big(   ari( \xi_{1,0} , \xi_{3,0} ) , \xi_{9,0} \big) \\
 &\qquad +3 ari\big(   ari( \xi_{7,0} , \xi_{1,0} ) , \xi_{5,0} \big) 
 +3 ari\big(   ari( \xi_{1,0} , \xi_{5,0} ) , \xi_{7,0} \big)
\end{align*}
where the first row is a Lie product with a relation in $\mathfrak{R}_{Cusp}$ and the second row consists of Lie products with relations in $\mathfrak{R}_{Eis}$.

More precisely we have
\[
ari\big(\xi_{1,0}, \mathfrak{R}_{Cusp} \big)  \subset ari\big(\biekma,\mathfrak{R}_{Eis}\big).
\]
Denote by $i_k$ the number of this kind of intersections in $\biekma$ in depth $3$ and weight $k$.

\begin{Lemma} \label{lem:HPS_eq_depth 3}
We have
\[
b_3(x)=\sum_{k\geq3} i_k x^k=\mathsf{D}(x)x\mathsf{S}(x)
\]  
where $\mathsf{S}(x)=\sum_{k \ge 12} \dim S_k \,x^k =\frac{x^{12}}{(1-x^4)(1-x^6)}$.
\end{Lemma}     

\begin{proof} The series $\mathsf{S}(x)$ is exactly the Hilbert-Poincare series of the cusp forms and thus counts the number of period polynomial relations induced by the embedding $\iota$. Since the intersections occur when considering the Lie product of these period polynomial relations and $\xi_{1,0}$, the series $\mathsf{S}(x)$ needs to be multiplied with $x$. Finally, there exist derivations of the intersections of relations illustrated above, this is encoded in the term $\mathsf{D}(x)$. 
\end{proof}

We expect that all relations and intersections of relations in the Lie algebra $\biekma$ are described in Theorem \ref{thm:HPS_eq_depth2}, and Lemma \ref{lem:HPS_eq_depth 3}. Together with Lemma \ref{lem:HPS_eq_depth1}, this leads to the following dimension conjecture.

\begin{Conjecture}\label{conj:q_ekma_dim}
Define the coefficients $e_{k,d}$ by
\begin{align*}
\sum_{k,d\geq0} e_{k,d} x^k y^d 
&= \frac{1}{1-b_1(x)y+b_2(x)y^2-b_3(x)y^3}.
\end{align*}
Then, we have
\[
\dim \mathcal{U}(\biekma)_{k,d}=e_{k,d}.
\]
\end{Conjecture}

A direct consequence of Lemma \ref{lem:HPS_eq_depth1} and Theorem \ref{thm:HPS_eq_depth2} is the following.

\begin{Theorem}\label{thm:q_ekma_dim}
The numbers $e_{k,d}$ from Conjecture \ref{conj:q_ekma_dim} satisfy 
\begin{align*}
\dim \mathcal{U}(\biekma)_{k,d} &= e_{k,d}   \quad \mbox{ for }  d=1,2  \mbox{ and all } k\ge 0\, .\\
\end{align*}
\end{Theorem}

Experiments using Pari/GP with parallel algorithms support the dimension conjecture, for example
$\dim \biekma_{k,d}$ coincides with the expected number from Conjecture \ref{conj:q_ekma_dim} for
up to weight 35 in depth 3 and for higher depth as shown in the following table.
\begin{table}[H]\footnotesize 
\resizebox{1.0\textwidth}{!}{
\begin{tabular} {c|ccccccccccccccccccccccccccc} $d \backslash k$&1&2&3&4&5&
6&7&8&9&10&11&12&13&14&15&16&17&18&19&20&21&22&23&24&25&26&27
\\ \hline \noalign{\medskip}1&1&0&2&0&3&0&4&0&5&0&6&0&7&0&8&0&9&0&10&0&11&0&
12&0&13&0&14\\ \noalign{\medskip}2&0&0&0&1&0&2&0&7&0&12&0&20&0&31&0&47
&0&63&0&89&0&115&0&148&0&186&0\\ \noalign{\medskip}3&0&0&0&0&1&0&3&0&8
&0&22&0&47&0&89&0&161&0&270&0&430&0&663&0&983&0&1414
\\ \noalign{\medskip}4&0&0&0&0&0&1&0&3&0&11&0&30&0&80&0&182&0&392&0&
764&0&1427&0&2507&0&4246&0\\ \noalign{\medskip}5&0&0&0&0&0&0&1&0&4&0&
14&0&44&0&124&0&324&0&780&0&1746&0&3666&0&\bb{7285}&0&\bb{13797}
\\ \noalign{\medskip}6&0&0&0&0&0&0&0&1&0&4&0&17&0&58&0&184&0&519&0&
\bb{1386}&0&\bb{3417}&0&\bb{7969}&0&\bb{17509}&0  \\
 \end{tabular}}
\caption{$\dim \biekma_{k,d}$, computed / \bb{conjectured}  }\label{tab:dim_biekma}
\end{table}
The numbers $e_{k,d}$ and the 
entries in the table are related via a Möbius transform.

\subsection{The \texorpdfstring{$\bicarma$}{bicarma} Lie algebra} \label{subsec:bcarma}

Here we report on a computational approach to obtain bimoulds in $\BARI_{\underline{il},swap}^{\operatorname{pol}}$ of depth $4$ using the Conjecture \ref{conj:intro_ilswap}.

For this construction the obvious refinement of Lemma \ref{lem:grD_ilswap_alswap} to alternil, swap invariant, polynomial bimoulds turns out to be crucial, i.e. that we have
\[
\gr_D\big(\BARI_{\underline{il},swap}^{\operatorname{pol}},uri\big) \subset \big( \BARI_{\underline{al},swap}^{\operatorname{pol}},ari\big).
\]
We expect $\biekma \subset \gr_D\big(\BARI_{\underline{il},swap}^{\operatorname{pol}},uri\big)$. In particular, for each ekma bimould $\xi_{a,b} \in \biekma$ an extension to a swap invariant, alternil bimould 
\[
\widehat\xi_{a,b} = \big( \xi_{a,b}(w_1) ,\, \xi_{a,b}^{(2)}(w_1,w_2),\, \xi_{a,b}^{(3)}(w_1,w_2,w_3), \ldots \big) \in \BARI^{\operatorname{pol}}_{\underline{il},swap}
\]
is expected. It depends on the choices made for
$\widehat\xi_{a-b,0}= \iota\big(\widehat\xi_{a-b}\big)$ and then applying $\widehat\delta^b$ from Conjecture \ref{conj:BILS_derivation} might be a natural choice. 

The main idea is to lift the relations  satisfied by the ekma bimoulds $\xi_{a,b}$ to equations for their lifts $\widehat\xi_{a,b}$ in $\BARI^{\operatorname{pol}}_{\underline{il},swap}$ and then project to the component in depth $4$.   
If we apply this procedure to the second relation of Example \ref{exa:bekma_relations}, then we obtain 
a bimould $\chi_{f,g}\in \BARI^{\operatorname{pol}}_{\underline{al},swap}$ in depth $4$ and weight $12$ from
\begin{align} \label{eq:def_bcarma}
uri\big( \widehat\xi_{3,0} , \widehat\xi_{9,0} \big)
- 3 \,uri\big(  \widehat\xi_{5,0}, \widehat\xi_{7,0} \big) =
\big( 0,0,0, \chi_{f,g},\ldots \big).
\end{align}

Observe the similarity to the construction of the carma moulds we described in Subsection \ref{subsec:carma}. Actually, by Theorem \ref{thm:uri_alil} we get a compatibility to the construction described in Subsection \ref{subsec:carma}. In particular
$\chi_{f,g} = \iota(\chi_\Delta)$ where 
$\Delta$ is the discriminant, i.e. the only cusp form of weight $12$. Note that Conjecture \ref{conj:uri_il_swap} was used  to make sense of the left hand side of \eqref{eq:def_bcarma}.

By Theorem \ref{thm:HPS_eq_depth2} each  primitive quadratic relation in $\biekma$ corresponds to a pair of period polynomials $f,g\in W^+$ or $f,g\in W^-$. 
As $\widehat{\xi}_{a,b}$ and $\xi_{a,b}$ agree in depth $1$, the component in depth $2$ vanishes. Then, the depth $3$ component must be of odd degree and hence vanishes by Lemma \ref{lem:al_swap_parity}\footnote{The statement actually holds for bimoulds, see the proof of \cite[Lemma 2.5.5]{SchnepsARI}}. By Remark \ref{rem:il_to_al} the first nontrivial component $\chi_{f,g}$ is alternal, and by construction also swap invariant. We refer to $\chi_{f,g}$ as carma bimoulds.

By Theorem \ref{thm:HPS_eq_depth2}, we have $( \dim M_k )^2$ primitive quadratic relations of weight $k$ in $\biekma$.
In Conjecture \ref{conj:bari_ilsw} we expect 
the dimension of the space spanned by the new  bimoulds $\chi_{f,g}$ in depth $4$ is
only $(\dim S_{k})^2$.   Thus,  there should be  $2 \dim S_{k} +1$  linear combinations within   these quadratic relations in $\biekma$ that extend to relations in the Lie algebra $\BARI_{\underline{il},swap}^{\operatorname{pol}}$. 
Finally, we observe that the non-primitive quadratic relations in $\biekma$ correspond to $\delta ^m \chi_{f,g}$ with $f,g$ period polynomials of lower weight.

If the quadratic relation in $\biekma$ with respect to the ari brackets lifts to a relations in $\BARI^{\operatorname{pol}}_{\underline{il},swap}$ with respect to the uri bracket, then of course also the depth 4 component, and any higher depth component, vanishes. 
In particular, this holds for the Eisenstein relation given in Lemma \ref{lem:esr_relation_uri}. 
However, since there are so many choices possible in this construction, e.g. choosing a basis for the space of period polynomials, choosing the isotropic vector needed for mapping a pair of period polynomials to a bi-period polynomial or choosing the generators $\widehat \xi_{2n+1}$ for $\ARI^{\operatorname{pol}}_{\underline{al}*\underline{il}}$, 
we can only expect, that the spaces
\begin{align}\label{eq:carma-spaces}
\Big( \langle \delta^m \chi_{f,g} \,| \,f,g\in W^{\pm}_{k-2m} ,\, m \ge 0 \,
\rangle_\Q  + \biekma_{k,4} \,\Big) \Big \slash  \,   \biekma_{k,4}
\end{align}
have the correct dimensions. 

We calculated in depth $4$ the dimension of the spaces spanned by the $\delta^m\chi_{f,g}$ for weight $k\le 26$. 
\begin{center}
 \begin{tabular}{c|ccccccccc}
    weight $k$ & 12&14&16&18&20&22&24&26 \\
   $\dim$ & 1 & 1 & 2 & 3 & 4 & 5 & 9 &10 
  \end{tabular}   
\end{center}


For these experiments we used a computer at DESY Hamburg with 128 cores and 1 terabyte RAM. For $k=26$ the calculation took about a week.  

Let $\bicarma$ 
be the Lie subalgebra generated by the 
basis for \eqref{eq:carma-spaces} for each $k\ge 12$ considered as a subspace of $\big( \BARI^{\operatorname{pol}}_{\underline{al},swap}, ari\big)$.


\begin{Conjecture} \label{conj:bcarma}
The Lie algebra $\bicarma$ is free and we have
 \begin{align*}
\sum_{k,d\geq0} \dim \mathcal{U}(\bicarma)_{k,d} \, x^k y^d &=
\frac{1}{1-b_4(x) y^4},
\end{align*}
where 
\begin{align*} b_4(x) &=  \mathsf{D}(x) \, \sum_{k \ge 12} (\dim  S_{k})^2   \, x^k.
\end{align*}
\end{Conjecture}

\section{Lie subalgebras of \texorpdfstring{$(\BARI_{il}^{\operatorname{rat}},uri)$}{BARIil}}

\subsection{An embedding of
\texorpdfstring{$(\ARI^{\operatorname{pol}}_{\underline{al}*\underline{il}}, ari)$}{ARIalil} 
into  \texorpdfstring{$(\BARI_{il }^{\operatorname{rat}},uri)$}{BARIil}
}

For $A\in \ARI_{\underline{al}\ast\underline{il}}^{\operatorname{pol}}$, denote by $C_A$ the constant mould such that $swap(A)+C_A$ is alternil. Define the map
\begin{align} \label{eq:def_iota_map}
\iota:\ARI_{\underline{al}*\underline{il}}^{\operatorname{pol}} &\rightarrow \BARI_{\underline{il},swap}^{\operatorname{pol}} \\
A&\mapsto A+swap(A)+C_A.  \nonumber
\end{align} 
In this section, we show that $\iota$ is compatible with the $ari$ and $uri$ bracket, i.e., 
\[
\iota(ari(A,B))=uri(\iota(A),\iota(B).
\]
In particular, assuming Conjecture \ref{conj:intro_ilswap} we deduce that \eqref{eq:def_iota_map} would be a Lie algebra morphism $\iota:(\ARI_{\underline{al}*\underline{il}}^{\operatorname{pol}},ari) \rightarrow (\BARI_{\underline{il},swap}^{\operatorname{pol}},uri)$. To prove the statement, we first need some lemmas.

\begin{Lemma}\label{aritu_of_v}
Let $A\in \ARI$ and $Q\in \overline{\ARI}$. Then
\begin{equation}\label{aritandlu}
\arit(A)(Q)=lu(A,Q).
\end{equation}
\end{Lemma}

\begin{proof} For arbitrary bimoulds $A,Q\in \BARI$, the mould 
$arit(A)(Q)$ is given by
\begin{equation}\label{aritdef}
arit(A)(Q)(w)=\sum_{\substack{w=abc\\b,c\ne\emptyset}}
Q(a\lceil c)A(b\rfloor)-\sum_{\substack{w=abc \\ a,b\ne \emptyset}}
Q(a\rceil c)A(\lfloor b).
\end{equation}
But if $A$ is in $\ARI$ the lower flexions disappear since they concern
only the $v$-variables, and similarly if $Q\in \overline{\ARI}$ the
upper flexions disappear since they concern only the $u$-variables. So
we have
$$arit(A)(Q)(w)=\sum_{\substack{w=abc \\ b,c\ne\emptyset}}
Q(ac)A(b)-\sum_{\substack{w=abc\\ a,b\ne \emptyset}}
Q(ac)A(b).$$
All the terms in this sum with $a,b,c\ne \emptyset$ cancel out, leaving
only the terms in the first sum with $a=\emptyset$ and the terms in the 
second sum with $c=\emptyset$, so we have
$$arit(A)(Q)(w)=\sum_{\substack{w=bc\\ b,c\ne\emptyset}}
Q(c)A(b)-\sum_{\substack{w=ab\\ a,b\ne \emptyset}} Q(a)A(b).$$
Renaming the decomposition $w=bc$ as $w=ab$ in the first sum, this becomes
\[arit(A)(Q)(w)=\sum_{\substack{w=ab \\ a,b\ne\emptyset}}
Q(b)A(a)-\sum_{\substack{w=ab\\a,b\ne \emptyset}} Q(a)A(b)=lu(A,Q).
\qedhere \]
\end{proof}

\begin{Lemma}\label{aritarit}
Let $A\in \ARI$ and $Q\in \overline{\ARI}^{\operatorname{rat}}$, and let $D=\arit(\lopil)$. Then
$$\arit\bigl(D^rA\bigr)(Q)=
lu\bigl(D^rA,Q\bigr).$$
\end{Lemma}

\begin{proof} Let $F_r$ denote the operator on $\BARI^{\operatorname{rat}}$ given by
$$F_r=\arit(D^rA)- \operatorname{ad}_{lu}(D^rA).$$
Then we claim that restricted to the space $\overline{\ARI}^{\operatorname{rat}}$ of 
$v$-moulds, the operator $F_r$ is zero.  We prove it by induction on $r$.
The base case $r=0$ is given by the previous lemma. Suppose now that
$F_r$ is zero on $\overline{\ARI}^{\operatorname{rat}}$, and let us compute $F_{r+1}$.
Since $\lopil\in \overline{\ARI}^{\operatorname{rat}}$, we have
$$\arit(D^rA)(\lopil)=lu(D^rA,\lopil)$$
by the induction hypothesis. This shows that
\begin{align}\label{id1}
\ari(\lopil,D^rA)&=\arit(D^rA)(\lopil)-\arit(\lopil)(D^rA)+lu(\lopil,D^rA)\notag\\
&=-\arit(\lopil)(D^rA)\notag\\
&=-D^{r+1}A.
\end{align}
Using the standard operator identity
$$[\arit(P),\arit(Q)]=-\arit\bigl(\ari(P,Q)\bigr)$$
with $P=\lopil$, $Q=D^rA$, we find that
\begin{align}\label{id2}
[\arit(\lopil),\arit(D^rA)]&=-\arit\bigl(\ari(\lopil,D^rA)\bigr)\notag\\
&=\arit\bigl(D^{r+1}A\bigr)
\end{align}
where the last line follows from \eqref{id1}.
Now consider $[\arit(\lopil),\operatorname{ad}_{lu} (D^rA)]$ as an operator on $\BARI^{\operatorname{rat}}$. For $Q\in \BARI^{\operatorname{rat}}$, we have
\begin{align*}
[\arit(\lopil),\operatorname{ad}_{lu}(D^rA)](Q)&=\arit(\lopil)\big(\lu(D^rA,Q)\big)
-lu(D^rA,\arit(\lopil)(Q))\\
&=\lu\bigl(\arit(\lopil)(D^rA),Q)\\
&=\lu\bigl(D^{r+1}A,Q),
\end{align*}
since $\arit(\lopil)$ is a $\lu$-derivation.
This shows that we have the equality of operators
\begin{equation}\label{id3}
[\arit(\lopil),\operatorname{ad}_{lu}(D^rA)]=\operatorname{ad}_{lu}(D^{r+1}A).
\end{equation}
Subtracting \eqref{id3} from \eqref{id2} allows us to calculate the bracket
of the operator $\arit(\lopil)$ with $F_r$ as 
follows:
\begin{equation}\label{id4}
[\arit(\lopil),F_r]=\arit(D^{r+1}A)-\operatorname{ad}_{lu}(D^{r+1}A)=F_{r+1}.
\end{equation}
Now let us apply the left-hand operator to a mould $Q\in \overline{\ARI}^{\operatorname{rat}}$.
By the induction hypothesis, $F_r(Q)=0$ and thus $\arit(\lopil)(F_r(Q))=0$. 
Also, since $\arit(\lopil)(Q)$ also lies in $\overline{\ARI}^{\operatorname{rat}}$, we have
$F_r\bigl(\arit(\lopil)(Q)\bigr)=0$. Thus $F_{r+1}(Q)=0$ for all
$Q\in \overline{\ARI}^{\operatorname{rat}}$, which completes the proof.
\end{proof}

\begin{Lemma}\label{lem:Adaripil_is_mu_aut}
Restricted to the subspace $\ARI\subset
\BARI$, the operator $\Adari(\pil)$ is an automorphism for the $mu$ multiplication.
\end{Lemma}

\begin{proof} By \eqref{eq:Ad_ad_ari_exp} the operator
$\Adari(\pil)$ is given by
\begin{equation}\label{Adaripildef}
\Adari(\pil)=\sum_{n\ge 0} \frac{1}{n!}\adari(\lopil)^n.
\end{equation}
Let $A\in \ARI$. Then taking $Q=\lopil$ in Lemma \ref{aritarit},
we find that
\begin{equation}\label{Lemmaaritarit}
\arit\bigl(\arit(\lopil)^r(A)\bigr)(\lopil)=lu(\arit(\lopil)^r(A),
\lopil).
\end{equation}
We know that
$$\ari(\lopil,B)=\arit(B)(\lopil)-\arit(\lopil)(B)+lu(\lopil,B)$$
for any bimould $B$, so taking $B=\arit(\lopil)^r(A)$, we compute
\begin{align}\label{adarilopil}
\adari&(\lopil)\big( \arit(\lopil)^r(A)\big)=
\ari(\lopil,\arit(\lopil)^r(A))\notag\\
&=\arit\bigl(\arit(\lopil)^r(A)\bigr)\cdot \lopil
-\arit(\lopil)^{r+1}(A)+
lu\bigl(\lopil, \arit(\lopil)^r(A)\bigr)\notag\\
&= 
lu\bigl(\arit(\lopil)^r(A),\lopil)
-\arit(\lopil)^{r+1}(A)
+lu\bigl(\lopil,\arit(\lopil)^r(A)\bigr)  \notag\\
&=-\arit(\lopil)^{r+1}(A).
\end{align}
For the second to last identity we used \eqref{Lemmaaritarit}. 
Using \eqref{adarilopil}, we now show that for $A\in \ARI$, we have
\begin{equation}\label{adarilopiln}
\adari(\lopil)^n(A)=(-1)^n\arit(\lopil)^n(A).
\end{equation}
To see this, we start with $r=0$ in \eqref{adarilopil} to obtain
$$\adari(\lopil)(A)=-\arit(\lopil)(A),$$
proving \eqref{adarilopiln} for $n=1$. Now suppose that \eqref{adarilopiln}
holds for $n$, and let us show it for $n+1$. Applying 
$\adari(\lopil)$ to both sides of \eqref{adarilopiln} gives
\begin{equation*}
\adari(\lopil)^{n+1}(A)=(-1)^n\adari(\lopil)\big( \arit(\lopil)^n(A)\big).
\end{equation*}
Applying \eqref{adarilopil} with $r=n$ to the right-hand side of this, we 
thus find that 
\begin{equation*}
\adari(\lopil)^{n+1}(A)=(-1)^{n+1}\arit(\lopil)^{n+1}(A),
\end{equation*}
proving that \eqref{adarilopiln} is valid for all $n\ge 1$.
This allows us to rewrite $\Adari(\pil)$ from \eqref{Adaripildef} as
\begin{equation}\label{newAdaripildef}
\Adari(\pil)=\sum_{n\ge 0} \frac{(-1)^n}{n!}\arit(\lopil)^n
=\exp\bigl(-\arit(\lopil)\bigr),
\end{equation}
noting that this equality of operators is only valid on the subspace
of $u$-moulds $\ARI$, not on all of $\BARI$. Since $\arit(P)$ is
a $mu$-derivation of $\BARI$ for any bimould $P$, this shows that
restricted to $\ARI$, the operator $\Adari(\pil)$ is a $mu$-automorphism.
\end{proof}

\begin{Proposition}\label{prop:ganit_poc_Ad_ari}
Let $M \in \operatorname{ARI}^{\operatorname{pol}}$. Then
\begin{align}\label{eq:keyidentity}
ganit_{poc}(M) = \operatorname{Ad}_{ari}(pil)(M) \in \BARI^{\operatorname{rat}}. 
\end{align}
\end{Proposition} 

\begin{proof} The space $\ARI^{\operatorname{pol}}$ is generated by the moulds $P_i=(u^i,0,\ldots)$ for  $i\ge0$ with respect to the $mu$ multiplication. By Lemma \ref{lem:ganit_mu_auto}, the operator $\ganitpoc$ is an automorphism with respect to $mu$, and by  Lemma \ref{lem:Adaripil_is_mu_aut}, the same is true for
$\Adari(\pil)$. Direct computation with the defining formula of
$\ganitpoc$ shows that the mould $\ganitpoc(P_i)$ is given in depth 
$d\ge 1$ by 
\begin{equation}\label{eq:specialform}
\frac{(u_1+\cdots+u_d)^i}{(v_1-v_2)\cdots (v_{d-1}-v_d)}.
\end{equation}
Since $\Adari(\pil)$ and $\ganitpoc$ are both $mu$-automorphisms, if we can
show that they agree on the generators $P_i$, then they agree on all of
$\ARI^{\operatorname{pol}}$. It remains thus only to show that
\begin{equation*} 
\Adari(\pil)(P_i)\binom{u_1,\ldots,u_d}{v_1,\ldots,v_d}=\frac{(u_1+\cdots+u_d)^i}{(v_1-v_2)\cdots (v_{d-1}-v_d)}.
\end{equation*}

Let $H_d$ be the bimould concentrated in depth $d$, taking the value 
\eqref{eq:specialform} in depth $d$. By \eqref{newAdaripildef} it suffices to show that
\begin{align} \label{eq:exp(-arit(lopil))_and H_d}
\exp(-arit(lopil))(P_i)\binom{u_1,\ldots,u_d}{v_1,\ldots,v_d}=H_d\binom{u_1,\ldots,u_d}{v_1,\ldots,v_d}.
\end{align}
Let $T:\Q[[x]]\to \Q[[x]]$ be the operator on power series given by
\begin{align*} 
T(g(x))=\frac{d}{dx}(F(x)g(x)),
\end{align*}
where $F(x)=\sum_{m\geq1} c_m x^{m+1}$ and the $c_m$ are the coefficients of $lopil$ given in \eqref{eq:def_coeffs_lopil}. We first show that under the linear map
\[
\rho: \Q[[x]] \to \Q[H_1,H_2,\ldots],\quad x^{d-1} \mapsto H_d,
\]
we have
\begin{align} \label{eq:T_and_exp(lopil)}
\rho \circ \exp(-T)=\exp(-arit(lopil))\circ \rho.    
\end{align} 
On the one hand, we calculate
\begin{equation}\label{eq:T(x^{d-1})}
T(x^{d-1})=\frac{d}{dx}\bigl(F(x) x^{d-1}\bigr)=\sum_{m\ge 1} (d+m)c_mx^{d+m-1}=
\sum_{r>d} rc_{r-d}x^{r-1}.
\end{equation}
Thus, we have to show that for all $r>d$,
\begin{align} \label{eq:arit(lopil)(H_d)}
arit(lopil)(H_d)\binom{u_1,\ldots,u_r}{v_1,\ldots,v_r}=rc_{r-d}H_r\binom{u_1,\ldots,u_r}{v_1,\ldots,v_r}.
\end{align}
Then, we get from \eqref{eq:T(x^{d-1})} and \eqref{eq:arit(lopil)(H_d)} that $\rho\circ T=\arit(lopil)\circ \rho$ and hence also
$\rho\circ T^n=\arit(\lopil)^n\circ \rho$ and thus \eqref{eq:T_and_exp(lopil)}.

To prove \eqref{eq:arit(lopil)(H_d)} we simply apply the defining formula \eqref{aritdef}:
$$\arit(\lopil)(H_d)(w)=\sum_{\substack{w=abc \\ b,c\ne\emptyset}}
H_d(a\lceil c)\lopil(b\rfloor)-\sum_{\substack{w=abc\\a,b\ne \emptyset}}
H_d(a\rceil c)\lopil(\lfloor b).$$
Since $H_d$ is concentrated in depth $d$, when $w=(w_1,\ldots,w_r)$,
the first decomposition sums reduces to the
$d$ decompositions $abc=(w_1,\ldots,w_k)(w_{k+1},\ldots,w_{k+r-d})
(w_{k+r-d+1},\ldots,w_r)$ for $k=0,\ldots,d-1$, and the second sum
reduces to the same $d$ decompositions but for $k=1,\ldots,d$. Writing
\[
(w_1,\ldots,w_r)=\binom{u_1,\ldots,u_r}{v_1,\ldots,v_r}
\]
the expression for $\arit(\lopil)(H_d)(w_1,\ldots,w_r)$ becomes
\begin{align*}
&\sum_{k=0}^{d-1} H_d\binom{u_1,\ldots,u_k,u_{k+1}+\cdots+u_{k+r-d+1},\ldots,u_r}{v_1,\ldots,v_k,v_{k+r-d+1},\ldots,v_r} \\
&\hspace{7cm} 
\cdot\lopil(v_{k+1}-v_{k+r-d+1},\ldots,v_{k+r-d}-v_{k+r-d+1})\\ 
&-\sum_{k=1}^d H_d\binom{u_1,\ldots,u_{k-1},
u_k+\cdots+u_{k+r-d},\ldots,u_r}{v_1,\ldots,v_{k-1},v_k,\ldots,v_r}\lopil(v_{k+1}-v_k,\ldots,v_{k+r-d}-v_k)\\
&=\sum_{k=0}^{d-1} \frac{(u_1+\cdots+u_r)^i}{(v_1-v_2)\cdots(v_{k-1}-v_k)
(v_k-v_{k+r-d+1})(v_{k+r-d+1}-v_{k+r-d+2})\cdots (v_{r-1}-v_r)}\\
&\hspace{2cm} \cdot
\frac{c_{r-d}\bigl(v_{k+1}+\cdots+v_{k+r-d}-(r-d)v_{k+r-d+1}\bigr)}{(v_{k+1}-v_{k+r-d+1})
(v_{k+1}-v_{k+2})\cdots(v_{k+r-d-1}-v_{k+r-d})(v_{k+r-d}-v_{k+r-d+1})}\\ 
&-\sum_{k=1}^{d} \frac{(u_1+\cdots+u_r)^i}{(v_1-v_2)\cdots(v_{k-1}-v_k)
(v_k-v_{k+r-d+1})(v_{k+r-d+1}-v_{k+r-d+2})\cdots (v_{r-1}-v_r)}\\
&\hspace{2cm}
\cdot\frac{c_{r-d}\bigl(v_{k+1}+\cdots+v_{k+r-d}-(r-d)v_k\bigr)}{(v_{k+1}-v_k)
(v_{k+1}-v_{k+2})\cdots(v_{k+r-d-1}-v_{k+r-d})(v_{k+r-d}-v_k)}.
\end{align*}
Note in particular that the flexion terms $H_d(a\rceil c)=H_d(a\lceil c)$
are equal. This sum divides into three parts: the $k=0$ term from
the first sum
\begin{equation}\label{firstpart}
c_{r-d} \frac{(u_1+\cdots+u_r)^i}{(v_{ 1}-v_{ 2})\cdots(v_{r-1}-v_r)}\cdot
\frac{v_1+\cdots+v_{r-d}-(r-d)v_{r-d+1}}{v_1-v_{r-d+1}},
\end{equation}
the $k=d$ term from the second sum
\begin{equation}\label{secondpart}
c_{r-d} \frac{(u_1+\cdots+u_r)^i}{(v_{ 1}-v_{ 2})\cdots(v_{r-1}-v_r)} 
\cdot
\frac{v_{d+1}+\cdots+v_r-(r-d)v_d}{v_r-v_d},
\end{equation}
and finally the sums of inner terms
\begin{align} \label{thirdpart}
& c_{r-d}\frac{(u_1+\cdots+u_r)^i}{(v_1-v_2)(v_2-v_3)\cdots (v_{r-1}-v_r)}\\
&\cdot  \sum_{k=1}^{d-1}\Bigg(\frac{(v_k-v_{k+1})(v_{k+1}+\cdots+v_{k+r-d}-(r-d)v_{k+r-d+1})}{(v_k-v_{k+r-d+1})(v_{k+1}-v_{k+r-d+1})} \nonumber\\
&\hspace{5.5cm}-\frac{(v_{k+r-d}-v_{k+r-d+1})(v_{k+1}+\cdots+v_{k+r-d}-(r-d)v_k)}{(v_k-v_{k+r-d+1})(v_k-v_{k+r-d})}\Bigg). \nonumber
\end{align}
For the second line in \eqref{thirdpart}, we observe 
\[
\frac{v_k-v_{k+1}}{(v_k-v_{k+r-d+1})(v_{k+1}-v_{k+r-d+1})}=
\frac{1}{v_{k+1}-v_{k+r-d+1}}-\frac{1}{v_k-v_{k+r-d+1}},
\]
and similarly for the third line in \eqref{thirdpart}, we get
\[
\frac{v_{k+r-d}-v_{k+r-d+1}}{(v_k-v_{k+r-d+1})(v_k-v_{k+r-d})}=\frac{1}{v_k-v_{k+r-d}}-\frac{1}{v_k-v_{k+r-d+1}}.
\]
Hence, we get
\begin{align*}
& arit(lopil)(H_d)(w)= c_{r-d}\frac{(u_1+\cdots+u_r)^i}{(v_1-v_2)(v_2-v_3)\cdots (v_{r-1}-v_r)}\\
&\hspace{1cm}\cdot \Bigg( \frac{v_1+\cdots+v_{r-d}-(r-d)v_{r-d+1}}{v_1-v_{r-d+1}} + \sum_{k=1}^{d-1}\Bigg(-\frac{v_{k+1}+\cdots+v_{k+r-d}-(r-d)v_k}{v_k-v_{k+r-d}} \\
&\hspace{1.4cm}+ \frac{v_{k+1}+\cdots+v_{k+r-d}-(r-d)v_k}{v_k-v_{k+r-d+1}}-\frac{v_{k+1}+\cdots+v_{k+r-d}-(r-d)v_{k+r-d+1}}{v_k-v_{k+r-d+1}} \\
&\hspace{1.4cm}+\frac{v_{k+1}+\cdots+v_{k+r-d}-(r-d)v_{k+r-d+1}}{v_{k+1}-v_{k+r-d+1}}\Bigg)+ \frac{v_{d+1}+\cdots+v_r-(r-d)v_d}{v_r-v_d}\Bigg).
\end{align*}
First, we see that the terms in the third row reduce to $-(r-d)$. Secondly, the second term in the second row for $k+1$ together with the first term in the fourth row for $k$ simplify to $r-d+1$ (taking also the boundary terms in the second and fourth row into account for $k=1$ and $k=d-1$).
As we get $(d-1)$-times the term $-(r-d)$, $d$-times the term $r-d+1$, and
\[
-(r-d)(d-1)+(r-d+1)d=r,
\]
we deduce that 
\begin{align*}
arit(lopil)(H_d)(w)= rc_{r-d}\frac{(u_1+\cdots+u_r)^i}{(v_1-v_2)(v_2-v_3)\cdots (v_{r-1}-v_r)}=rc_{r-d}H_r\binom{u_1,\ldots,u_r}{v_1,\ldots,v_r},
\end{align*}
which proves \eqref{eq:arit(lopil)(H_d)}.

With $f(x)=1-e^{-x}$, $V=F(x)\frac{d}{dx}$, and $F(x)$ as before, we can rewrite \eqref{eq:def_coeffs_lopil} as
\begin{equation}\label{new1}
\bigl({\mathrm{exp}}(V)\bigr)x=f(x).
\end{equation}
For any power series $h(x)\in \Q[[x]]$, we have
\begin{equation*}
T(\frac{d}{dx}h(x))=\frac{d}{dx}\bigl(F(x) \frac{d}{dx}h(x)\bigr)=\frac{d}{dx}\Bigl(V(h)\Bigr).
\end{equation*}
Applying this repeatedly, we obtain for $g=\frac{d}{dx}h(x)$ that $T^n(g)=\frac{d}{dx}(V^n(h))$, so
\begin{equation}\label{new5}
\mathrm{exp}\bigl(T(g)\bigr)=\frac{d}{dx}\bigl(\mathrm{exp}\bigl(V(h)\bigr).
\end{equation}
Now, $V$ is a derivation and thus $\mathrm{exp}(V)$ is an automorphism,
so by \eqref{new1}, we have
\begin{equation}\label{new6}
\mathrm{exp}(V)\bigl(h(x)\bigr)=h\bigl(\mathrm{exp}(V)(x)\bigr)=h\bigl(f(x)\bigr).
\end{equation}
Then \eqref{new5} and \eqref{new6} together show that
\begin{equation*}
\mathrm{exp}(T)(g)=\frac{d}{dx}\Bigl(h\bigl(f(x)\bigr)\Bigr)
=h'\bigl(f(x)\bigr)f'(x)=g\bigl(f(x)\bigr)f'(x).
\end{equation*}
Letting $f^{-1}(x)=-\mathrm{log}(1-x)$ be the inverse function of $f(x)$, we deduce
\begin{equation*}
\mathrm{exp}(-T)(g)=g\bigl(f^{-1}(x)\bigr)(f^{-1})'(x).
\end{equation*}
In particular, when $g=1$ we have
\begin{equation}\label{eq:exp(-T)(1)}
\mathrm{exp}(-T)(1)=(f^{-1})'(x)=\frac{1}{1-x}=\sum_{d\ge 1} x^{d-1}.
\end{equation}
With \eqref{eq:T_and_exp(lopil)} and \eqref{eq:exp(-T)(1)}, we get
\begin{align*}
\exp(-arit(lopil))(H_1)=\rho \circ \exp(-T)(1)=\rho(\sum_{d\geq1} x^{d-1})=\sum_{d\geq1} H_d,
\end{align*}
which is equal to \eqref{eq:exp(-arit(lopil))_and H_d}.
\end{proof}

\begin{Lemma}\label{lem:uri_alil}
Let $A,B \in\ARI^{\operatorname{pol}}_{\underline{al}*\underline{il}}$, then we have
\begin{align}
uri(A, B) &= ari(A, B), \label{eq:uri_AB}\\
uri (A, swap(B) ) &= 0, \label{eq:uri_AswapB}\\
uri(swap(A), swap(B)) &= swap (ari(A, B)).\label{eq:uri_swapAswapB}
\end{align}
  
\end{Lemma}

\begin{proof}
In order to prove \eqref{eq:uri_AB}, we first note that
\begin{align}
ganit_{poc}\bigl(ari(A,B)\bigr)=ari\bigl(ganit_{poc}( A),ganit_{poc}(B)\bigr).\label{eq:leila7}
\end{align}
Indeed by Proposition \ref{prop:ganit_poc_Ad_ari}, the left-hand
side is equal to $\operatorname{Ad}_{ari}(pil)(ari(A,B))$ and the right-hand side
to $ari\bigl(\operatorname{Ad}_{ari}(pil)(A),\operatorname{Ad}_{ari}(pil)(B)\bigr)$.  
By Lemma \ref{lem:Ad_ari}, the operator $Ad_{ari}(pil)$ respects the 
ari bracket, so these two expressions are equal, proving \eqref{eq:leila7}. The equality \eqref{eq:leila7} implies that
\begin{align*}
uri(A,B)&=ganit_{pic}\bigl( ari\bigl(ganit_{poc}(A),ganit_{poc}(B)\bigr)\bigr)=ganit_{pic}\bigl(ganit_{poc}( ari(A,B))\bigr)\\
&=ari(A,B),
\end{align*}
which proves the  identity   \eqref{eq:uri_AB}.

Next, we prove \eqref{eq:uri_AswapB}. By Proposition \ref{prop:alal_push}, all moulds in $\ARI^{\operatorname{rat}}_{\underline{al}\ast\underline{al}}$ are push invariant, and by Proposition \ref{prop:Ad_ari_alal_alil} the map $\operatorname{Ad}_{ari}(pal)$ is an isomorphism from $\ARI^{\operatorname{rat}}_{\underline{al}\ast\underline{al}}$ to $\ARI^{\operatorname{rat}}_{\underline{al}\ast\underline{il}}$. In particular, we can apply \eqref{eq:ecalle_fundamental_applied} to the moulds in $A, B\in \ARI^{\operatorname{pol}}_{\underline{al}\ast\underline{il}}\subset \ARI^{\operatorname{rat}}_{\underline{al}\ast\underline{il}}$ and obtain
\begin{align*}
ganit_{poc}\bigl(uri\bigl(A,swap(B)\bigr)\bigr)
&=ari\bigl(ganit_{poc}(A),ganit_{poc}(swap(B))\bigr)\cr
&=ari\bigl(ganit_{poc}(A),
Ad_{ari}(pil)\bigl(swap(Ad_{ari}(pal)^{-1}(B))\bigr)\bigr).
\end{align*}
With  Proposition \ref{prop:ganit_poc_Ad_ari} and \ref{lem:Ad_ari} this simplifies to
\begin{align} 
ganit_{poc}\bigl(uri\bigl(A,swap(B)\bigr)\bigr)&= ari\bigl(\operatorname{Ad}_{ari}(pil)(A),
Ad_{ari}(pil)\bigl(swap(Ad_{ari}(pal)^{-1}(B))\bigr)\bigr) \nonumber \\
\label{eq:ganitpoc_uri(A,swap(B))}
&=\operatorname{Ad}_{ari}(pil)\Big(ari\bigl(A,
swap(Ad_{ari}(pal)^{-1}(B))\bigr)\Big).
\end{align}
By Proposition \ref{prop:Ad_ari_alal_alil}, we have $\operatorname{Ad}_{ari}(pal)^{-1}(B) \in \ARI^{\operatorname{rat}}_{\underline{al}\ast\underline{al}}$, and hence 
\[M:=swap\bigl(\operatorname{Ad}_{ari}(pal)^{-1}(B)\bigr)\in \overline{ARI}^{\operatorname{rat}}_{\underline{al}\ast\underline{al}}.\] 
In particular, by Proposition \ref{prop:alal_push} $M$ is push invariant and we can apply Lemma \ref{lem:ari_vs_push} to obtain $
ari(A,M)=0$. We deduce with \eqref{eq:ganitpoc_uri(A,swap(B))} that
\[
ganit_{poc}\bigl(uri\bigl(A,swap(B)\bigr)\bigr)=0,
\]
and hence $uri\bigl(A,swap(B)\bigr)=0$ as desired.

Finally we study \eqref{eq:uri_swapAswapB}. We have for $A,B\in \ARI^{\operatorname{pol}}_{\underline{al}\ast\underline{il}}$ that
\begin{align*}
ari\Bigl(&ganit_{poc}(swap(A)),ganit_{poc}(swap(B))\Bigr) \\
&=ari\Bigl(\operatorname{Ad}_{ari}(pil)\bigl( swap\bigl(\operatorname{Ad}_{ari}(pal)^{-1}(A)\bigr)\bigr),
\operatorname{Ad}_{ari}(pil)\bigl( swap\bigl(\operatorname{Ad}_{ari}(pal)^{-1}(B)\bigr)\bigr)\Bigr) \\
&=\operatorname{Ad}_{ari}(pil)\Bigl(ari\Bigl(swap\bigl(\operatorname{Ad}_{ari}(pal)^{-1}(A)\bigr),
swap\bigl(\operatorname{Ad}_{ari}(pal)^{-1}(B)\bigr)\Bigr)\Bigr),
\end{align*}
where the first equality follows from \eqref{eq:ecalle_fundamental_applied} and the second equality from Lemma \ref{lem:Ad_ari}. With Proposition \ref{prop:pushinv_swap} (note that $\operatorname{Ad}_{ari}(pal)^{-1}(A)\in \ARI^{\operatorname{rat}}_{\underline{al}\ast\underline{al}}$ is  push  invariant) and again Lemma \ref{lem:Ad_ari} we deduce
\begin{align*}
ari\Bigl(&ganit_{poc}(swap(A)),ganit_{poc}(swap(B))\Bigr) \\
&=\operatorname{Ad}_{ari}(pil)\Bigl( swap\Bigl(ari\bigl(\operatorname{Ad}_{ari}(pal)^{-1}(A),
\operatorname{Ad}_{ari}(pal)^{-1}(B)\bigr)\Bigr)\Bigr) \\
&=\operatorname{Ad}_{ari}(pil)\Bigl( swap\Bigl(\operatorname{Ad}_{ari}(pal)^{-1}\bigl( ari(A,B)\bigr)\Bigr)\Bigr).
\end{align*}
Therefore, applying $ganit_{pic}$ to both sides of this identity, we have
\begin{align}
uri(swap(A),swap(B))
&=ganit_{pic}\Bigl(\operatorname{Ad}_{ari}(pil)\Bigl(swap\Bigl(\operatorname{Ad}_{ari}(pal)^{-1}\bigl( 
ari(A,B)\bigr)\Bigr)\Bigr)\Bigr) \nonumber \\
&=ganit_{pic}\bigl(\operatorname{Ad}_{ari}(pil)\bigl(swap(M)\bigr)\bigr) 
\label{eq:uri(swap(A),swap(B))}
\end{align}
with $M=\operatorname{Ad}_{ari}(pal)^{-1}\bigl(ari(A,B)\bigr)$.
By Theorem \ref{thm:alil_lie}, we have $ari(A,B)\in\ARI^{\operatorname{pol}}_{\underline{al}*\underline{il}}$, and hence by Proposition \ref{prop:Ad_ari_alal_alil} we get $M\in \ARI^{\operatorname{rat}}_{\underline{al}\ast\underline{al}}$. In particular, by Proposition \ref{prop:alal_push} $M$ is push invariant and we can apply \eqref{eq:ecalle_fundamental}
to the previous equation \eqref{eq:uri(swap(A),swap(B))} and get
\begin{align*}
uri(swap(A),swap(B))&=swap\bigl(\operatorname{Ad}_{ari}(pal)(M)\bigr) \\
&=swap\Bigl(\operatorname{Ad}_{ari}(pal)\bigl(\operatorname{Ad}_{ari}(pal)^{-1}\bigl(ari(A,B)\bigr)\bigr)\Bigr)\\
&=swap\bigl(ari(A,B)\bigr). \qedhere
\end{align*}
\end{proof}

\begin{Lemma}\label{lem:uri_alil_const}
For $A \in\ARI^{pol}_{\underline{al}*\underline{il}}$ 
and constant moulds $C,C'$, we have
\begin{align}
uri(A , C) = uri( swap(A) , C) = uri(C', C) &= 0. \label{eq:uri_alil_const}
\end{align}
  
\end{Lemma}
\begin{proof} We start by showing that $\uri(A,C)=\uri(C',C)=0$. Since 
a constant mould can be considered as a push-invariant mould in 
$\overline{\ARI}$, we can apply Lemma \ref{lem:ari_vs_push} to see that 
$\ari(A,C)=\ari(C',C)=0$. Now, by
Proposition \ref{prop:ganit_poc_Ad_ari} we have
$$\ganitpoc(M)=\Ad_{ari}(\pil)(M)$$
for all moulds $M\in \ARI^{\operatorname{pol}}$. Therefore, letting $M=A$ or $M=C'$, we have
\begin{align*}
\ari\bigl(\ganitpoc(M),\ganitpoc(C)\bigr)&=
\ari\bigl(\Ad_{ari}(\pil)(M),\Ad_{ari}(\pil)(C)\bigr)\\
&=\Ad_{ari}(\pil)\big(\ari(M,C)\big)\\
&=0,
\end{align*}
since $\Ad_{ari}(\pil)$ is an automorphism for the $\ari$-bracket by Lemma \ref{lem:Ad_ari}. Then by the Definition \ref{def:uri} for the $\uri$-bracket, we see that
$\uri(M,C)=0$ for $M=A$ or $M=C'$.

It remains to show that $\uri\bigl(swap(A),C\bigr)=0$. For this, we
note that since $A\in \ARI^{\operatorname{pol}}_{\underline{al}*\underline{il}}$, 
the mould $A':=\Ad_{ari}(pal)^{-1}\cdot A$
lies in $\ARI^{\operatorname{rat}}_{\underline{al}*\underline{al}}$, so $A'$ is push-invariant and
$A$ is the image of a push-invariant mould under $\Ad_{ari}(pal)$. 
So we can apply \eqref{eq:ecalle_fundamental_applied} to get
\begin{equation}\label{useful}
ganit_{poc}(swap(A)) = \Ad_{ari}(pil)\big( swap(A')\big).   
\end{equation}
We take the expression for $\uri\bigl(swap(A),C\bigr)$ as in Definition 
\ref{def:uri}, and we replace the term $\ganitpoc\bigl(swap(A)\bigr)$ by the
right-hand side of \eqref{useful} and $\ganitpoc(C)$ by 
$\Ad_{ari}(\pil)(C)$ thanks to Proposition \ref{prop:ganit_poc_Ad_ari},
obtaining
\begin{align*}
uri&(swap(A),C)=ganit_{pic}\Bigl(\ari\Bigl(\ganitpoc\bigl(swap(A)\bigr),\ganitpoc(C)\Bigr)\Bigr)\\
&=ganit_{pic}\Bigl(\ari\Bigl(\operatorname{Ad}_{ari}(pil)\bigl(swap(A')\bigr), \Ad_{ari}(\pil)(C)\Bigr)\Bigr)\\
&=ganit_{pic}\Bigl(\Ad_{ari}(\pil)\big(\ari\bigl(swap(A'), C\bigr)\big)\Bigr).
\end{align*}
Now, since as said above we have $A'\in \ARI^{\operatorname{rat}}_{\underline{al}*\underline{al}}$,
we have $swap(A')\in \overline{\ARI}^{\operatorname{rat}}_{\underline{al}*\underline{al}}$, so
like all bialternal bimoulds, both $A'$ and $swap(A')$ are push invariant.
Thus we can again apply Lemma \ref{lem:ari_vs_push} to deduce that
$\ari\bigl(swap(A'),C\bigr)=0$, and therefore $\uri\bigl(swap(A),C\bigr)=0$.
\end{proof}

\begin{Theorem}\label{thm:alil_ilswap_embedding} For $A,B\in \ARI_{\underline{al}*\underline{il}}^{\operatorname{pol}}$, we have
\[
\iota\big(ari(A,B)\big)=uri\big(\iota(A),\iota(B)\big).
\]
In particular, the tuple
$\big( \iota\big(\ARI^{\operatorname{pol}}_{\underline{al}*\underline{il}}\big), uri \big)$ is a Lie algebra.    
\end{Theorem}

\begin{proof}
We compute
\begin{align*}
 uri\bigl(\iota(A),\iota(B)\bigr) &= uri( A +swap(A) +C_A, B +swap(B)+C_B) \\
 & \hspace{-0.5cm}= uri(A,B) + uri(A,swap(B)) + uri(swap(A),B) + uri(swap(A),swap(B)) \\
 & 
+ uri( A+swap(A) +C_A,C_B) + uri(C_A , B+swap(B) +C_B).
\end{align*}
By Lemma \ref{lem:uri_alil}, the first term in the second line is equal to $ari(A, B)$, both the middle terms are zero, and  the fourth term is $swap(ari(A, B))$. The terms in the third line vanish by Lemma \ref{lem:uri_alil_const}. Therefore, we derive
\[
uri\bigl(\iota(A),\iota(B)\bigr) = ari(A,B) + swap\bigl(ari(A,B)\bigr). 
\]
Now, since $A,B\in \ARI^{\operatorname{pol}}_{\underline{al}*\underline{il}}$, we deduce from Theorem \ref{thm:alil_lie} that
$\ari(A,B)\in \ARI^{\operatorname{pol}}_{\underline{al}*\underline{il}}$. Now, as shown in \S 3.4
of \cite{SchnepsARI}, the constant correction $C_A$ for a mould
$A\in \ARI^{\operatorname{pol}}_{\underline{al}*\underline{il}}$ is determined entirely
by the depth-one component of $A$. Since $\ari(A,B)$ has no depth one term,
its associated constant correction term is zero, 
i.e.~$swap\bigl(\ari(A,B)\bigr)$ is already alternil. Thus we have
\[ \iota\bigl(\ari(A,B)\bigr)=\ari(A,B)+swap\bigl(\ari(A,B)\bigr). \qedhere 
\]
\end{proof}

\subsection{Conjectures for  
\texorpdfstring{$\big( \BARI_{\underline{il},swap}^{\operatorname{pol}},uri\big)$}{BARIilswap}} \label{subsec:bari_ilswap}

Throughout this subsection we assume Conjecture \ref{conj:uri_il_swap}, i.e. that $\big(\BARI_{\underline{il},swap}^{\operatorname{pol}},uri\big)$ is a Lie algebra, and Conjecture \ref{conj:BILS_derivation}, i.e. that $\widehat\delta$ is a derivation on the Lie algebra $\big(\BARI_{\underline{il},swap}^{\operatorname{pol}},uri\big)$.

By Conjecture \ref{conj:ALIL_generators}, the Lie algebra $\ARI^{\operatorname{pol}}_{\underline{al}\ast\underline{il}}$ is generated by elements $\widehat\xi_{2n+1},\ n\geq1$. We define the elements 
\[
\widehat\xi_{2n+1,0} = \iota( \widehat\xi_{2n+1} ),\quad n\geq1,
\]
where $\iota:\ARI^{\operatorname{pol}}_{\underline{al}\ast\underline{il}}\to \BARI_{\underline{il},swap}^{\operatorname{pol}}$ is defined in \eqref{eq:def_iota_map}. Moreover, we set $\widehat\xi_{1,0}=(1,0,0,\ldots)$ and define
\begin{align}\label{eq:biekma_extension}
\widehat\xi_{2n+1+m,m} 
= \widehat{\delta}^{m}( \widehat\xi_{2n+1,0}), \quad n,m\geq0.
\end{align}
Observe that $\gr_D\widehat\xi_{a,b}=\xi_{a,b}$, therefore they are alternil, swap invariant extensions of the ekma bimoulds $\xi_{a,b}\in \BARI_{\underline{al},swap}^{\operatorname{pol}}$ given in \eqref{def:ekma_bimoulds}.

\begin{Conjecture}\label{conj:biekma_il}
We have
\[
\big(\BARI_{\underline{il},swap}^{\operatorname{pol}},uri\big)  =  \Lie\big(\{\widehat\xi_{2n+1+m,m}\}_{m,n\ge 0}; uri\big).
\]
The generating series of the dimensions of the homogeneous subspaces of the universal enveloping algebra is given by
\begin{align*}
\sum_{k\geq0} \dim \mathcal{U}(\BARI_{\underline{il},swap}^{\operatorname{pol}})_{k} \,x^k  &=\frac{1}{1-\mathsf{D}(x)\mathsf{O}_1(x)+\mathsf{D}(x)\sum_{k\geq4} \dim(M_k\oplus S_k)x^k}.
\end{align*}
\end{Conjecture}

As mentioned several times before we have numerical evidences for this conjecture also.


\begin{Lemma} \label{lem:esr_relation_uri}
We have the Eisenstein relations
\[
uri(\widehat\xi_{2n+1,0},\widehat\xi_{1,0})=0, \qquad n\geq1.
\]
\end{Lemma}
\begin{proof}
This follows from Theorem \ref{thm:alil_ilswap_embedding} and Lemma \ref{lem:orth_rel}.
\end{proof}

Beside the Eisenstein relations and their derivatives no other explicit general description for the relations in $\BARI^{\operatorname{pol}}_{\underline{il},swap}$ has been found so far, cf. Subsection \ref{subsec:bcarma}.

\subsection{Conjectures for
\texorpdfstring{$\gr_D \big(\BARI_{\underline{il},swap}^{\operatorname{pol}},uri\big)$}{grDBARI}}

We use the two Lie algebras $\biekma$ and $\bicarma$ from Subsection \ref{subsec:bekma} and \ref{subsec:bcarma} to describe the conjectural structure of $\gr_D\big( \BARI_{\underline{il},swap}^{\operatorname{pol}},uri\big)$.

\begin{Conjecture} We have
\[
\gr_D\big( \BARI_{\underline{il},swap}^{\operatorname{pol}},uri\big)= \Lie\big(\{\xi_{2n+1+m,m})\}_{m,n \ge 0} \cup \{\delta^m(\chi_{f,g})\}_{f,g\in W^\pm,m\geq0};ari\big).
\]    
\end{Conjecture}

By Lemma \ref{lem:orth_rel}, we have
\[
ari(\xi_1,\carma)=0.
\]
Those relations give rise to relations in $\gr_D\big( \BARI_{\underline{il},swap}^{\operatorname{pol}},uri\big)$ in depth $5$, precisely we have
\begin{align} \label{eq:rel_xi1_iota(carma)}
ari(\xi_{1,0},\iota(\carma))=0,
\end{align}
and also derivatives of those relations
\begin{align} \label{eq:rel_deriv_xi1_iota(carma)}
\delta^m\big(ari(\xi_{1,0},\iota(\carma))\big)=0.
\end{align}
Denote the number of these kind of relations in weight $k$ by $j_k$.
 
\begin{Lemma} \label{lem:HPS_eq_depth 5}
We have 
\begin{align*}
b_5(x) &=  \sum_{k\geq5} j_kx^k= \mathsf{D}(x) \,x \mathsf{S}(x).
\end{align*} 
\end{Lemma}      

\begin{proof}
The number of generators in $\carma$ are counted by the term $\mathsf{S}(x)$. As we consider the ari bracket with $\xi_{1,0}$ in \eqref{eq:rel_xi1_iota(carma)}, we have to multiply with $x$. Finally, we also consider derivatives in \eqref{eq:rel_deriv_xi1_iota(carma)}, which are encoded in the term $\mathsf{D}(x)$.
\end{proof}

We do not expect any further relations or dependencies of relations between the Lie algebras $\biekma$ and $\bicarma$. With Lemma \ref{lem:HPS_eq_depth1}, Theorem \ref{thm:HPS_eq_depth2}, Lemma \ref{lem:HPS_eq_depth 3}, Conjecture \ref{conj:bcarma}, and Lemma \ref{lem:HPS_eq_depth 5}, this results in the following dimension conjecture.

\begin{Conjecture}
We have for the dimensions of the homogeneous subspaces of the universal enveloping algebra $\mathcal{U}\big(\gr_D\big( \BARI_{\underline{il},swap}^{\operatorname{pol}},uri\big)\big)$ that
\begin{align*}
\sum_{k,d\geq0} \dim \mathcal{U}(\gr_D &\big( \BARI_{\underline{il},swap}^{\operatorname{pol}},uri\big))_{k,d} \,x^k y^d \\
&=\frac{1}{1-b_1(x)y + b_2(x) y^2  - b_3(x) y^3 -  b_4(x)y^4 +b_5(x) y^5}.
\end{align*}
\end{Conjecture}

\begin{Remark}
We have a proper inclusion $ \gr_D \big( \BARI_{il,swap}^{pol} , uri \big) \subsetneq \big( \BARI_{al,swap}^{pol} , ari \big)$. For example, the calculated dimensions of those spaces in weight $8$ and depth $2$ differ, cf. table \ref{tab:dim_bari_al_swap} and \ref{tab:dim_biekma}.
\end{Remark}

\newpage
\appendix

\section{Bimoulds and quasi-shuffle products} \label{app:bimoulds_quasish}

\subsection{Symmetries via the coefficient map}
Let $A$ be an \emph{alphabet}, so $A$ is a countable set whose elements are called \emph{letters}. By $\Q A$ we denote the $\Q$-module spanned by the letters of $A$ and we let $\Q\langle A\rangle $ be the free non-commutative algebra generated by the alphabet $A$. The monic monomials in $\Q\langle A\rangle$ are called \emph{words} with letters in $A$, and the set of all words is denoted by $A^*$.

We now relate different algebraic structures on $\Q\langle A\rangle$ to the theory of bimoulds. In the context of alphabets we often use $u,v$ to denote words. Therefore, to distinguish those clearly from the variables $u_i,v_i$ used for bimoulds before, we will now identify $u_i=Y_i$ and $v_i=X_i$, $i\ge 1$, in the following. 

\begin{Definition} \label{def:quasi-shuffle}
Let $\diamond:\Q A\times \Q A\to \Q A$ be an associative and commutative product. Define the \emph{quasi-shuffle product} $\ast_\diamond$ on $\Q\langle A\rangle$ recursively by $1\ast_\diamond w=w\ast_\diamond 1=w$ and 
\begin{align*}
au\ast_\diamond bv=a(u\ast_\diamond bv)+b(au\ast_\diamond v)+(a\diamond b)(u\ast_\diamond v)
\end{align*}
for all $u,v,w\in \Q\langle A\rangle $ and $a,b\in A$.
\end{Definition} 

There are several well-known quasi-shuffle products. If
\begin{align} \label{eq:def_shuffle}
a\diamond b=0 \text{ for all } a,b\in A, 
\end{align}
then we get the \emph{shuffle product}, which is usually denoted by $\shuffle$.

More specific, consider the bi-alphabet $Y=\{y_{k,m}\mid k\geq1,\ m\geq0\}$. On $\Q Y$ define the product 
\begin{align} \label{eq:def_stuffle}
y_{k_1,m_1}\diamond y_{k_2,m_2}= y_{k_1+k_2,m_1+m_2}.
\end{align}
We call the associated quasi-shuffle product the \emph{stuffle product} and denote it by $\ast$. 

For a given quasi-shuffle algebra $(\Q\langle A\rangle,\ast_\diamond)$, define the \emph{generic diagonal series} of $\Q\langle A\rangle$ by
\begin{align*}
\mathcal{W}(A)=\sum_{w\in A^*} w\otimes w.
\end{align*}
Let $\dep:A^*\to \mathbb{Z}_{\geq0}$ be a \emph{depth map} compatible with concatenation, i.e., we have 
\[\dep(uv)=\dep(u)+\dep(v),\qquad u,v\in A^*.\] 
Denote by $(A^*)^{(d)}$ the set of all words in $A^*$ of depth $d$ and by $\Q\langle A\rangle^{(d)}$ the space spanned by $(A^*)^{(d)}$. 
\begin{Definition} \label{def:gen_series_words} Let $\rho_A:\Q\langle A\rangle\to\Q[Z_1,Z_2,\ldots]$ be a $\Q$-linear map having the following properties with respect to the depth map:
\begin{itemize}
\item[(i)] There is a strictly increasing sequence $\ell(0)<\ell(1)<\ell(2)<\dots$ of non-negative integers, such that for each $d\geq0$ the restriction of $\rho_A$ to $\Q\langle A\rangle^{(d)}$ is an injective map
\[\rho_A|_{\Q\langle A\rangle^{(d)}}:\Q\langle A\rangle^{(d)}\to \Q[Z_1,\ldots,Z_{\ell(d)}].\] 
\item[(ii)] We have
\[\rho_A(uv)=\rho_A(u)\rho_A^{[\ell(n)]}(v), \qquad u\in\Q\langle A\rangle^{(n)},\ v\in \Q\langle A\rangle,\]
where $\rho_A^{[n]}$ denotes the $\Q$-linear map obtained from $\rho_A$ by shifting the variables $Z_i$ to $Z_{n+i}$.
\end{itemize}
The (commutative) \emph{generating series of words} in $\Q\langle A\rangle$ associated to $\rho_A$ are given by
\[\rho_A(\mathcal{W})_d(Z_1,\ldots,Z_{\ell(d)})=\sum_{w\in (A^*)^{(d)}} w\rho(w)\in \Q\langle A\rangle\llbracket Z_1,\ldots,Z_{\ell(d)}\rrbracket , \qquad d\geq0.\]
\end{Definition}

\begin{Proposition} 
Let $\rho_A:\Q\langle A\rangle \to\Q[Z_1,Z_2,\ldots]$ be a $\Q$-linear map as in Definition \ref{def:gen_series_words} with $\ell(d_1)+\ell(d_2)=\ell(d_1+d_2)$ for all $d_1,d_2\geq0$. Then the $\Q\llbracket Z_1,Z_2,\ldots\rrbracket $-linear extension of the concatenation product satisfies
\[ \rho_A(\mathcal{W})_n(Z_1,\ldots,Z_{\ell(n)})\cdot \rho_A(\mathcal{W})_{d-n}(Z_{\ell(n)+1},\ldots,Z_{\ell(d)})=\rho_A(\mathcal{W})_d(Z_1,\ldots,Z_{\ell(d)})
\]
for all $0\leq n\leq d$.
\end{Proposition}

For the alphabet $Y=\{y_{k,m}\mid k\geq 1, m\geq0\}$, a depth map is given by
\begin{align*}
\dep: Y^*\to \mathbb{Z}_{\geq0},\quad y_{k_1,m_1}\cdots y_{k_d,m_d}\mapsto d.
\end{align*}
The map 
\begin{align*}
\rho_Y:\Q\langle Y\rangle&\to \Q[X_1,Y_1,X_2,Y_2,\ldots], \\
y_{k_1,m_1}\dots y_{k_d,m_d}&\mapsto X_1^{k_1-1}\frac{Y_1^{m_1}}{m_1!}\dots X_d^{k_d-1}\frac{Y_d^{m_d}}{m_d!}
\end{align*}
satisfies the conditions from Definition \ref{def:gen_series_words} with $\ell(d)=2d$. The associated generating series of words are given by $\rho_Y(\mathcal{W})_0=1$ and for $d\geq1$ by
\begin{align} \label{eq:def_rhoY}
\rho_Y(\mathcal{W})_d\binom{X_1,\ldots,X_d}{Y_1,\ldots,Y_d}=\sum_{\substack{k_1,\ldots,k_d\geq1 \\ m_1,\ldots,m_d\geq0}} y_{k_1,m_1}\dots y_{k_d,m_d}X_1^{k_1-1}\frac{Y_1^{m_1}}{m_1!}\dots X_d^{k_d-1}\frac{Y_d^{m_d}}{m_d!}.
\end{align}

Extend the quasi-shuffle product $\ast_\diamond$ on $\Q\langle A\rangle$ to $\Q\langle A\rangle\llbracket Z_1,Z_2,\ldots\rrbracket$ by $\Q\llbracket Z_1,Z_2,\ldots\rrbracket$-linearity, then by definition of quasi-shuffle products we have for all $0\leq n\leq d$ that
\[\rho_A(\mathcal{W})_n(Z_1,\ldots,Z_{\ell(n)})\quasish \rho_A(\mathcal{W})_{d-n}(Z_{\ell(n)+1},\ldots,Z_{\ell(d)})\in\Q\langle A\rangle\llbracket Z_1,\ldots,Z_{\ell(d)}\rrbracket.\]
For some quasi-shuffle products, in particular the ones mentioned before, it is possible to describe this product on the generating series of words by a recursive explicit formula with respect to concatenation. This will explain the origin of a particular kind of symmetries occurring in Ecalle's theory of bimoulds. 

\begin{Proposition} \label{prop:quasish_gen_series}
Let $Y=\{y_{k,m}\mid k\geq1,m\geq0\}$
\begin{enumerate}
\item For the shuffle product $\shuffle$ on $\Q\langle Y\rangle$ from \eqref{eq:def_shuffle}, one obtains for all $0<n<d$ that $1\shuffle \rho_Y(\mathcal{W})_n=\rho_Y(\mathcal{W})_n\shuffle 1=\rho_Y(\mathcal{W})_n$ and
\begin{align*} 
&\rho_Y(\mathcal{W})\binom{X_1,\ldots,X_n}{Y_1,\ldots,Y_n}\shuffle \rho_Y(\mathcal{W})\binom{X_{n+1},\ldots,X_d}{Y_{n+1},\ldots,Y_d}\\
&=\rho_Y(\mathcal{W})\binom{X_1}{Y_1}\cdot\left(\rho_Y(\mathcal{W})\binom{X_2,\ldots,X_n}{Y_2,\ldots,Y_n}\shuffle \rho_Y(\mathcal{W})\binom{X_{n+1},\ldots,X_d}{Y_{n+1},\ldots,Y_d} \right)\\
&+\rho_Y(\mathcal{W})\binom{X_{n+1}}{Y_{n+1}}\cdot\left(\rho_Y(\mathcal{W})\binom{X_1,\ldots,X_n}{Y_1,\ldots,Y_n}\shuffle \rho_Y(\mathcal{W})\binom{X_{n+2},\ldots,X_d}{Y_{n+2},\ldots,Y_d}\right).
\end{align*} 
\item For the stuffle product $\ast$ on $\Q\langle Y\rangle$ from \eqref{eq:def_stuffle}, one obtains for all $0<n<d$ that $1\ast \rho_Y(\mathcal{W})_n=\rho_Y(\mathcal{W})_n\ast 1=\rho_Y(\mathcal{W})_n$ and 
\begin{align*} 
&\rho_Y(\mathcal{W})\binom{X_1,\ldots,X_n}{Y_1,\ldots,Y_n}\ast \rho_Y(\mathcal{W})\binom{X_{n+1},\ldots,X_d}{Y_{n+1},\ldots,Y_d}\\
&=\rho_Y(\mathcal{W})\binom{X_1}{Y_1}\cdot\left(\rho_Y(\mathcal{W})\binom{X_2,\ldots,X_n}{Y_2,\ldots,Y_n}\ast \rho_Y(\mathcal{W})\binom{X_{n+1},\ldots,X_d}{Y_{n+1},\ldots,Y_d} \right)\\
&+\rho_Y(\mathcal{W})\binom{X_{n+1}}{Y_{n+1}}\cdot\left(\rho_Y(\mathcal{W})\binom{X_1,\ldots,X_n}{Y_1,\ldots,Y_n}\ast \rho_Y(\mathcal{W})\binom{X_{n+2},\ldots,X_d}{Y_{n+2},\ldots,Y_d}\right) \\ &+\frac{\rho_Y(\mathcal{W})\binom{X_1}{Y_1+Y_{n+1}}-\rho_Y(\mathcal{W})\binom{X_{n+1}}{Y_1+Y_{n+1}}}{X_1-X_{n+1}}\\
&\hspace{3cm}\cdot \left(\rho_Y(\mathcal{W})\binom{X_2,\ldots,X_n}{Y_2,\ldots,Y_n}\ast \rho_Y(\mathcal{W})\binom{X_{n+2},\ldots,X_d}{Y_{n+2},\ldots,Y_d}\right).
\end{align*} 
\end{enumerate}
\end{Proposition}

\begin{Definition} \label{def:rho_phi_gen_series}
Let $\rho_A:\Q\langle A\rangle\to \Q[Z_1,Z_2,\ldots]$ be as in Definition \ref{def:gen_series_words}, $R$ be any commutative $\QQ$-algebra and $\varphi:\Q\langle A\rangle\to R$ be a $\Q$-linear map. The \emph{(commutative) generating series with coefficients in $R$} associated to $(\varphi,\rho_A)$ are given by
\[(\varphi\otimes\rho_A)(\mathcal{W})_d(Z_1,\ldots,Z_{\ell(d)})=\sum_{w\in (A^*)^{(d)}} \varphi(w)\rho_{\A}(w)\in R\llbracket Z_1,\ldots,Z_{\ell(d)}\rrbracket ,\qquad d\geq0.\]
\end{Definition} 

If $\rho_A:\Q\langle A\rangle \to \Q[X_1,X_2,\ldots]$ is a map as in Definition \ref{def:gen_series_words} with $\ell(d)=d$ for all $d\geq0$, and $\varphi:\Q\langle A\rangle\to R$ is any linear map, then the family of generating series $(\varphi\otimes \rho_A)(\mathcal{W})=((\varphi\otimes\rho_A)(\mathcal{W})_d)_{d\geq0}$ is a mould with coefficients in $R$. 

Similarly, if $\rho_A:\Q\langle A\rangle\to \Q[X_1,Y_1,X_2,Y_2,\ldots]$ is a map as in Definition \ref{def:gen_series_words} with $\ell(d)=2d$ for every $d\geq0$, and $\varphi:\Q\langle A\rangle \to R$ is any $\Q$-linear map, then the family of generating series $(\varphi\otimes\rho_A)(\mathcal{W})=((\varphi\otimes\rho_A)(\mathcal{W})_d)_{d\geq0}$ is a bimould with coefficients in $R$.

\begin{Definition} \label{def:phi_rho_sym}
A sequence $M=(M_d)_{d\geq0}\in\prod_{d\geq0} R\llbracket Z_1,\ldots,Z_{\ell(d)}\rrbracket $ is called \emph{$(\varphi_{\ast_\diamond},\rho_A)$-symmetric} if there exists a $\Q$-algebra morphism $\varphi_{\ast_\diamond}:(\Q\langle A\rangle,\quasish)\to R$ and a $\Q$-linear map $\rho_A:\Q\langle A\rangle \to\Q[ Z_1,Z_2,\ldots]$ satisfying the conditions in Definition \ref{def:gen_series_words}, such that for all $d\geq0$
\[M_d=(\varphi_{\ast_\diamond}\otimes\rho_A)(\mathcal{W})_d.\]
\end{Definition} 
Let $M=(M_d)_{d\geq0}\in \prod_{d\geq0} R\llbracket Z_1,\ldots,Z_{\ell(d)}\rrbracket $ be a $(\varphi_{\ast_\diamond},\rho_A)$-symmetric sequence. Then, one obtains immediately from the definition that for $0<n<d$
\begin{align} \label{formula phi,rho sym}
&M_n(Z_1,\ldots,Z_{\ell(n)})M_{d-n}(Z_{\ell(n)+1},\ldots,Z_{\ell(d)})\\
&\hspace{3cm}=\varphi_{\ast_\diamond}\Big(\rho_A(\mathcal{W})_n(Z_1,\ldots,Z_{\ell(n)})\quasish \rho_A(\mathcal{W})_{d-n}(Z_{\ell(n)+1},\ldots,Z_{\ell(d)})\Big). \nonumber
\end{align}
The right-hand side is an element in $\Qnca{\A}\llbracket Z_1,\ldots,Z_{\ell(d)}\rrbracket $, so the map $\varphi_{\quasish}:\Qnca{\A}\to R$ needs to be extended by $\QQ\llbracket Z_1,Z_2,\ldots\rrbracket $-linearity. 

For the shuffle product from \eqref{eq:def_shuffle} and the stuffle product from \eqref{eq:def_stuffle}, the symmetries from Definition \ref{def:phi_rho_sym} reduces to well-known ones among bimoulds.

\begin{Definition} \label{def:symmtral_il}
\begin{enumerate}
\item A bimould $M=(M_d)_{d\geq0}\in \prod_{d\geq0}R[X_1,Y_1,\ldots,X_d,Y_d]$ is \emph{symmetral} if there is an algebra morphism $\varphi_\shuffle:(\Q\langle Y\rangle,\shuffle)\to R$, such that $M$ is $(\varphi_\shuffle,\rho_Y)$-symmetric, i.e., we have $M_0=1$ and $M_d$ is for $d\geq 1$ given 
\[M_d\binom{X_1,\ldots,X_d}{Y_1,\ldots,Y_d}=\sum_{\substack{k_1,\ldots,k_d\geq1 \\ m_1,\ldots,m_d\geq0}} \varphi_\shuffle(y_{k_1,m_1}\dots y_{k_d,m_d})X_1^{k_1-1}\frac{Y_1^{m_1}}{m_1!}\dots X_d^{k_d-1}\frac{Y_d^{m_d}}{m_d!}.\]
We refer to $\varphi_\shuffle$ as the \emph{coefficient map} of $M$.
\item A bimould $M=(M_d)_{d\geq0}\in \prod_{d\geq0}R[X_1,Y_1,\ldots,X_d,Y_d]$ is \emph{symmetril} if there is an algebra morphism $\varphi_\ast:(\Q\langle Y\rangle,\ast)\to R$, such that $M$ is $(\varphi_\ast,\rho_Y)$-symmetric.
\end{enumerate}
\end{Definition} 
\begin{Example}
Applying the recursive formula for $\ast$ on generating series of words given in Proposition \ref{prop:quasish_gen_series}, one obtains that symmetrility in depths $2$ and $3$ is equivalent to 
\begin{align*}
M\binom{X_1}{Y_1}\cdot M\binom{X_2}{Y_2} =&\ M\binom{X_1,X_2}{Y_1,Y_2}+M\binom{X_2,X_1}{Y_2,Y_1}+\frac{M\binom{X_1}{Y_1+Y_2}-M\binom{X_2}{Y_1+Y_2}}{X_1-X_2}, \\
M\binom{X_1}{Y_1}\cdot M\binom{X_2,X_3}{Y_2,Y_3} = & \ M\binom{X_1,X_2,X_3}{Y_1,Y_2,Y_3}+M\binom{X_2,X_1,X_3}{Y_2,Y_1,Y_3}+M\binom{X_2,X_3,X_1}{Y_2,Y_3,Y_1} \\
&\hspace{-0.5cm}+\frac{M\binom{X_1,X_3}{Y_1+Y_2,Y_3}-M\binom{X_2,X_3}{Y_1+Y_2,Y_3}}{X_1-X_2}+\frac{M\binom{X_2,X_1}{Y_2, Y_1+Y_3}-M\binom{X_2,X_3}{Y_2,Y_1+Y_3}}{X_1-X_3}.
\end{align*}
\end{Example}

\begin{Definition} \label{def:alternal_il}
\begin{enumerate}
\item A bimould $M=(M_d)_{d\geq0} \in \prod_{d\geq0} R[X_1,Y_1,\ldots,X_d,Y_d]$ is \emph{alternal} if there is a $\Q$-linear map $\varphi_\shuffle:\Q\langle Y\rangle \to R$ with
\[
\varphi_\shuffle(u\shuffle v)=0 \quad \text{ for all } u,v \in Y^*\backslash\{1\},
\]
such that $M$ is $(\varphi_\shuffle,\rho_Y)$-symmetric.
\item A bimould $M=(M_d)_{d\geq0} \in \prod_{d\geq0} R[X_1,Y_1,\ldots,X_d,Y_d]$ is \emph{alternil} if there is a $\Q$-linear map $\varphi_\ast:\Q\langle Y\rangle \to R$ with
\[
\varphi_\ast(u\ast v)=0 \quad \text{ for all } u,v \in Y^*\backslash\{1\},
\]
such that $M$ is $(\varphi_\ast,\rho_Y)$-symmetric.
\end{enumerate}
\end{Definition}

\begin{Example}
Alternility in depth $2$ is equivalent to 
\begin{align*}
0 =&\ M\binom{X_1,X_2}{Y_1,Y_2}+M\binom{X_2,X_1}{Y_2,Y_1}+\frac{M\binom{X_1}{Y_1+Y_2}-M\binom{X_2}{Y_1+Y_2}}{X_1-X_2}.
\end{align*}
\end{Example}

\makeinvisible{
\section{A general approach to moulds and their symmetries}

Let $A$ be an alphabet and denote by $A^*$ its words. A mould $M$ is a map from $A^*$ to a 
ring $R$ \cite{Sauzin_intro}. Observe that there is a bijection between moulds and non-commutative power series
\[
\{ M: A^* \to R  \,\}   \cong    R \langle \langle A \rangle \rangle,
\] 
since the coefficients of   $ \sum_{a \in A^*}  ( M |a) \, a \in R \langle \langle A \rangle \rangle$ 
determine a mould $M$ uniquely.

\begin{Example} 
Let $A=\{a_1,a_2, ...\}$, 
$R=k[[ u_1, u_2, u_3, ...]]$ and 
\[
M=(M_0, M_1, M_2, ...  )
\]  
with $M_0 \in k$ and $M_l(u_1,u_2,..., u_d)\in k[[u_1,u_2, ...,u_d]]$. Then we view $M$ as  mould via 
\begin{align*}
M:  \,A^* &\to R \\ 
a_{i_1}a_{i_2}...a_{i_d} & \mapsto M_l ( u_{i_1},u_{i_2},...,u_{i_d}).
\end{align*} 
\end{Example}

By abuse of notation we just write $M( u_{i_1},u_{i_2},...,u_{i_d})$ instead of 
$ (M  |{a_{i_1}a_{i_2}...a_{i_d}})$. 
Thus, the example explains the meaning of Ecalle's definition: 
\begin{center}
\emph{A mould is a collection of functions depending on a variable number of variables.}
\end{center}
 
\begin{Definition}
A mould  
\[
M: A^* \to k[u_1,u_2,u_3,...]
\]
is called a polynomial mould. 
\end{Definition}

A particular example is the sequence $\{M_d\}_{d \in \N}$ with 
\[
M_d= \begin{cases}  M_l(u_1,...,u_l)  \in k[u_1, u_2, ..., u_l] \quad  &\mbox{ if } d=l \\
\qquad 0  &\mbox{ else}. 
\end{cases}
\] 
Then, we get a polynomial mould 
$ M: A^* \to k[u_1,u_2,u_3,...] $ via
\[
(M | a ) = \begin{cases}  M_d(  u_{i_1},u_{i_2},...,u_{i_d}) & \quad \mbox{if} \quad\underline{a}= a_{i_1}a_{i_2}...a_{i_d} \\
0  & \quad \mbox{ else}.
\end{cases}
\]
We say $M$ is concentrated in depth $l$. By abuse of notion we even say $M_l$ is a mould.
       
Most properties assigned to moulds correspond to functional equations. 
 
Given a mould $M=(M_1(u_1) , M_2(u_1,u_2) , M_3(u_1,u_2,u_3) ,...)$
we use  the notation\footnote{\color{cyan}{toldo:  clarify the ambiguities with this notation} {\color{orange} I think the main problem was that the operator needs to be extended $k[(v_i-v_j)^{-1}]$-linearly for alternility/symmetrility. As this space might be contained in the values of the moulds, the notation would become non unique}}
\begin{align}\label{eq:fcns_marriage}
M_1(u_{j_1}) \dplus M_{d-1}( u_{j_2}, ..., u_{j_d}) = M_{r}(u_{j_1}, u_{j_2}, ..., u_{j_d}).
\end{align}
We extend this notation by $M_1(u_{j_1}) \dplus M_0(\emptyset) = M_1(u_{j_1}) $. This allows us to define recursively a set of equations with the initial condition $M_0(\emptyset) \shuffle M =M \shuffle M_0(\emptyset)=M$  
\begin{align*}
M_d( &u_1, ..., u_d) \shuffle M_l(u_{d+1}, ..., u_{d+sl})=\\ 
&  M_1(u_1) \dplus \big( M_{d-1}( u_2, ..., u_d) \shuffle M_l(u_{d+1}, ..., u_{d+l}) \big) \\
&+ M_1(u_{d+1})\dplus \big( M_d( u_1, ..., u_d) \shuffle  M_{l-1}(u_{d+2}, ..., u_{d+l})\big). 
 \end{align*}
For example  we get
\[
M_1(u_1)\shuffle M_2(u_2,u_3)  = M_3(u_1,u_2,u_3) +  M_3(u_2,u_1,u_3) +  M_3(u_2,u_3,u_1).
\]  
  
\begin{Definition} 
We say a mould $M=(M_d)_{\d \geq0}$  
is alternal, if for all $d,l \ge 0$  we have
\[
M_d( u_1, ..., u_d) \shuffle M_l(u_{d+1}, ..., u_{d+l})=0 .
\]    
\end{Definition}      

\begin{Example}
The mould $M: A^* \to  k(u_1,u_2,u_3,...)$ 
given for $a=a_{i_1}a_{i_2}...a_{i_d}$ by 
\[
(M|a)= \begin{cases} 0 &
a = \emptyset \\
0 &  a \mbox{ has a repeated letter} \\   \frac{1}{u_{i_2}-u_{i_1}} \frac{1}{u_{i_3}-u_{i_2}}  \cdots \frac{1}{u_{i_d}-u_{i_{d-1}} }  & \mbox{else}
\end{cases}
\]
is alternal. 
For example we have  for $d=2,3,4$ 
\begin{align*}
0&= M_2(u_1,u_2) + M_2(u_2,u_1)\,,\\
0&= M_3(u_1,u_2,u_3) +M_3(u_2,u_1,u_3) +M_3(u_2,u_3,u_1) \,, \\
0&=M_4(u_1,u_2,u_3,u_4) +M_4(u_2,u_1,u_3,u_4) +M_4(u_2,u_3,u_1,u_4)+M_4(u_2,u_3,u_4,u_1)  \,, \\
0&=M_4(u_1,u_2,u_3,u_4) +M_4(u_1,u_3,u_2,u_4) +M_4(u_1,u_3,u_4,u_2)+ M_4(u_3,u_1,u_2,u_4)\\
&\,\, +M_4(u_3,u_1,u_4,u_2) +M_4(u_3,u_4,u_1,u_2) \, .
\end{align*}
\end{Example}
           
\begin{Definition}
A bimould $M=(M_0, M_1, M_2, ...  )$ in the  pairs of variables $w_i=\mk{u_i}{v_i}$ is a mould 
\[
M: A^* \to  R=k[[u_1,v_1, u_2, v_2, u_3, v_3, ... ]],
\]
such that  $M_0 \in k$ and  $M_d \in k[[ u_1,v_1, u_2, v_2,...,u_d,v_d]] $ for all $d\ge 1$.
  \end{Definition}      
We will use the notation  
 \begin{align*}
(M | {a_{i_1}a_{i_2}...a_{i_l}})=  M ({w_{i_1},w_{i_2},...,w_{i_l}}) = M \scalebox{1.2}{\mbox{$\mk{ u_{i_1},   ...,  u_{i_l} }{v_{i_1}, ...,    v_{i_l}}$ }}   =   M_d ( u_{i_1}, v_{i_1}, ...,   u_{i_l},v_{i_l}).
\end{align*} 
 
There are symmetries of moulds $M: A^* \to R$ induced by endomorphisms of $R$.  The involution \emph{swap} 
is of particular interest for us.
An alternal bimould $M$ is alternal 
with respect to to both set of variables simultaneously.
 
Any mould $M$  as before becomes a bimould by setting 
$M(w_{i_1},w_{i_2},\ldots,w_{i_d})= \binom{u_{i_1},u_{i_2},\ldots,u_{i_d}}{v_{i_1},v_{i_2},\ldots,v_{i_d}}= M(u_{i_1},u_{i_2},\ldots,u_{i_l}).$

A mould $M$ is bi-alternal if $M$ and 
$swap(M)$ are alternal. 

Given a bimould $M=(M_0, M_1(w_1) , M_2(w_1,w_2) , M_3(w_1,w_2,w_3) ,...)$
we use  the notation \eqref{eq:fcns_marriage}
to define recursively a set of equations with the initial condition $1*M =M*1=M$
\begin{align*}
M_d( &w_1, ..., w_d) * M_l(w_{d+1}, ..., w_{d+l})=\\ 
&  M_1(w_1) \dplus \big( M_{d-1}( w_2, ..., w_d) * M_l(w_{d+1}, ..., w_{d+l}) \big) \\
&+ M_1(w_{d+1})\dplus \big(M_d( w_1, ..., w_d) * M_{l-1}(w_{d+2}, ..., w_{d+l})\big) \\ 
&+ \frac{  \big( M_1 \scalebox{1.1}{\mbox{$\mk{ u_1+u_{d+1}}{v_1} $ }}
-M_1 \scalebox{1.1}{\mbox{$\mk{ u_1+u_{d+1}}{v_{d+1}} $ }} \big) \dplus \big( M_{d-1}( w_2, ..., w_d) * M_{l-1}(w_{d+2}, ..., w_{d+l}) \big)}{v_1-v_{d+1}}\,.
\end{align*}
For example, we have
\[
M_1(w_1)*M_1(w_2) = M_2 (w_1,w_2) +M_2(w_2,w_1) + \frac{M_1 \scalebox{1.1}{\mbox{$\mk{ u_1+u_2}{v_1} $ }}
-M_1 \scalebox{1.1}{\mbox{$\mk{ u_1+u_2}{v_2} $ }} }{v_1-v_2}
\]
and
\begin{align*}
  &  M_1(w_1)*M_2(w_2,w_3)\\
  &= 
    M_1(w_1) \dplus ( 1 * M_2(w_2,w_3)) 
    +M_1(w_2 ) \dplus ( M_1(w_1) * M_1(w_3))\\ 
   & + 
  \frac{ \big( M_1 \scalebox{1.1}{\mbox{$\mk{ u_1+u_2}{v_1} $ }}
-M_1 \scalebox{1.1}{\mbox{$\mk{ u_1+u_2}{v_2} $ }}\big)  \dplus (1 * M_1(w_3))}{v_1-v_2} \\
&= M_3( w_1, w_2, w_3) + M_3( w_2, w_1, w_3) +M_3( w_2, w_3, w_1) \\
&+ \frac{ M_2 \scalebox{1.1}{\mbox{$\mk{u_2, u_1+u_3 }{v_2, v_1} $ }}
-M_2 \scalebox{1.1}{\mbox{$\mk{ u_2, u_1+u_3}{v_2,v_3  }$ }}}{v_1-v_3}
+ \frac{ M_2 \scalebox{1.1}{\mbox{$\mk{u_1+u_2,  u_3 }{v_1, v_3} $ }}
-M_2 \scalebox{1.1}{\mbox{$\mk{ u_1+u_2, u_3}{v_2,v_3  }$ }}}{v_1-v_2}
\end{align*}

\begin{Definition} 
We say a bimould $M=(M_d)_{d\geq0}$  
is alternil, if for all $d,l \ge 1$  we have  
\[
M_d( w_1,\ldots, w_d) * M_l(w_{d+1},\ldots, w_{d+l})=0 .
\]    
\end{Definition}      

\begin{Example} The bimoulds $M$ and $N$ given by
\begin{align*}
M&= (u_1^2+v_1^2,\, v_1-2 v_2-u_1+u_2,\textstyle{ \frac{1}{3}}, 0,\ldots )\\
N&= (u_1v_1, -2 v_1+2 v_2-2 u_2, 0,\ldots)
\end{align*}
are alternil and swap invariant. 
\end{Example}

\begin{Lemma}
Let $ M=(0,...,0, M_d, M_{d+1}, ...)$ be an alternil bimould with $M_d \neq 0$, then $M_d$ is an alternal bimould.  
\end{Lemma} 
}
\subsection{The Bachmann-van Ittersum derivation} \label{app:BvI-stuff}

We first introduce the maps
\begin{align*}
\varphi_i^+ \binom{u_1,\ldots, u_d}{v_1,\ldots, v_d} &=    \binom{u_1,\ldots,u_{i-1}, u_i+u_{i+1}, u_{i+2},\ldots, u_d}{v_1,\ldots,v_{i-1}, v_{i+1},v_{i+2}, \ldots,  v_d} & \mbox{ if } 1 \le i < d,\\
\varphi_i^+ \binom{u_1,\ldots, u_d}{v_1,\ldots, v_d} &=
\binom{u_1,\ldots, u_{d-1}}{v_1,\ldots, v_{d-1}} & \mbox{ if } i = d,
\intertext{and} 
\varphi_i^- \binom{u_1,\ldots, u_d}{v_1,\ldots, v_d} &=    \binom{u_1,\ldots,u_{i-2}, u_{i-1}+u_i,u_{i+1}, \ldots, u_d}{v_1,\ldots,v_{i-2}, v_{i-1},  v_{i+1}, \ldots,  v_d}  &\mbox{ if } 2\le i \le d.
\end{align*}
We set $\varphi_1^+ \binom{u_1}{v_1}=\emptyset$.

\begin{Definition}
Let $\delta_{BvI}$ be the operator on bimoulds given by  
\begin{align*}
 \delta_{BvI}(A)&\binom{u_1,\ldots,u_d}{v_1,\ldots,v_d}=  
 \delta(A)\binom{u_1,\ldots,u_d}{v_1,\ldots,v_d}
 \\ 
 & -  \sum_{i= 1}^d  
\frac{ v_i - v_{i+1} + u_i   }{2}  
A \left(\varphi_i^+ \binom{u_1,\ldots, u_d}{v_1,\ldots, v_d} \right) 
+\frac{1}{4}  A \left( (\varphi_i^+)^2 \binom{u_1,\ldots, u_{d}}{v_1,\ldots, v_{d}} \right)
\\
&-\sum_{i= 2}^d \frac{v_{i-1} - v_{i} + u_i }{2}
A \left(\varphi_i^- \binom{u_1,\ldots, u_{d} }{v_1,\ldots, v_{d} } \right)
- \frac{1}{4}  A \left( (\varphi_i^-)^2 \binom{u_1,\ldots, u_{d}}{v_1,\ldots, v_{d}} \right),
\end{align*}
where we use the convention $v_{d+1}=0$ and $\delta$ is given in Definition \ref{def:ari_derivation}. 
    
\end{Definition}

\begin{Example}\label{exa:bvi_one}
Let $\mathbf{1}_{mu}$ be the neutral element of $\big(\operatorname{GBARI}, mu\big)$, i.e. $\mathbf{1}_{mu}$ is the bimould which equals $1$ in depth zero and vanishes for all other depths.
Then, we have 
\[
\delta_{BvI}(\mathbf{1}_{mu}) = (  -\frac{u_1 +v_1}{2}, \frac{1}{4}, 0, \ldots ) \in \BARI^{\operatorname{pol}}.
\]  
\end{Example}

The map $\delta_{BvI}$ was studied by Bachmann-van Ittersum in the context of swap invariant derivations on quasi-shuffle algebras.
\begin{Proposition}\cite[Proposition 4.11]{BvI} (i) The map $\delta_{BvI}$ commutes with $swap$.

(ii) The coefficient map of $\delta_{BvI}(\rho_Y(\mathcal{W}))$ is a derivation of weight $-2$ on the quasi-shuffle algebra $(\Q \langle Y\rangle,*)$.
\end{Proposition}

The definition of the generating series $\rho_Y(\mathcal{W})$ is given in \eqref{eq:def_rhoY}.

\section{The space of bi-period polynomials} \label{sec:biperiod_pol}

\subsection{Period polynomials}\label{app:period} Let $V \subset \Q[u_1,u_2]$ be the subspace of homogeneous polynomials and for $k\geq2$ even set 
\[
V_k = \{ f \in V \mid \deg(f)=k-2\}. 
\]
\begin{Definition}
The space of period polynomials for $\operatorname{Sl}_2(\mathbb{Z})$ is given by  
\begin{align*}
W_k= \left\{ f \in V_k \, \middle| \,  \begin{matrix} f(u_1,u_2)+f(u_2,-u_1)=0 \\
f(u_1,u_2)+f(-u_1-u_2,u_1)+f(-u_2,u_1+u_2)=0 \end{matrix} \right\}.
\end{align*}
\end{Definition}
It is a well-known fact that $W_k$ is stable under the change of variables $(u_1,u_2) \mapsto (-u_1,u_2)$ and in fact 
it decomposes into the eigenspaces $W_k^+$ and $W_k^-$ corresponding to the eigenvalues $\pm 1$. 
The dimensions of this subspaces may be determined through the Eichler-Shimura isomorphisms
\begin{align*}
M_k( \operatorname{Sl}_2(\mathbb{Z}))  &\overset{\sim}{\longrightarrow} W_k^+\\
S_k( \operatorname{Sl}_2(\mathbb{Z})) &\overset{\sim}{\longrightarrow} W_k^-,
\end{align*}
where the spaces on the left are modular forms resp. the cusp forms for the  the modular group $\operatorname{Sl}_2(\mathbb{Z})$. In particular, we deduce
\begin{align*}
\sum_{k\geq 0} \dim W_k^+ x^k &=\sum_{k\geq0} \dim M_k(\operatorname{Sl}_2(\mathbb{Z}))\, x^k =\frac{1}{(1-x^4)(1-x^6)}, \\
\sum_{k\geq 12} \dim W_k^- x^k &=\sum_{k\geq 12} \dim S_k(\operatorname{Sl}_2(\mathbb{Z})) \,x^k =\frac{x^{12}}{(1-x^4)(1-x^6)}.
\end{align*}
\begin{Example} We have
\begin{align*}
    W_{12}^+ = \operatorname{span}_\Q\{ p_{12},\, p_\Delta^+\},\quad
    W_{12}^- = \operatorname{span}_\Q \{p_\Delta^-\},
\end{align*}
where
\begin{align*}
p_{12}(u_1,u_2) &=u_1^{10} - u_2^{10} \\
p_\Delta^+(u_1,u_2) &= (u_1^8u_2^2-u_1^2u_2^8) - 3(u_1^6u_2^4 - u_1^4u_2^6),\\
p_\Delta^-(u_1,u_2) &=  4(u_1^9u_2+u_1u_2^9) - 25(u_1^7u_2^3+u_1^3u_2^7) + 42u_1^5u_2^5.
\end{align*}
Via the Eichler-Shimura isomorpisms these polynomials correspond to the Eisenstein series $G_{12} \in M_{12}(\operatorname{Sl}_2(\mathbb{Z}))$ and the cusp form $\Delta \in S_{12}( \operatorname{Sl}_2(\mathbb{Z}))$. 
\end{Example} 
In general, we refer to the polynomial $p_k(u_1,u_2)=  u_1^{k-2} - u_2^{k-2}$ as the Eisenstein polynomial. We define
\[
\widetilde{W}_k^+=W_k^+/(p_k).
\]

\subsection{Bi-period polynomials} \label{app:bi_period}

We consider the quadratic space $\Q^4$ with
the quadratic form 
\[
q(u_1,u_2,v_1,v_2)= u_1 v_1 + u_2 v_2. 
\]
The  signature of $q$ is $(2,2)$, therefore 
the group of isometries is isomorphic to the orthogonal group 
$\operatorname{O}(2,2)(\Q)$. 

We can view 
$\operatorname{Sl}_2(\Q) \times\operatorname{Sl}_2(\Q) $ 
as a subgroup of $\operatorname{O}(2,2)(\Q)$ as follows. Write
\[
X=\begin{pmatrix} u_1 & u_2\\ -v_2 & v_1  \end{pmatrix},
\]
then we have
\begin{align} \label{eq:q_det}
q(u_1,u_2,v_1,v_2)=\det(X).
\end{align}
In particular, if $(A,B)\in\operatorname{Sl}_2(\Q)\times\operatorname{Sl}_2(\Q) $ then $\det(AXB^t)=\det(X)$ and hence by \eqref{eq:q_det} the map $X\mapsto AXB^t$ is an isometry of the quadratic form $q$. 

More generally, we view any homogeneous polynomial $P \in \Q[u_1,u_2, v_1, v_2] $ as a function on $\operatorname{Mat}(2\times 2,\Q)$ via 
\begin{align*}
P(A) := P(a_{11},a_{12},a_{22},-a_{21}), \qquad A=\begin{pmatrix} a_{11} & a_{12} \\ a_{21} & a_{22} \end{pmatrix} \in \operatorname{Mat}(2\times 2,\Q).
\end{align*}
Then, $\operatorname{Sl}_2(\Q)$ acts on the space of homogeneous polynomials  $\mathcal{V} \subset  \Q[u_1,u_2, v_1, v_2]$
by matrix multiplication. For $k\geq2$ even, set 
\[
\mathcal{V}_k = \{ P\in \mathcal{V} \mid \deg(P)=k-2\}.
\]
Let $\Delta$ be the Laplacian on $\mathcal{V}_k$, which is given by
\[
\Delta=\partial_{u_1}\partial_{v_1}+\partial_{u_2}\partial_{v_2}.
\]
For $P\in\mathcal{V}_k$, we use the notation 
\[
\partial(P)= P(\partial_{u_1},\partial_{u_2},\partial_{v_1},\partial_{v_2}).
\]
In particular, we have $\partial(q)=\Delta$. Finally, we define a symmetric and positive definite pairing on $\mathcal{V}_k$  by
\begin{align} \label{eq:pairing_Vk}
\langle P_1,P_2 \rangle = \left( \partial(P_1) P_2 \right) (0).
\end{align}
There are the following standard facts for spherical harmonics, see for example  \cite[Section 5.6.4]{goodman_wallach}.

\begin{Lemma}\label{lem:decomp}
(i)  There is a decomposition 
\[
\mathcal{V}_k  = \ker(\Delta) \oplus (u_1v_1+u_2v_2) \mathcal{V}_{k-2}. 
\]
(ii)  The map
\begin{align*}
\rho: \mathcal{V}_k &\to V_k \otimes V_k, \\
P(u_1,u_2,v_1,v_2) &\mapsto \rho(P)(a,b;c,d) =
P(ac,ad ,bd,- bc) 
\end{align*}
restricts to an isomorphism $\ker \Delta \xrightarrow{\sim}V_k \otimes V_k$.

(iii) The inverse map $V_k \otimes V_k \to \ker \Delta$ to the map $\rho$ from (ii) is given by 
\[
f( a,b; c,d) \mapsto  \frac{1}{(k-2)!^2} \left\langle f( a,b; c,d) ,  (u_1 a c +u_2 a d + v_1  b d - v_2 b c)^{k-2} \, \right\rangle.
\] 
\end{Lemma}

\begin{proof}
By definition of the pairing, we have
\begin{align*}
\langle  qP_1 , P_2 \rangle =  \langle  P_1 , \partial(q) P_2 \rangle  = \langle P_1 , \Delta P_2 \rangle.
\end{align*}
This means, the operator $\Delta$ and multiplication by $q$ are adjoint operators with respect to the pairing $\langle-,-\rangle$. As the pairing is positive definite, we deduce that 
\[\ker(\Delta)^\perp=q\mathcal{V}_{k-2},\]
where the right hand side is exactly the image of multiplication by $q$. This implies the claim.

Next, we prove (ii). If $P\in \mathcal{V}_k$ is homogeneous of degree $k-2$, then $\rho(P)$ is homogeneous of degree $k-2$ in $(a,b)$
and of degree $k-2$ in $(c,d)$. Hence, we get $\rho(P)\in V_k\otimes V_k$ and $\rho$ is well-defined. We have 
\[
\rho(q)=q(ac,ad,bd,-bc)=0,
\]
and hence $\rho(qP)=0$ for all $P\in\mathcal{V}_k$. Using the harmonic decomposition from (i)
\[
\mathcal V_k=\ker(\Delta)\oplus q\,\mathcal{V}_{k-2}, 
\]
it follows that $\rho$ restricts to
\[
\rho:\ker(\Delta)\rightarrow V_k\otimes V_k.
\]
We now prove the injectivity of $\rho$ on $\ker(\Delta)$. Suppose $P \in\ker(\Delta)$ and $\rho(P)=0$. Then $P(u_1,u_2,v_1,v_2)$ vanishes on all quadruples  
$
(u_1,u_2,v_1,v_2)
=
(ac,ad,bd,-bc)$, 
i.e. on all rank–one matrices
\begin{align}\label{eq:rho_rank1}
\begin{pmatrix}
u_1 & u_2 \\
-v_2 & v_1
\end{pmatrix}
=
\begin{pmatrix}
a\\ b
\end{pmatrix}
\begin{pmatrix}
c & d
\end{pmatrix}.
\end{align}

The Zariski closure of rank–one matrices is precisely the quadratic
hypersurface defined by the vanishing of the determinant, i.e. by $q=0$.
Hence $P$ vanishes on the variety $\{q=0\}$.
By Hilbert's Nullstellensatz (applied over $\overline{\Q}$),
this implies $P$ lies in the radical of $q\mathcal{V}$. Since $q\mathcal{V}$ is prime and $P$ is homogeneous, we get
\[
P\in q\mathcal{V}_{k-2}.
\]
But (i) implies $\ker(\Delta) \cap q\mathcal{V}_{k-2}=\{0\}$, and hence $P=0$. Thus $\rho$ is injective on $\ker(\Delta)$.

It remains show that the dimensions of $\ker(\Delta)$ and $V_k\otimes V_k$ agree, since this completes the proof of (ii). We have
\[
\dim \mathcal{V}_k=\binom{k+1}{3}.
\]
From the decomposition in (i), we deduce
\begin{align} \label{eq:dim_ker(Delta)}
\dim \ker(\Delta)=\dim \mathcal{V}_k - \dim q\mathcal{V}_{k-2}= \binom{k+1}{3}-\binom{k-1}{3}=(k-1)^2.
\end{align}
On the other hand, we have $\dim(V_k)=k-1$ and hence
\begin{align} \label{eq:dim_Vk_tensor_Vk}
\dim(V_k\otimes V_k)=(k-1)^2.
\end{align}
Comparing the dimensions in \eqref{eq:dim_ker(Delta)} and \eqref{eq:dim_Vk_tensor_Vk} yields the claim.

(iii) can be checked with a straight-forward calculation.
\end{proof}

\begin{Definition}\label{app:bi_period_def} 
A  bi-period polynomial is a homogeneous polynomial $P \in \mathcal{V}$ which satisfies
\begin{align*}
P(X)+P(XU)+P(XU^2) &=0,\\
P(X)+P(XS) &=0,\\  
P(X^t)-P(X) &=0.
\end{align*}
where  
$U= \begin{pmatrix} 0 & 1 \\ -1 & 1 \end{pmatrix}$ and
$S= \begin{pmatrix} 0 & 1 \\ -1 & 0 \end{pmatrix}$. We denote the space of bi-period polynomials of degree $k-2$ by $\mathcal{W}_k$.

If in addition 
\[
P(\epsilon X\epsilon)=P(X) ,\qquad \epsilon=  \begin{pmatrix} -1 & 0 \\ 0 & 1 \end{pmatrix},
\]
then it is called even. The space of 
even bi-period polynomials of degree $k-2$ is denoted by $\mathcal{W}^{\epsilon}_k$.

Moreover, if $P\in \mathcal{V}$ satisfies
\[
\Delta (P)=0,
\]
then it is called primitive. The space of 
primitive bi-period polynomials of degree $k-2$ is denoted by  
$\mathcal{P}_k$ and we set
$\mathcal{P}^{\epsilon}_k = \mathcal{P}_k  \cap \mathcal{W}^{\epsilon}_k$.
\end{Definition}

\begin{Lemma}\label{lem:primitive_subspace}
We have a decomposition
\[
\mathcal{W}^{\epsilon}_k = \mathcal{P}^{\epsilon}_k  \oplus 
(u_1v_1+u_2v_2) \mathcal{W}^{\epsilon}_{k-2}
\]   
\end{Lemma}
\begin{proof} For $P\in \mathcal{W}_{k-2}^\epsilon$, we have $(u_1v_1+u_2v_2)P\in \mathcal{W}_k^\epsilon$. Hence, the decomposition follows from Lemma \ref{lem:decomp}.
\end{proof} 

For any polynomials $f,g\in V_k$, we define
\[
P_{f,g} = \frac{1}{(k-2)!^2}\left\langle f(a,b)g(c,d)+g(a,b)f(c,d),(u_1bd-u_ 2bc+v_1ac+v_2ad)^{k-2}\right\rangle\in \mathcal{V}_k.
\]
By Lemma \ref{lem:decomp}, the elements $P_{f,g}$ are primitive polynomials in $\mathcal{V}_k$.  

\begin{Theorem}\label{app:bi_period_thm}
(i) A basis for   $\mathcal{P}_k$ is given by the elements $P_{f,g}$, where $f,g$ are elements of a basis of $W_k$.

(ii) A basis for   $\mathcal{P}^{\epsilon}_k$ is given by the elements $P_{f,g}$, where $f,g$ are elements of a basis of $W_k^+$  or the same for $W_k^-$.
\end{Theorem}

\begin{proof} 
(i) By \eqref{eq:rho_rank1}, the map  $ \rho:\, \mathcal{V}_k \to V_k \otimes V_k$ from Lemma \ref{lem:decomp} is given by 
\[
P(X) \mapsto  
P\left( \left(\begin{smallmatrix} ac & ad \\ bc & bd  \end{smallmatrix} \right)\right) =
 P\left( \left(\begin{smallmatrix} a \\ b \end{smallmatrix} \right)
\!\! \begin{smallmatrix} \phantom{,}& \\(c & d )\end{smallmatrix}   \right).
\]
Now the first and second   defining condition for $\mathcal{W}_k$ given in Definition \ref{app:bi_period_def}  imply that the right tensor factors in $\rho(\mathcal{W}_k)$ belong to $W_k$. Then the
third defining condition implies the same conditions for the left tensor factor plus the symmetry. Therefore 
\[
\rho(\mathcal{P}_k ) =\left\langle  f(a,b)\otimes g(c,d) +  g(a,b)\otimes f(c,d)
\, \big|\, f,g \in W_k  \,\right\rangle_\Q
\subset  W_k \otimes W_k .
\]
Using the inverse map given in Lemma \ref{lem:decomp} the first claim follows.

(ii) The additional bi-even condition says that $P_{f,g}(a,b;c,d)=P_{f,g}(-a,b;-c,d)$. We have $W_k =  W_k^+ \oplus W_k^-$ and if either both $f,g\in W_k^+$ or both $f,g\in W_k^-$, then
\[
f(a,b)g(c,d)+g(a,b)f(c,d)=f(-a,b)g(-c,d)+g(-a,b)f(-c,d),
\]
and hence the claim follows.
\end{proof}

\begin{Corollary} \label{cor:dim_Pk}
We have
\begin{align*} 
(i) & \quad \sum_{k>0} \dim \mathcal{P}_k  x^k =    \sum_{k>0} \ \dim M_k(\operatorname{Sl}_2(\mathbb{Z})) \big(\dim M_k(\operatorname{Sl}_2(\mathbb{Z})) + \dim S_k(\operatorname{Sl}_2(\mathbb{Z}))\big) \, x^k, \\
(ii) & \quad \sum_{k>0} \dim \mathcal{P}^\epsilon_k  x^k = \sum_{k>0} \dim ( M_k(\operatorname{Sl}_2(\mathbb{Z})))^2 \, x^k.    
\end{align*}   
\end{Corollary}
\begin{proof}
(i) From Appendix \ref{app:period}, we know that
\[
\dim W_k= \dim M_k(\operatorname{Sl}_2(\mathbb{Z}))+ \dim S_k(\operatorname{Sl}_2(\mathbb{Z})).
\]
Thus, by Theorem \ref{app:bi_period_thm} the number of linear independent $P_{f,g}$ in weight $k$ is given by 
\begin{align*}
\binom{\dim W_k+1}{2}&=\frac{(\dim M_k(\operatorname{Sl}_2(\mathbb{Z})) +\dim S_k(\operatorname{Sl}_2(\mathbb{Z}))+1) (\dim M_k(\operatorname{Sl}_2(\mathbb{Z})) +\dim S_k(\operatorname{Sl}_2(\mathbb{Z})))}{2} \\
&= \dim M_k(\operatorname{Sl}_2(\mathbb{Z}))(\dim M_k(\operatorname{Sl}_2(\mathbb{Z})) + \dim S_k(\operatorname{Sl}_2(\mathbb{Z}))). 
\end{align*}
In the last equality we made use of $\dim M_k(\operatorname{Sl}_2(\mathbb{Z}))=\dim S_k(\operatorname{Sl}_2(\mathbb{Z}))+1$.

(ii) We have by Appendix \ref{app:period} that
\[
\dim W_k^+=\dim M_k(\operatorname{Sl}_2(\mathbb{Z})),\qquad \dim W_k^-=\dim S_k(\operatorname{Sl}_2(\mathbb{Z})).
\]
Thus, by Theorem \ref{app:bi_period_thm} the number of linear independent $P_{f,g}$ in $\mathcal{P}_k^\epsilon$ is
\begin{align*}
&\binom{\dim W_k^++1}{2}+\binom{\dim W_k^-+1}{2}\\
&\hspace{0.7cm}=\frac{(\dim M_k(\operatorname{Sl}_2(\mathbb{Z}))+1)\dim M_k(\operatorname{Sl}_2(\mathbb{Z}))}{2}+\frac{(\dim S_k(\operatorname{Sl}_2(\mathbb{Z}))+1)\dim S_k(\operatorname{Sl}_2(\mathbb{Z}))}{2}\\
&\hspace{0.7cm}=\dim M_k(\operatorname{Sl}_2(\mathbb{Z}))^2.  \qedhere
\end{align*}
\end{proof}

\begin{Corollary}
We have
\[ 
 \sum_{k>0} \dim  \mathcal{W}^{\epsilon}_k \, x^k =    \mathsf{D}(x) \, \sum_{k>0} \dim (M_k(\operatorname{Sl}_2(\mathbb{Z})))^2  \, x^k.
\]   
where $\mathsf{D}(x) = \frac{1}{1-x^2}$. 
\end{Corollary}

\begin{proof} By Corollary \ref{cor:dim_Pk}, we have
\[
\sum_{k>0} \dim \mathcal{P}^\epsilon_k  x^k = \sum_{k>0} \dim ( M_k(\operatorname{Sl}_2(\mathbb{Z}))^2 \, x^k. 
\]
Using the decomposition
\[
\mathcal{W}_k^\epsilon=\mathcal{P}_k^\epsilon\oplus q\mathcal{W}_{k-2}^\epsilon
\]
from Lemma \ref{lem:primitive_subspace}, we get the desired dimensions, as multiplying with the quadratic form $q$ corresponds to the term $\mathsf{D}(x)$.
\end{proof}

%


\section{More on the uri bracket}\label{app:uri_swap_alternative}

\subsection{Proof of Theorem \ref{thm:ganit_with_urit}} \label{subsec:proof_urit}

We want to prove that for bimoulds $A,B\in \BARI^{\operatorname{pol}}$, we have
\[
\urit(B)(A)=\ganit_{pic}\Bigl(\arit\bigl(\ganit_{poc}(B))(\ganit_{poc}(A)\bigr)\Bigr).
\]
\noindent {\bf Step 1.} Let $P_{i,j}$ denote the bimould in $\BARI^{\operatorname{pol}}$ concentrated
in depth $1$ given by
$$P_{i,j}\begin{pmatrix}u_1\\ v_1\end{pmatrix}=u_1^iv_1^j.$$
We first show that it is enough to prove 
Theorem \ref{thm:ganit_with_urit} for bimoulds $A=P_{i,j}\in \BARI^{\operatorname{pol}}$ and bimoulds $B\in \BARI^{\operatorname{pol}}$ that are products of depth $1$ bimoulds, i.e.
\begin{equation}\label{Bmould}
B=mu(B_1,\ldots,B_d)
\end{equation}
with $B_r=P_{k_r,l_r}$ for $r=1,\ldots,d$.

The key point is that as a conjugation of a derivation by an automorphism,
the operator $ganit_{pic}\circ \arit\bigl(\ganit_{poc}(B)\bigr)$ is a derivation, and by Definition \ref{def:urit}, $\urit(B)$ is also a derivation. By the remark
just before Definition \ref{def:urit}, every bimould $A,B\in \BARI^{\operatorname{pol}}$ is
a linear combination of $mu$-products of polynomial bimoulds concentrated in 
depth $1$, i.e.~of bimoulds $P_{i,j}$, so it is enough to check that the two derivations $\ganit_{pic}\circ arit\bigl(\ganit_{poc}(B)\bigr)$ and $\urit(B)$ are equal on the depth $1$ bimoulds $P_{i,j}$. Since every bimould in $\BARI^{\operatorname{pol}}$ is a linear combination of $mu$-products of these depth $1$ bimoulds and $\arit$ is also linear in its argument, it is enough to assume that
$B$ is a $mu$-product of depth $1$ bimoulds as in \eqref{Bmould}.

\vspace{.2cm}
\noindent {\bf Step 2.} Let us now explicitly compute the bimould
$\urit(B)(A)$ for a depth $1$ bimould $A=P_{i,j}$ and a bimould 
$B$ as in \eqref{Bmould}, directly from the formula in Definition \ref{def:urit}. The lowest depth that occurs is $n=d+1$. Letting
$$U_i=u_1+\cdots+u_i,$$ 
we find  
\begin{align}\label{uritBAlowestdep}
urit(B)(A)\binom{u_1,\ldots,u_{d+1}}{v_1,\ldots,v_{d+1}}&=arit(B)(A)\binom{u_1,\ldots,u_{d+1}}{v_1,\ldots,v_{d+1}} \\
&=U_{d+1}^i\Biggl(v_{d+1}^j
\prod_{a=1}^du_a^{k_a}(v_a-v_{d+1})^{l_a}
-v_1^j \prod_{a=1}^d u_{a+1}^{k_a}(v_{a+1}-v_1)^{l_a}\Biggr). \nonumber
\end{align}
For depths $n>d+1$, sums with products in the denominator appear, but
all such sums can be simplified using the following general formula:
for $m$ variables $x_1,\ldots,x_m$, we have
\begin{equation}\label{Lagrange}
\sum_{s=1}^m \frac{(a-x_s)^l} {\prod_{t\ne s}(x_t-x_s)} = 
\begin{cases}
0,& l<m-1,\\[2mm]
S^{l-m+1}(a-x_1,\dots,a-x_m),& l\ge m-1,
\end{cases},
\end{equation}
where $S^e(x_1,\ldots,x_m)$ denotes the complete homogeneous symmetric 
polynomial of degree $e$. Let $L=l_1+\cdots+l_d$. Let us generalize 
the definitions of the symmetric polynomials as follows: 
for variables $x_1,\ldots,x_m$ and
any one-variable polynomial $F(x)$ of fixed homogeneous degree $e+m-1$,
we define the associated symmetric polynomial $S^e_F(x_1,\ldots,x_m)$ 
as the degree $e$ polynomial
\begin{equation}\label{gensympol}
S^e_F(x_1,\ldots,x_m)=\sum_{s=1}^m \frac{F(x_s)}{\prod_{t\ne s}(x_t-x_s)}.
\end{equation}
Observe that when $m=1$, we have $S_F^e(x_1)=F(x_1)$, and that in the
case $n=d+1$, the degree $L$ polynomials 
$F_1(v_{d+1})$ and $F_2(v_1)$ are exactly the polynomials appearing in 
\eqref{uritBAlowestdep}.  Therefore we can rewrite \eqref{uritBAlowestdep}
as
\begin{equation}\label{uritBAlowestdep2}
urit(B)(A)\binom{u_1,\ldots,u_{d+1}}{v_1,\ldots,v_{d+1}}=U_{d+1}^i\Biggl(v_{d+1}^j
S^L_{F_1}(v_{d+1})\bigl(\prod_{a=1}^du_a^{k_a}\bigr)
-v_1^jS^L_{F_2}(v_1)\bigl(\prod_{a=1}^d u_{a+1}^{k_a}\bigr)\Biggr).
\end{equation}
Then the full expression for
$urit(B)(A)$ in depth $n$ (simplified using \eqref{Lagrange} and
\eqref{gensympol}) is given by
\begin{align}\label{uritBA}
&urit(B)(A)\binom{u_1,\ldots,u_n}{v_1,\ldots,v_n}=\\
&U_n^iv_1^j\Biggl(\Bigl(\prod_{a=1}^d u_{n-d+a-1}^{k_a}\Bigr)S_{F_1}^{L-n+d+1}(v_1,\ldots,v_{n-d-1},v_n)
-\Bigl(\prod_{a=1}^d u_{n-d+a}^{k_a}\Bigr)S_{F_2}^{L-n+d+1}(v_1,\ldots,v_{n-d})\Biggr) \nonumber
\end{align}
where 
\begin{equation*}
\begin{cases}
F_1(x)=\prod_{a=1}^d (v_{n-d+a-1}-x)^{l_a}\\
F_2(x)=\prod_{a=1}^d (v_{n-d+a}-x)^{l_a}.
\end{cases}
\end{equation*}

\vspace{.1cm}
Note in particular that since the symmetric polynomials are of degree
$L-n+d+1$, we must have $L-n+d+1\ge 0$ or
$n\le L+d+1$, so the bimould $urit(B)(A)$ is non-zero only in depths
$d+1\le n\le L+d+1$. Equation \eqref{uritBAlowestdep2} gives its value when $n=d+1$
and \eqref{uritBA} for $d+1<n\le L+d+1$.

\vspace{.2cm}
\noindent {\bf Step 3.}
Now let us compute $ganit_{pic}\Bigl(arit(ganit_{poc}(B))(ganit_{poc}(A))\Bigr)$ with $A=P_{i,j}$ and $B$ as in \eqref{Bmould}.
We start by computing 
$ganit_{poc}(A)$, which has a reasonably simple form.
Let $Q=ganit_{poc}(B)$, $P=ganit_{poc}(A)$ and $R=arit(Q)(P)$. Then
\begin{equation}\label{ganitpocA}
P\binom{u_1,\ldots,u_n}{v_1,\ldots,v_n}=
\frac{U_n^iv_1^j}{(v_1-v_2)\cdots(v_{n-1}-v_n)}.
\end{equation}
For $Q$, since $ganit_{poc}$ is a $mu$-automorphism, we simply
have to compute the $mu$-product of expressions like \eqref{ganitpocA} for
each $B_r$, $r=1,\ldots,d$. The bimould $Q=ganit_{poc}(B)$ is $0$ in depths
$n<d$, and for $n\ge d$ we obtain
\begin{equation*}\label{ganitpocB}
Q\binom{u_1,\ldots,u_n}{v_1,\ldots,v_n}=
\sum_{0=s_0<s_1<\cdots<s_d=n}\ \prod_{a=1}^d\ 
\frac{(u_{s_{a-1}+1}+\cdots+u_{s_a})^{k_a}v_{s_{a-1}+1}^{l_a}}
{\prod_{t=s_{a-1}+1}^{s_a-1} (v_t-v_{t+1})  }.
\end{equation*}

Recall that 
\begin{equation}\label{aritQP}
R(w)=\sum_{{{w=abc}\atop{b,c\ne\emptyset}}} P(a\lceil c)Q(b\rfloor)
-\sum_{{{w=abc}\atop{a,b\ne\emptyset}}} P(a\rceil c)Q(\lfloor b).
\end{equation}
Pick a decomposition
\begin{equation}\label{decomp}
w=(w_1,\ldots,w_n)=\binom{u_1,\ldots,u_n}{v_1,\ldots,v_n}
=(w_1,\ldots,w_p)(w_{p+1},\ldots,w_{p+q})(w_{p+q+1},\ldots,w_n)=abc,
\end{equation}
and note that because $Q$ starts in depth $d$, the $b$-part of the 
decomposition must be of length $\ge d$ in order for the $Q$-terms of
\eqref{aritQP} to be non-zero, so we may assume that $q\ge d$.

Let us compute the terms corresponding to this decomposition in \eqref{aritQP}.
Start with the terms where neither $a$ nor $c$ are empty, so $p\ge 1$ and 
$q\ge d$.
Let $U_n=u_1+\cdots+u_n$ and $U_b=u_{p+1}+\cdots+u_{p+q}$.
Then for the positive term we have
\begin{equation*}\label{genterm1}
P(a\lceil c)=
\frac{U^i v_1^j}
{\displaystyle
\left(\prod_{t=1}^{p-1}(v_t-v_{t+1})\right) (v_p-v_{p+q+1})
\left(\prod_{t=p+q+1}^{n-1}(v_t-v_{t+1})\right) }
\end{equation*}
and
\begin{equation*}\label{genterm2}
Q(b\rfloor)=\sum_{0=s_0<s_1<\cdots<s_d=q} \prod_{a=1}^d
\Biggl(
\frac{(u_{p+s_{a-1}+1}+\cdots+u_{p+s_a})^{k_a}
(v_{p+s_{a-1}+1}-v_{p+q+1})^{l_a}}
{\prod_{t=p+s_{a-1}+1}^{p+s_a-1} (v_t-v_{t+1})}
\Biggr).
\end{equation*}
For the negative term we have
\begin{equation*}\label{genterm3}
P(a\rceil c)
=
\frac{ U_n^i v_1^j }
{ \displaystyle
\left(\prod_{t=1}^{p-1}(v_t-v_{t+1})\right)
(v_p-v_{p+q+1})
\left(\prod_{t=p+q+1}^{n-1}(v_t-v_{t+1})\right) }
\end{equation*}
and
\begin{equation*}\label{genterm4}
Q(\lfloor b)
=
\sum_{0=s_0<s_1<\cdots<s_d=q}
\prod_{a=1}^d
\Biggl(
\frac{
\left(u_{p+s_{a-1}+1}+\cdots+u_{p+s_a}\right)^{k_a}
\left(v_{p+s_{a-1}+1}-v_p\right)^{l_a}
}{
\displaystyle
\prod_{t=p+s_{a-1}+1}^{p+s_a-1}(v_t-v_{t+1})
}
\Biggr).
\end{equation*}
Thus for $p\ge 1$, $q\ge d$ and $p+q<n$, 
setting
$$F_{\mathbf{s}}(x)=\prod_{a=1}^d (v_{p+s_{a-1}+1}-x)^{l_a},$$
for $\mathbf{s}=(0=s_0,s_1,\ldots,s_d=q)$, 
we find that the difference 
$$T^n_{p,q}:=P(a\lceil c)Q(b\rfloor)-P(a\rceil c)Q(\lfloor b)$$
is equal to
\begin{equation}\label{Tpq}
T^n_{p,q} =
U_n^i v_1^j \sum_{0=s_0<s_1<\cdots<s_d=q}
\frac{ \displaystyle \Bigl( \prod_{a=1}^d
(u_{p+s_{a-1}+1}+\cdots+u_{p+s_a})^{k_a} \Bigr)
S_{F_{\mathbf s}}^{L-1}(v_p,v_{p+q+1}) }
{\displaystyle
\Bigl(\prod_{t=1}^{p-1}(v_t-v_{t+1})\Bigr)
\Bigl( \prod_{a=1}^d
\prod_{t=p+s_{a-1}+1}^{p+s_a-1}(v_t-v_{t+1})
\Bigr)
\Bigl(\prod_{t=p+q+1}^{n-1}(v_t-v_{t+1})\Bigr) }.
\end{equation}
In order to compute $R=arit(Q)(P)=arit(ganit_{poc}(B))(ganit_{poc}(A))$ using \eqref{aritQP}, it remains to deal with the
positive terms where $a=\emptyset$, meaning decompositions
$w=abc$ as in \eqref{decomp} with $p=0$, and the negative terms
where $c=\emptyset$, meaning decompositions as in \eqref{decomp}
with $p+q=n$, i.e.~$p=n-q$. By the same method as above, we compute
\begin{equation*}\label{posspecialterm}
T_{0,q} = \frac{U_n^i v_{q+1}^j}{\displaystyle\prod_{t=q+1}^{n-1}(v_t-v_{t+1})}
\ \ \sum_{0=s_0<s_1<\cdots<s_d=q}
\ \prod_{a=1}^d
\Biggl(
\frac{
\left(u_{s_{a-1}+1}+\cdots+u_{s_a}\right)^{k_a}
\left(v_{s_{a-1}+1}-v_{q+1}\right)^{l_a}
}{
\displaystyle
\prod_{t=s_{a-1}+1}^{s_a-1}(v_t-v_{t+1})
}
\Biggr)
\end{equation*}
and
\begin{align*}\label{negspecialterm}
&T_{n-q,q} =
- \frac{U_n^i v_1^j} {\displaystyle\prod_{t=1}^{n-q-1}(v_t-v_{t+1})}
\\
& \hspace{1.8cm}\cdot \sum_{0=s_0<s_1<\cdots<s_d=q}
\ \prod_{a=1}^d
\Biggl(
\frac{
\left( u_{n-q+s_{a-1}+1}+\cdots+u_{n-q+s_a} \right)^{k_a}
\left( v_{n-q+s_{a-1}+1}-v_{n-q} \right)^{l_a} }
{
\displaystyle
\prod_{t=n-q+s_{a-1}+1}^{n-q+s_a-1}
(v_t-v_{t+1})
}
\Biggr),
\end{align*}
both expressions valid for $d\le q\le n-1$. Then we have
\begin{equation}\label{R}
R(w_1,\ldots,w_n)=\sum_{q=d}^{n-1} T^n_{0,q}+\sum_{p=1}^{n-d-1}\sum_{q=d}^{n-p-1} T^n_{p,q}+\sum_{q=d}^{n-1} T^n_{n-q,q}.
\end{equation}
Now we apply $\ganit_{pic}$ to $R$, to obtain the desired derivation
$$ganit_{pic}(R)=ganit_{pic}\big(arit(Q)(P)\big)
=ganit_{pic}\Big(arit(\ganit_{poc}(B))(\ganit_{poc}(A))\Big).$$
We will show that this is equal to \eqref{uritBA}, completing the proof.

We start by writing out the action of $\ganit_{pic}$ with explicit terms
instead of the flexions.  We do this by replacing the decomposition
$b_1c_1\cdots b_sc_s$ in the usual form for $\ganit_{pic}$ with
the tuple of numbers
$$1=i_1<i_2<\cdots<i_m\le n$$
defined by taking $\{i_1,\ldots,i_m\}\subset \{1,\ldots,n\}$ to be
the subset of numbers in $b_1\cup \cdots \cup b_s$. 

Let $i_{m+1}=n+1$, and set
$$\tilde u_r=u_{i_r}+\cdots+u_{i_{r+1}-1},\ \ \tilde v_r=v_{i_r}.$$
Then the $b$-flexion expression occurring in $\ganit_{pic}$ is given by
$$b_1\rceil\cdots b_s\rceil=\begin{pmatrix}\tilde u_1&\cdots&\tilde u_m\\
\tilde v_1&\cdots&\tilde v_m\end{pmatrix}.$$
Writing $\ganit_{pic}(R)$ in this notation, we have
\begin{align*}
\ganit_{pic}(R)\binom{u_1,\ldots,u_n}{v_1,\ldots,v_n}
&=\sum_{m=1}^n \ \sum_{1=i_1<i_2<\cdots<i_m\le n}\ 
R\begin{pmatrix}\tilde u_1&\cdots&\tilde u_m\\ \tilde v_1&\cdots&\tilde v_m
\end{pmatrix}\times\ \prod_{r=1}^m\prod_{t=i_r+1}^{i_{r+1}-1}
\frac{1}{(v_t-v_{i_r})}\notag\\
&=\sum_I\ 
R\begin{pmatrix}\tilde u_1&\cdots&\tilde u_m\\ \tilde v_1&\cdots&\tilde v_m
\end{pmatrix}\times\ \prod_{r=1}^m\prod_{t=i_r+1}^{i_{r+1}-1}
\frac{1}{(v_t-v_{i_r})},
\end{align*}
where the sum over $I$ runs over all subsets $\{i_1,\ldots,i_m\}\subset
\{1,\ldots,n\}$ with $i_1=1$.

To compute the right-hand side, we fix one subset $I$ and compute the
corresponding term by breaking up the depth $m$ piece $R_m$ into a sum
of pieces $T^m_{p,q}$ as in \eqref{R}.

We first consider a piece $T^m_{p,q}$ from the middle sum, i.e.~such that
$p\ge 1$, $q\ge d$ and $p+q<m$. Denote the $pic$ part corresponding to 
$I=\{i_1,\ldots,i_m\}$ by 
$$\Pi_{I}=\prod_{r=1}^m \prod_{t=i_r+1}^{i_{r+1}-1} \frac{1}{(v_t-v_{i_r})},$$
and let us compute
$$\Pi_I\,T^m_{p,q}\binom{\tilde u_1,\ldots,\tilde u_m}{\tilde v_1,\ldots,\tilde v_m}.$$
Recall from the expression for $T^m_{p,q}$ given in \eqref{Tpq} (with $n$
replaced by $m$ and $u_i,v_i$ by $\tilde u_i,\tilde v_i$) that 
$T^m_{p,q}$ is a sum over tuples $\mathrm{s}=(s_0=0,s_1,\ldots,s_d=q)$.
For each fixed $I$ and fixed $\mathbf{s}$ inside the $I$-term, let
$$\tilde F_{\mathbf{s},I}(x)=\prod_{a=1}^d(v_{i_{p+s_{a-1}+1}}-x)^{l_a}.$$
Then the corresponding $\mathbf{s}$-term of $T^m_{p,q}$ times
the $pic$ part $\Pi_I$ is given by
\begin{align}
&U_n^iv_1^j\prod_{a=1}^d\ (u_{i_{p+s_{a-1}+1}}+\cdots+
u_{i_{p+s_a+1}-1})^{k_a}S^{L-1}_{\tilde F_{\mathbf{s},I}}(v_{i_p},v_n) \label{sterm}
\\ &\hspace{6.5cm}\cdot\frac{\displaystyle{\prod_{r=1}^{m-1}\prod_{t=i_r+1}^{i_{r+1}-1}\frac{1}{(v_t-v_{i_r})}}}
{\displaystyle{\Bigl(\prod_{r=1}^{p-1}(v_{i_r}-v_{i_{r+1}})\Bigr)
\Bigl(\prod_{a=1}^d\prod_{r=p+s_{a-1}+1}^{p+s_a-1}(v_{i_r}-v_{i_{r+1}})\Bigr)}},\nonumber
\end{align}
with $i_m=n$ and $p+q+1=m$; all terms where $p+q+1<m$ vanish.

Next, we take the full sum over the ${\mathbf{s}}$-terms in \eqref{Tpq}
giving $T^m_{p,q}$. Let
$$G_{I,p}(x)=\sum_{p+1=r_1<\cdots<r_d<m}\ 
\Biggl(\prod_{a=1}^d\ (u_{i_{r_a}}+\cdots+u_{i_{r_{a+1}}-1})^{k_a}\Biggr)
\Biggl(\prod_{a=2}^d (v_{i_{r_a}-1}-v_{i_{r_a}})\Biggr)
\Biggl(\prod_{a=1}^d(v_{i_{r_a}}-x)^{l_a}\Biggr).$$
Then, adding up \eqref{sterm} over all the tuples $(s_0,\ldots,s_d)$, we 
end up with
$$\Pi_I\ T^m_{p,q}\binom{\tilde u_1,\ldots,\tilde u_m}{\tilde v_1,\ldots,\tilde v_m}=
\frac{U_n^iv_1^j\Pi_I}{\prod_{r=1}^{p-1}(v_{i_r}-v_{i_{r+1}})
\prod_{r=p+1}^{m-2}(v_{i_r}-v_{i_{r+1}})}
S^{L-1}_{G_{I,p}}(v_{i_p},v_n).$$
When these terms are summed over all $I$ and all $p,q$ with
$p\ge 1$, $q\ge d$ and $p+q<m$ (so the inner terms of \eqref{Tpq}),
many terms cancel, leaving the total inner sum
\begin{equation}\label{innersum}
U_n^iv_1^j\Biggl(\prod_{a=1}^d u_{n-d+a-1}^{k_a}\Biggr)S^{L-n+d+1}_{F_1}(v_1,
\ldots,v_{n-d-1},v_n).
\end{equation}
Note that this is precisely the first term of \eqref{uritBA}, and that
it is equal to zero when $n=d+1$.

It remains to add on the terms corresponding to $p=0$ and to $p+q=m$.
Adding up the explicit expressions for the terms $T^m_{0,q}\Pi_I$
over varying $I$ and $m$, we find that everything cancels and we obtain 0 
for $n>d+1$. When $n$ is the lowest possible depth, $n=d+1$, we find
the term
\begin{equation}\label{pequals0}
U_{d+1}^iv_{d+1}^j\prod_{a=1}^du_a^{k_a}(v_a-v_{d+1})^{l_a}.
\end{equation}
In the other boundary case $p+q=m$, we obtain the same expression for
all $n\ge d+1$:
\begin{equation}\label{pplusqequalsm}
-U_n^iv_1^j\Bigl(\prod_{a=1}^du_{n-d+a}^{k_a}\Bigr)S^{L-n+d+1}_{F_2}(v_1,
\ldots,v_{n-d}).
\end{equation}
This is exactly the second term of \eqref{uritBA}, which proves that
$$ganit_{pic}\big(arit(ganit_{poc}(B))(ganit_{poc}(A))\big)=urit(B)(A)$$
in depths $n>d+1$. For $n=d+1$, since \eqref{innersum} is zero, the
total is given by the remaining term \eqref{pequals0} coming from $p=0$ minus 
the term \eqref{pplusqequalsm} from $p+q=n=d+1$. This difference is precisely
equal to \eqref{uritBAlowestdep}.

\subsection{The uri bracket and the swap operator}

Let $anti$, $neg$, $anit$, $axit$, $ani$ and $gani$ be the operators defined in \cite{SchnepsARI}.
Denote by $asna$ the bimould operator $anti\circ swap\circ neg\circ anti$,
and write $asnapic=asna(pic)$ and $asnapoc=asna(poc)$, where $pic$ and $poc$ are introduced in Example \ref{ex:pic_poc}. Finally, let
$$Pi=gani(pic,asnapoc).$$

\vspace{.3cm}  
\begin{Theorem}\label{Equ} 
Let 
$A \in \BARI$ be swap invariant, $
B\in \BARI_{\underline{\mathrm{il}},\mathrm{swap}}$, and set 
$$ P=\ganit_{poc}(A), \qquad Q=\ganit_{poc}(B).$$
Then the following statements are equivalent: 
\begin{enumerate}[(i)]
    \item The bimould $\ganitpic\bigl(\preari(P,Q)\bigr)$ is swap invariant. 
    \item We have 
$$\ganit_{Pi}\bigl(\arit(Q)(P)\bigr) = \axit\bigl(\swap(Q),-\push(\swap(Q))\bigr)(\swap(P)).$$
\end{enumerate}
\end{Theorem} 

The key of the proof is the following technical result.

\begin{Proposition}\label{ganitprop} For all bimoulds $R\in \BARI$, we have
\begin{equation}\label{eq2}
\swap\circ\ganit_R = \ganit_{\asna(R)}\circ\swap.
\end{equation}
\end{Proposition}

We will prove this proposition at the end of this section. First let us 
use the statement to prove Theorem \ref{Equ}. 

\begin{proof}  As $ganit_{pic}$ and $ganit_{poc}$ are inverse to each other by Proposition \ref{prop:ganit_pic_poc}, we get
$$swap\circ\ganitpic\circ\ganitpoc\circ swap={\rm id}.$$
But by using \eqref{eq2} twice, this is equal to
$$\ganit_{asnapic}\circ swap\circ\ganitpoc\circ swap=
\ganit_{asnapic}\circ \ganit_{asnapoc}\circ swap\circ swap,$$
and since $swap\circ swap={\rm id}$, we find that
\begin{equation}\label{eq3}
\ganit_{asnapic}\circ \ganit_{asnapoc}={\rm id}.
\end{equation}

Now, recall from Definition \ref{def:preari} that 
$$\preari(P,Q)=mu(P,Q)+\arit(Q)(P).$$
Since $\ganitpic$ is an automorphism for the $mu$-multiplication by Lemma \ref{lem:ganit_mu_auto} and we have
$\ganitpic\circ \ganitpoc={\rm id}$, we get 
\begin{equation}\label{eq4}
\ganitpic\bigl(\preari(P,Q)\bigr) = mu(A,B) + \ganitpic\bigl(\arit(Q)(P)\bigr). 
\end{equation}
We will need the identity
\begin{equation*}\label{eq5}
\swap\bigl(\preari(P,Q)\bigr) = mu(\swap (P),\swap (Q)) + \axit\bigl(\swap (Q),-\push(\swap (Q))\bigr)(\swap (P))
\end{equation*}
from \cite[(2.4.10)]{SchnepsARI}.
Applying \(ganit_{asnapic}\) to both sides and using \eqref{eq2} gives 
\begin{align}\label{eq6} 
\swap\Big( \ganitpic\bigl(\preari(P,Q)\bigr)\Big) &= 
\ganit_{asnapic}\Big(swap(preari(P,Q))\Big)\\
&=\ganit_{asnapic} \bigl(mu(\swap (P),\swap (Q))\bigr) \notag\\ 
&+ \ganit_{asnapic} \left( \axit\bigl(\swap (Q),-\push(\swap (Q))\bigr)(\swap (P)) \right). \notag
\end{align} 
Since $A$ is swap-invariant, \eqref{eq2} gives
$$swap(P)=swap(\ganitpoc(A))=\ganit_{asnapoc} (swap(A))=\ganit_{asnapoc}(A),$$
and similarly $swap(Q)=\ganit_{asnapoc}(B)$.
Since $\ganit$ is a $mu$-automorphism by Lemma \ref{lem:ganit_mu_auto}, by \eqref{eq3} we deduce
\begin{align*}
\ganit_{asnapic} \big(mu(\swap (P),\swap (Q))\big) &= 
\ganit_{asnapic}\big(mu\bigl(\ganit_{asnapoc}(A),\ganit_{asnapoc}(B)\bigr)\big) \\
&= mu(A,B).
\end{align*}
Consequently, comparison of \eqref{eq4} and \eqref{eq6} shows that 
$\ganitpic\bigl(\preari(P,Q)\bigr)$ is swap invariant if and 
only if 
\begin{equation}\label{eq7}
\ganitpic\bigl(\arit(Q)(P)\bigr) = \ganit_{asnapic} \left( \axit\bigl(\swap (Q),-\push(\swap (Q))\bigr)(\swap (P)) \right).
\end{equation}
By the definition of $Pi$, we have
$$\ganit_{Pi}=\ganit_{asnapoc} \circ\ganitpic,$$
so applying \(\ganit_{asnapoc}\) to both sides of \eqref{eq7} gives with \eqref{eq3}
$$\ganit_{Pi}\bigl(\arit(Q)(P)\bigr) = \axit\bigl(\swap Q,-\push(\swap Q)\bigr)(\swap P).$$
This proves the equivalence stated in the theorem. 
\end{proof}

It remains only to prove
Proposition \ref{ganitprop}.

\begin{proof}[Proof of Proposition \ref{ganitprop}.] Recall from Definition \ref{def:ganit} that 
\begin{equation*}
\ganit_R(S)(w) = \sum_{\substack{ w=b_1c_1\cdots b_sc_s\\ b_i\neq\emptyset,\;c_s=\emptyset }} S(b_1\rceil\cdots b_s\rceil) R(\lfloor c_1)\cdots R(\lfloor c_s).
\end{equation*}
We use the following ``block-transformation'' identities.
For every decomposition 
$$\swap(w)=b_1c_1\cdots b_sc_s, \qquad c_s=\emptyset,$$
there is a unique corresponding decomposition 
$$w=b'_1c'_1\cdots b'_sc'_s, \qquad c'_s=\emptyset$$ 
obtained by reversing the blocks and transferring the boundary points, 
such that 
\begin{equation}\label{eq8}
b_1\rceil\cdots b_s\rceil = \swap\bigl(b'_1\rceil\cdots b'_s\rceil\bigr)
\end{equation}
and, after the corresponding reversal of the nonempty $c$-blocks, 
\begin{equation}\label{eq9}
R(\lfloor c_i) = (\asna (R))(\lfloor c'_{s-i}).
\end{equation}

The identity \eqref{eq8} expresses the compatibility of $swap$
with the upper flexions of the $b$-blocks, while \eqref{eq9} expresses the 
transformation of a lower-flexed complementary block under the swap 
substitution. The reversal of the block, the change from differences 
measured from the left endpoint to differences measured from the right 
endpoint, and the restoration of the original order are precisely 
encoded by 
$$\asna = \anti\circ\swap\circ\neg\circ\anti.$$
Using the block-transformation identities, we compute 
\begin{align*} 
\swap\bigl(\ganit_R(S)\bigr)(w) &= \ganit_R(S)(\swap(w)) \\ 
&= \sum_{\substack{ \swap(w)=b_1c_1\cdots b_sc_s\\ 
b_i\neq\emptyset,\;c_s=\emptyset }} S(b_1\rceil\cdots b_s\rceil) \prod_{i=1}^{s} R(\lfloor c_i) \\ 
&= \sum_{\substack{ w=b'_1c'_1\cdots b'_sc'_s\\ 
b'_i\neq\emptyset,\;c'_s=\emptyset }} (\swap S)(b'_1\rceil\cdots b'_s\rceil) \prod_{i=1}^{s}(\asna R)(\lfloor c'_i) \\ 
&= \ganit_{\asna (R)}(\swap (S))(w). \qedhere
\end{align*} 
\end{proof}

\phantomsection

\addcontentsline{toc}{section}{References}
\bibliographystyle{amsalpha}
\bibliography{rdc_main}

\end{document}